\documentclass[11 pt]{article}

	\usepackage[english]{babel}
	\usepackage[utf8]{inputenc}  
	\usepackage[T1]{fontenc}
	\usepackage{amssymb}
	\usepackage[mathscr]{euscript}
	\usepackage{stmaryrd}
	\usepackage{amsmath}
	\usepackage{tikz}
    \usepackage{caption}
	\usepackage[all,cmtip]{xy}
	\usepackage{amsthm}
	\usepackage{varioref}
	\usepackage{geometry}
	\usepackage{float}
	\usepackage{lmodern}
	\usepackage[bookmarks=true]{hyperref}
 	\usepackage[numbered]{bookmark}
	\usepackage{lipsum}
	 \usepackage{fancyhdr}
	\usepackage{xargs}
\usepackage{thmtools}
\usepackage{thm-restate}
	\usepackage{ifthen}
	\theoremstyle{plain}
	\newcounter{n}
	\numberwithin{n}{section}
	\newtheorem{df}[n]{Definition}
	\newtheorem{theo}[n]{Theorem}
	\labelformat{theo}{#1}
     \newtheorem{rmq}[n]{Remark}
	\labelformat{rmq}{#1}
	\newtheorem{cor}[n]{Corollary}
	\newtheorem{lm}[n]{Lemma}
	\labelformat{lm}{#1}
	
	\newtheorem{nt}[n]{Notation}
	
	\newtheorem{prop}[n]{Proposition}
	\labelformat{prop}{#1}
	\renewcommand\epsilon{\varepsilon}
	\renewcommand\phi{\varphi}
	\newcommand\R{\mathbb{R}}
	\newcommand\s{\mathbb{S}}

\tikzset{
    partial ellipse/.style args={#1:#2:#3}{
        insert path={+ (#1:#3) arc (#1:#2:#3)}
    }
}

	\newcommand\Card{{\mathrm{Card}}}
	\newcommand\Int{{\mathrm{Int}}}

	\newcommand\sphereambiante{{M}}
	\newcommand\punct[1]{{#1^{\circ}}}
	\newcommand\espaceambient{\punct{M}}
	\newcommand\ambientspace\espaceambient
	\newcommandx\bouleambiante[1][1= M]{B(#1)}
	\newcommandx\voisinageinfini[1][1= ]{\ifthenelse{\equal{#1}{}}{\punct{B_\infty}}{\punct{B_\infty}(#1)}}
	
	\newcommand\voisinageinfinideux[1][\frac14]{{\punct{B_{\infty, #1}}}}

	\newcommandx\sommets[2][1=\Gamma,2= \empty]{\ifthenelse{\equal{#2}{}} {{V(#1)}}{{V^*(#1)}} }
	\newcommand\sommetsinternes[1][\Gamma]{{V_i(#1)}}
	\newcommand\sommetsexternes[1][\Gamma]{{V_e(#1)}}
	\newcommand\aretes[1][\Gamma]{E(#1)}
	\newcommand\aretesinternes[1][\Gamma]{E_i(#1)}
	\newcommand\aretesexternes[1][\Gamma]{E_e(#1)}

	\newcommand\graphes{\mathcal G_k}
	\newcommand\graphesnum{\widetilde{\mathcal G_k}}
	\definecolor{goldenrod}{rgb}{.85, .75, .03}

	\newcommandx\confignoeud[3][1=\Gamma, 2=\psi,3=\empty ]{C_{#1}^{#3}(#2)}
	\newcommand\configR{C_2(\R^n)}
	\newcommand\config{C_2}
	\newcommand\configM{\config(\ambientspace)}
	\newcommand\unitaire{{U}}

	\newcommand\injections{\mathcal I}

	\renewcommand\d{\mathrm d}

	\newcommand\orientationarete[1]{{\Omega_{#1}}}
	\newcommand\orientationconfig[1][\Gamma]{{\Omega(#1)}}

	\newcommand\propagateurinterne{\alpha}
	\newcommand\propint{\propagateurinterne}
	\newcommand\propagateurexterne{\beta}
	\newcommand\propext{\propagateurexterne}
	\newcommand\familleformes{{(\propagateurinterne_i,\propagateurexterne_i)_{1\leq i \leq 2k}}}
	\newcommand\Propint{{A}}
	\newcommand\Propext{{B}}

	\newcommandx\formearete[2][1=e,2=\sigma]{{\omega^F_{#1,#2}}}

	\newcommand\coeffs{\mathbb Z}
	\newcommand\I[1][F]{{I^{#1}(\Gamma,\sigma,\psi)}}
	\newcommand\Z[1][F]{{Z_k^{#1}}}
	\newcommand\forme{{\omega^{F}(\Gamma,\sigma, \psi)}}

	\newcommand\ecartinterne{{\zeta_{1}^{n-2}}}
	\newcommand\ecartexterne{{\xi_{1}^{n}}}
	\newcommand\volint{{\eta^{n-2}_{1}}}
	\newcommand\volext{{\theta^{n}_{1}}}
	\newcommandx\formeareteT[2][1=e,2=\sigma]{{\tilde\omega_{#1,#2}}}
	\newcommandx\formeT[2][1= \Gamma, 2= \sigma]{{\tilde\omega(#1,#2,\psi)}}

	\newcommand\face[1][S]{{\partial_{#1}\confignoeud}}
	\newcommand\faces[1][\Gamma]{{\mathcal F(#1)}}
	\newcommand\hfaces[1][\Gamma]{{\mathcal H_1(#1)}}
	\newcommand\Hfaces[1][\Gamma]{{\mathcal H_2(#1)}}
	
	\newcommand\HfaceA[1][\Gamma]{{\mathcal H_2^a(#1)}}
	\newcommand\HfaceB[1][\Gamma]{{\mathcal H_2^b(#1)}}
	
	\newcommandx\pes[2][1=e, 2=S]{{G_{#1, #2}}}
	\newcommand\faceinfiniU[1][S']{{C_{#1,\infty}}}
	\newcommand\faceinfiniC{{C^0_{\Gamma_{|\sommets\setminus S}}}}
	\newcommand\Esi{{E_i^{S'}(\Gamma)}}
	\newcommand\Ese{{E_e^{S'}(\Gamma)}}
	\newcommand\Es{{E^{S'}(\Gamma)}}

	\newcommand\facefinieC{{C_{\delta_S\Gamma}^0 }}
	\newcommand\facefinieU[1][S]{{\widehat{C_{#1}}}}
	\newcommand{\inj}{{\injections(\R^n,\R^{n+2})}}

	 \newcommandx\facepC[2][1= e, 2 =\Gamma]{{C_{\delta_{#1}{#2}}^0}}

	\newcommand\D{{D(\Gamma,\sigma)}}

	\newcommand\Par{{\mathrm{Par}(\ambientspace)}}
	\newcommand\indices{\{1,\ldots, 2k\}}
	\newcommand\B{{\mathbb B}}

	\newcommand{\noop}[1]{} 
	\newcommand\ray[1][j]{{r_{#1}}}
	\newcommand\confspaceAA{{C_{\Gamma_{S_1,S_2}}(\psi_{triv})}}
	\newcommand\confspaceA{{C_{\Gamma_{S_1,S_2}}}}

        \title{Generalized Bott-Cattaneo-Rossi invariants of high-dimensional long knots }
	\author{David Leturcq\footnote{Institut Fourier, Université-Grenoble-Alpes}}
	\date{ }
\begin{document}
\maketitle
\begin{abstract} Bott, Cattaneo and Rossi defined invariants of long knots $\R^n \hookrightarrow \R^{n+2}$ as combinations of configuration space integrals for $n$ odd $\geq 3$. Here, we give a more flexible definition of these invariants. Our definition allows us to interpret these invariants as counts of diagrams. It extends to long knots inside more general $(n+2)$-manifolds, called asymptotic homology $\R^{n+2}$, and provides invariants of these knots. 

 \end{abstract}


\section{Introduction}

In \cite{[Bott]}, Bott introduced an isotopy invariant $Z_2$ of knots $\s^n\hookrightarrow \R^{n+2}$ in odd dimensional Euclidean spaces. The invariant $Z_2$ is defined as a linear combination of configuration space integrals associated to graphs by integrating forms associated to the edges, which represent directions in $\R^n$ or in $\R^{n+2}$. The involved graphs have four vertices of two kinds and four edges of two kinds.

This invariant was generalized to a whole family $(Z_{k})_{k\in \mathbb N\setminus\{0\}}$ of isotopy invariants of long knots $\R^n\hookrightarrow \R^{n+2}$, for odd $n\geq3$, by Cattaneo and Rossi in \cite{Cattaneo2005} and by Rossi in his thesis \cite{[Rossi]}. The \emph{degree $k$ Bott-Cattaneo-Rossi} (BCR for short) \emph{invariant} $Z_{k}$ involves diagrams with $2k$ vertices.

In \cite{[Watanabe]}, Watanabe proved that, when restricted to ribbon long knots, the BCR invariants are finite type invariants with respect to some operations on ribbon knots, and he used this property to prove that the invariants $Z_{k}$ are not trivial for even $k\geq1$, and that they are related to the Alexander polynomial, for long ribbon knots.

In Theorem \ref{th1}, which is the main theorem of this article, we generalize the invariants $(Z_k)_{k\geq2}$ to long knots in the parallelized asymptotic homology $\R^{n+2}$ of Section \ref{S11} when $n\geq 3$ is odd, using the notion of \emph{propagating forms}. When the ambient space is $\R^{n+2}$, our extended definition also provides a more flexible definition for the original invariants $(Z_k)_{k\geq 2}$. In Theorem \ref{the}, we equivalently define our generalized BCR invariants as rational combinations of intersection numbers of chains in configuration spaces. In particular, our generalized invariants are rational. Theorem \ref{additivity} asserts that $Z_{k}$ is additive under connected sum. In \cite{article2}, we use our flexible definition to express our generalized $Z_2$ in terms of linking numbers or of Alexander polynomials for all long knots in parallelizable asymptotic homology $\R^{n+2}$, when $n\equiv1\mod 4$. 
In \cite{article3}, we extend the BCR invariants $Z_k$ to $1$-dimensional long knots in (rational) asymptotic homology $\R^3$, and the results of \cite{article2} express these extended invariants and the usual Alexander polynomial in terms of each other. 

Our invariants $Z_{k}$ are precisely defined in Section \ref{S1}, where the three forementioned theorems are stated. Their proofs are given in the following sections.

Our definition of $Z_{k}$ involves a parallelization of the ambient space, which is a trivialization of its tangent bundle that is standard outside a compact as precisely explained in Definition \ref{paral-def}. In Section \ref{S5}, we prove that $Z_{k}$ does not depend on the parallelization when it exists. In order to prove this result, we prove Theorem \ref{parallelization theorem}, which asserts that, up to homotopy, any two parallelizations of a parallelizable asymptotic homology $\R^{n+2}$ that are standard outside a compact coincide outside an (arbitrarily small) ball.

I do not know whether any asymptotic homology $\R^{n+2}$ admits a parallelization in the sense of Definition \ref{paral-def}. However, using the fact that the connected sum of any odd-dimensional asymptotic homology $\R^{n+2}$ with itself is parallelizable in the sense of Definition \ref{paral-def} (Proposition \ref{prop210}) and that $Z_k$ is additive (Theorem \ref{additivity}), we extend our invariants to long knots in any (possibly non-parallelizable) asymptotic homology $\R^{n+2}$ with $n$ odd $\geq3$ in Definition \ref{Zkext}.

I thank my advisor Christine Lescop for her help with the redaction of this article. I also thank the referee for her/his helpful comments. 
\section{Definition of the BCR invariants}\label{S1}
\subsection{Parallelized asymptotic homology $\R^{n+2}$ and long knots}\label{S11}

In this article, we fix an odd integer $n\geq3$, and $\sphereambiante$ denotes an $(n+2)$-dimensional closed smooth oriented manifold, such that $H_*(\sphereambiante ; \coeffs)= H_*(\s^{n+2}; \coeffs)$. Such a manifold is called a \emph{homology $(n+2)$-sphere}. 

In such a homology sphere, choose a point $\infty$ and a closed ball $B_\infty(\sphereambiante)$ around this point. Fix an identification of this ball $B_\infty(M)$ with the complement $B_\infty$ of the open unit ball of $\R^{n+2}$ in $\s^{n+2}=\R^{n+2}\cup \{\infty\}$. Let $\espaceambient$ denote the manifold $\sphereambiante \setminus\{\infty\}$ and let $\voisinageinfini[\sphereambiante]$ denote the punctured ball $B_\infty(\sphereambiante)\setminus\{\infty\}$. In all the following, this punctured ball $\voisinageinfini[\sphereambiante]$ is identified with the complement $\voisinageinfini$ of the open unit ball in $\R^{n+2}$. Let $\bouleambiante$ denote the closure of $\ambientspace\setminus\voisinageinfini$. Then, the manifold $\ambientspace$ can be seen as $\ambientspace = \bouleambiante \cup \voisinageinfini$, where $\voisinageinfini\subset \R^{n+2}$ (see Figure \ref{Fig0}). Note that such a manifold $\espaceambient$ has the same homology as $\R^{n+2}$. The manifold $\ambientspace$ equipped with the decomposition $\ambientspace = \bouleambiante \cup \voisinageinfini$ is called an \emph{asymptotic homology $\R^{n+2}$}.

\begin{figure}[!h]
\centering
\begin{tikzpicture}
\draw[white] (-1, -1) rectangle (1, 1);
\draw (0,0) circle (0.5);
\draw (-60:0.5) ++(.29,-.2) node{$M$};
\fill (0,0.5) circle (0.05);
\draw (0, 0.65) node {$\infty$};
\end{tikzpicture}
\begin{tikzpicture}
\draw[white] (-1, -1) rectangle (1, 1);
\draw (0,0) [partial ellipse=130:410:0.5cm and 0.5cm];
\draw[dotted] (0,0) [partial ellipse=50:130:0.5cm and 0.5cm];
\draw [thick] (50:.45)--(50:.55) (130:.45)--(130:.55);
\draw (-60:.5) ++(.45,-.2) node{$B(M)$};
\draw (90:.5) ++(.5,.3) node{$B_\infty(M)$};
\end{tikzpicture}
\begin{tikzpicture}
\draw[white] (-1, -1) rectangle (1, 1);
\draw (0,0) [partial ellipse=95:445:0.5cm and 0.5cm];
\draw (0,.5) circle (0.05);
\draw (-60:0.5) ++(.29,-.2) node{$\ambientspace$};
\end{tikzpicture}
\begin{tikzpicture}
\draw[white] (-1, -1) rectangle (1, 1);
\draw (0,.5) circle (0.05);
\draw (0,0) [partial ellipse=130:410:0.5cm and 0.5cm];
\draw[dotted] (0,0) [partial ellipse=50:130:0.5cm and 0.5cm];
\draw [thick] (50:.45)--(50:.55) (130:.45)--(130:.55);
\draw (90:.5) ++(.5,.3) node{$\punct{B_\infty}(M)= \punct{B_\infty}$};
\draw (-60:.5) ++(.45,-.2) node{$B(M)$};
\end{tikzpicture}
\caption{}
\label{Fig0}
\end{figure}

\emph{Long knots} of such a space $\ambientspace$ are smooth embeddings $\psi\colon \R^n\hookrightarrow \espaceambient$ such that $\psi(x)=(0,0,x)\in\voisinageinfini$ when $||x||\geq 1$, and $\psi(x) \in \bouleambiante$ when $||x||\leq 1$.

Two long knots $\psi$ and $\psi'$ are \emph{isotopic} if there exists a family $(\psi_t)_{0\leq t \leq 1}$ of long knots, such that the map $(t,x)\in[0,1]\times\R^n\mapsto \psi_t(x)\in\ambientspace$ is smooth, such that $\psi_0=\psi$ and $\psi_1=\psi'$. Such a family is called an \emph{isotopy} (between $\psi$ and $\psi'$).
\begin{df}\label{paral-def}
A \emph{parallelization} of an asymptotic homology $\R^{n+2}$ is a bundle isomorphism $\tau\colon \ambientspace\times\R^{n+2}\rightarrow T\ambientspace$ that coincides with the canonical trivialization of $T\R^{n+2}$ on $\voisinageinfini\times\R^{n+2}$. An asymptotic homology $\R^{n+2}$ equipped with such a parallelization is called a \emph{parallelized asymptotic homology $\R^{n+2}$}. 

Two parallelizations $\tau$ and $\tau'$ are homotopic if there exists a smooth family $(\tau_t)_{0\leq t \leq 1}$ of parallelizations such that $\tau_0=\tau$ and $\tau_1=\tau'$.
Given a parallelization $\tau$ and $x\in \ambientspace$, $\tau_x$ denotes the isomorphism $\tau(x,\cdot)\colon \R^{n+2} \rightarrow T_x\ambientspace$.
\end{df}

\subsection{BCR diagrams}
The definition of the BCR invariants involves the following graphs, called \emph{BCR diagrams}.\begin{df}\label{BCRdef}A \emph{BCR diagram} is an oriented connected graph $\Gamma$, defined by a set $\sommets$ of vertices, decomposed into $\sommets = \sommetsinternes\sqcup\sommetsexternes$, and a set $\aretes$ of ordered pairs of distinct vertices, decomposed into $\aretes = \aretesinternes\sqcup\aretesexternes$, whose elements are called \emph{edges}\footnote{Note that this implies that our graphs have neither loops nor multiple edges with same orientation.}, where the elements of $\sommetsinternes$ are called \emph{internal vertices}, those of $\sommetsexternes$, \emph{external vertices}, those of $\aretesinternes$, \emph{internal edges}, and those of $\aretesexternes$, \emph{external edges}, and such that, for any vertex $v$,  one of the five following properties holds: \begin{enumerate}
\item $v$ is external and trivalent, with two incoming external edges and one outgoing external edge, and one of the incoming edges comes from a univalent vertex. 
\item $v$ is internal and trivalent, with one incoming internal edge, one outgoing internal edge, and one incoming external edge, which comes from a univalent vertex.
\item $v$ is internal and univalent, with one outgoing external edge. 
\item $v$ is internal and bivalent, with one incoming external edge and one outgoing internal edge.
\item $v$ is internal and bivalent, with one incoming internal edge and one outgoing external edge.
\end{enumerate}
The external edges that come from a (necessarily internal) univalent vertex are called the \emph{legs} of $\Gamma$. The subgraph of $\Gamma$ made of all the other edges, and the non univalent vertices is called the \emph{cycle} of $\Gamma$.

Define the \emph{degree} of a BCR diagram $\Gamma$ as $\mathrm{deg}(\Gamma) = \frac12\mathrm{Card}(\sommets)$, and let $\graphes$ denote the set of all BCR diagrams of degree $k$. 
\end{df}

In the following, internal edges are depicted by solid arrows, external edges by dashed arrows, internal vertices by black dots, and external vertices by white dots (circles). This is the same convention as in \cite{[Watanabe]}, but it is the opposite of what was done in \cite{Cattaneo2005}, where the internal edges are dashed, and the external ones are solid. With these conventions, the five behaviors of Definition \ref{BCRdef} are depicted in Figure \ref{fig-BCR}.

\begin{figure}[H]
\centering
\begin{tikzpicture}
\draw (90:0.7) node {$1$};
\draw (0.265,0) node {$v$};
\draw (0, 0) circle (0.1) ;
\draw [->, >= latex, dashed](-90: 1)-- (-90: 0.1) ;
\draw [<-, >= latex, dashed](40: 1)-- (40: 0.1) ;
\draw [->, >= latex, dashed](140: 1)-- (140: 0.1) ;
\fill (140:1.1) circle (0.1);
\end{tikzpicture} \ \ \ \ \ \ \ \ 
\begin{tikzpicture}
\draw (90:0.7) node {$2$};
\fill (0, 0) circle (0.1) ;
\fill (140:1.1) circle (0.1);
\draw (0.265,0) node {$v$};
\draw [->, >= latex, thick](-90: 1)-- (-90: 0.1) ;
\draw [<-, >= latex, thick](40: 1)-- (40: 0.1) ;
\draw [->, >= latex, dashed](140: 1)-- (140: 0.1) ;
\end{tikzpicture} \ \ \ \ \ \ \ \ 
\begin{tikzpicture}
\draw (90:0.7) node {$3$};
\fill (0, 0) circle (0.1) ;
\draw (0.265,0) node {$v$};
\draw [<-, >= latex, dashed](-90: 1)-- (-90: 0.1) ;
\end{tikzpicture} \ \ \ \ \ \ \ \ 
\begin{tikzpicture}
\draw (90:0.7) node {$4$};
\fill (0, 0) circle (0.1) ;
\draw (0.265,0) node {$v$};
\draw [->, >= latex, dashed](-90: 1)-- (-90: 0.1) ;
\draw [<-, >= latex, thick](40: 1)-- (30: 0.1) ;
\end{tikzpicture}  \ \ \ \ \ \ \ 
\begin{tikzpicture}
\draw (90:0.7) node {$5$};
\fill (0, 0) circle (0.1) ;
\draw (0.265,0) node {$v$};
\draw [->, >= latex, thick](-90: 1)-- (-90: 0.1) ;
\draw [<-, >= latex, dashed](40: 1)-- (30: 0.1) ;
\end{tikzpicture}
\caption{ }\label{fig-BCR}
\end{figure}

Definition \ref{BCRdef} implies that any BCR diagram consists of one cycle with some legs attached to it, which is a cyclic sequence of pieces as in Figure \ref{fig-BCR2} with as many pieces of the first type than of the second type. In particular, the degree of a BCR diagram is an integer.

\begin{figure}[H]
\centering
\begin{tikzpicture}
\draw [->, dashed, >= latex] (-0.5, 0)-- (-0.1,0);
\fill (0,0) circle (0.1) ;
\draw  (.1, 0)-- (0.5,0);
\draw[white](0,1) circle(0.1);
\end{tikzpicture}\ \ \ \ \ \ \ \ \ \ 
\begin{tikzpicture}
\draw [->, >= latex] (-0.5, 0)-- (-0.1,0);
\fill (0,0) circle (0.1);
\draw [dashed] (.1, 0)-- (0.5,0);
\draw[white](0,1) circle(0.1);
\end{tikzpicture}\ \ \ \ \ \ \ \ \ \ 
\begin{tikzpicture}
\fill (1,0) circle (0.1) (3,0) circle (0.1);
\fill (1,1) circle (0.1) (3,1) circle (0.1);
\draw [->, dashed,>= latex] (1, 0.9)-- (1,0.1);
\draw [->,dashed,  >= latex] (3, 0.9)-- (3,0.1);
\draw [->, >= latex] (.5, 0)-- (0.9,0);
\draw  (1.1, 0)-- (1.7,0);
\draw [->, >= latex] (2.3, 0)-- (2.9,0);
\draw [dotted] (1.7, 0)-- (2.3,0);
\draw (3.1, 0)-- (3.5,0);
\end{tikzpicture}\ \ \ \ \ \ \ \ \ \ 
\begin{tikzpicture}
\draw (1,0) circle (0.1) (3,0) circle (0.1) ;
\fill (1,1) circle (0.1) (3,1) circle (0.1) ;
\draw [->,dashed, >= latex] (1, 0.9)-- (1,0.1);
\draw [->,dashed, >= latex] (3, 0.9)-- (3,0.1);
\draw [->,dashed, >= latex] (.5, 0)-- (0.9,0);
\draw [dashed] (1.1, 0)-- (1.7,0);
\draw [->, dashed,>= latex] (2.3, 0)-- (2.9,0);
\draw [dotted] (1.7, 0)-- (2.3,0);
\draw [dashed] (3.1, 0)-- (3.5,0);
\end{tikzpicture}
\caption{ }\label{fig-BCR2}
\end{figure}

For example, Figure \ref{fig-BCR3} depicts the five degree $2$ BCR diagrams, which respectively have two, two, one, one, and no leg.
\begin{figure}[H]
\centering
\begin{tikzpicture}
\fill (0,0) circle (0.1) ;
\draw [dashed, ->, >= latex] (0.1,0) -- (0.9,0) ;
\fill (1,0) circle (0.1);
\draw [<-, >= latex] (1.1,0) to[bend right] (1.9,0) ;
\draw [->, >= latex] (1.1,0) to[bend left] (1.9,0) ;
\fill (2,0) circle (0.1);
\draw [dashed, <-, >= latex] (2.1,0) -- (2.9,0) ;
\fill (3,0) circle (0.1);
\end{tikzpicture} \ \ \ 
\begin{tikzpicture}
\fill (0,0) circle (0.1) ;
\draw [dashed, ->, >= latex] (0.1,0) -- (0.9,0) ;
\draw (1,0) circle (0.1);
\draw [dashed, <-, >= latex] (1.1,0) to[bend right] (1.9,0) ;
\draw [dashed, ->, >= latex] (1.1,0) to[bend left] (1.9,0) ;
\draw (2,0) circle (0.1);
\draw [dashed, <-, >= latex] (2.1,0) -- (2.9,0) ;
\fill (3,0) circle (0.1);
\end{tikzpicture}\ \ \ 
\begin{tikzpicture}
\fill (0,0) circle (0.1) ;
\draw [<-, >=latex, dashed](0:0.1) -- (0: 0.9);
\draw [<-, >=latex] (120:0.1) -- (120:0.9);
\draw [->, >=latex] (240:0.1) -- (240:0.9) ;
\draw [->, dashed, >=latex ]  (240:1) ++(90:0.1) -- ++ (90:1.55);
\fill (240:1) circle (0.1) (120:1) circle (0.1) (0:1) circle (0.1);
\end{tikzpicture} \ \ \ 
\begin{tikzpicture}
\draw (0,0) circle (0.1) ;
\draw [<-, >=latex,dashed](0:0.1) -- (0: 0.9);
\draw [<-, >=latex,dashed] (120:0.1) -- (120:0.9);
\draw [->, dashed, >=latex] (240:0.1) -- (240:0.9) ;
\draw [->, >=latex ]  (240:1) ++(90:0.1) -- ++ (90:1.55);
\fill (240:1) circle (0.1) (120:1) circle (0.1) (0:1) circle (0.1);
\end{tikzpicture}\ \ \ 
\begin{tikzpicture}
\fill (0,0) circle (0.1) (1,0) circle (0.1) (1,1) circle (0.1) (0,1) circle (0.1);
\draw [<-, dashed, >=latex](0.1, 0) -- (0.9, 0);
\draw [<-, dashed, >=latex] (0.9,1)--(0.1,1);
\draw [<-, >=latex] (1,0.1) -- (1, 0.9);
\draw [<-, >=latex] (0,0.9) -- (0, 0.1);
\end{tikzpicture}
\caption{The degree $2$ Jacobi diagrams}\label{fig-BCR3}
\end{figure}
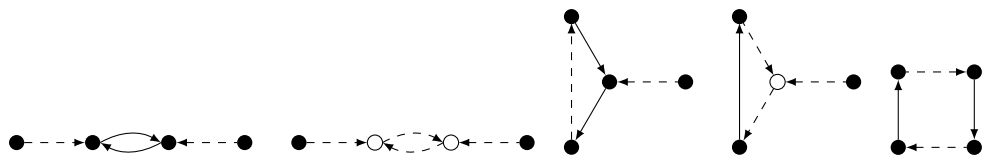

Since any vertex has exactly one outgoing edge, every BCR diagram of degree $k$ has exactly $2k$ edges. 
A \emph{numbering} of a degree $k$ BCR diagram $\Gamma$ is a bijection $\sigma \colon \aretes\rightarrow \indices$, and $\graphesnum$ denotes the set of all degree $k$ numbered BCR diagrams $(\Gamma,\sigma)$ (up to numbered graph isomorphisms).

\subsection{Two-point configuration spaces}

If $P$ is a submanifold of a manifold $Q$ such that $P$ is transverse to the boundary $\partial Q$ of $Q$ and $\partial P= P\cap \partial Q$, its \emph{normal bundle} $\mathfrak NP$ is the bundle over $P$ whose fibers are $\mathfrak N_xP = T_xQ/T_xP$. A fiber $U\mathfrak N_xP$ of the \emph{unit normal bundle $U\mathfrak NP$ of $P$} is the quotient of $\mathfrak N_xP \setminus\{0\}$ by dilations\footnote{Dilations are homotheties with positive ratio.}. 
The \emph{differential blow-up} of $Q$ along $P$ is the manifold obtained by replacing $P$ with its unit normal bundle $U\mathfrak NP$. It is diffeomorphic to the complement in $Q$ of an open tubular neighborhood of $P$. The boundary of the obtained manifold is canonically identified with $(\partial Q\setminus \partial P)\cup U\mathfrak NP$, and its interior is $Q\setminus (P\cup \partial Q)$.

Let $X$ be a $d$-dimensional closed smooth oriented manifold, let $\infty$ be a point of $X$, and set $ \punct{X}= X\setminus\{\infty\}$. Here, we give a short overview of the compactification $C_2(\punct{X})$ of the two-point configuration space defined in \cite[Section 2.2]{[Lescop]}.
Let $\config(\punct{X})$ be the space defined from $X^2$ by blowing up the point $(\infty,\infty)$, and next the closures of the sets $\infty\times \punct{X}$, $\punct{X} \times \infty$ and $\Delta_{\punct{X}}=\{(x,x)\mid x\in \punct{X}\}$.

The manifold $C_2(\punct{X})$ is compact and comes with a canonical map $p_b\colon\config(\punct{X})\rightarrow X^2$. This map induces a diffeomorphism from the interior of $C_2(\punct{X})$ to the open configuration space $C_2^0(\punct{X}) = \{(x,y) \in (\punct{X})^2 \mid x\neq y\}$, and $C_2(\punct{X})$ has the same homotopy type as $C_2^0(\punct{X})$. The manifold $C_2(\punct{X})$ is called the \emph{two-point configuration space of $\punct{X}$}.

Let $T_\infty X$ denote the tangent bundle to $X$ at $\infty$. Identify a punctured neighborhood of $\infty$ in $X$ with $\voisinageinfini$. Identify $T_\infty X\setminus\{0\}$ with $\R^{d}\setminus\{0\}$ so that $u\in \R^d \setminus\{0\}$ is the tangent vector at $0$ of the path $\gamma$ such that $\gamma(0) = \infty$ and for any $t\in \left]0,\frac1{||u||}\right] $, $\gamma(t)=\frac{tu}{||tu||^2}\in \voisinageinfini\subset \punct{X}$. Use this identification to see the unit tangent space $U_\infty X$ to $X$ at $\infty$ as $\s^{d-1}$, so that we have the following description of $\partial C_2(\punct{X})$.

\begin{nt}\label{desc-bord}
The boundary of $\config(\punct{X})$ is the union of:\begin{itemize}
\item the closed face $\partial_{\infty, \infty}C_2(\punct{X})= p_b^{-1}(\{(\infty,\infty)\})$, whose interior\footnote{The boundary of this closed face contains the three codimension $2$ faces of $\config(\punct{X})$, which we do not describe here.} is the set of all classes of pairs $(u,v)\in (\R^d\setminus \{0\})^2\cong (T_\infty X\setminus\{0\})^2 $ such that $u\neq v$, up to dilations.
\item the unit normal bundles to $\punct{X} \times\{\infty\}$ and $\{\infty\}\times \punct{X}$, which are  $\partial_{\punct{X}, \infty} C_2(\punct{X})=\punct{X} \times U_\infty X\cong \punct{X} \times \s^{d-1}$
 and $\partial_{\infty,\punct{X}}C_2(\punct{X}) =U_\infty X \times \punct{X}\cong \s^{d-1}\times \punct{X}$, 
\item the face $\partial_\Delta C_2(\punct{X})= p_b^{-1}(\Delta_{\punct{X}})$, which identifies with the unit normal bundle to the diagonal $\Delta_{\punct{X}}$, which is diffeomorphic to the unit tangent bundle $\unitaire \punct{X}$ via the map $[(u,v)]_{(x,x)}\in U\mathfrak N_{(x,x)}\Delta_{\punct{X}} \mapsto \left[v-u\right]_x\in \unitaire_x \punct{X}$.\end{itemize}
\end{nt}

The following lemma can be proved as \cite[Lemma 2.2]{[Lescop]}.
\begin{lm}\label{Gauss}
When $\punct{X}=\R^d$, the Gauss map $$\begin{array}{lll} C_2^0(\R^d) & \rightarrow & \s^{d-1}\\ (x,y) & \mapsto & \frac{y-x}{||y-x||}\end{array}$$ extends to a map $G\colon C_2(\R^d)\rightarrow \s^{d-1}$. 

Furthermore, $G$ reads as follows on the faces\footnote{Here, we do not give the expression of $G$ on the three codimension $2$ faces. It can be found inside the proof of \cite[Lemma 2.2]{[Lescop]}} of codimension $1$ of $C_2(\R^d)$:

$$G(c) =\left\{\begin{array}{lll} 
  \frac{\frac{v}{||v||^2} - \frac{u}{||u||^2}}{\big|\big| \frac{v}{||v||^2} - \frac{u}{||u||^2}\big|\big|} &\text{if $c= [u, v]$ is in the interior of $ \partial_{\infty,\infty}C_2(\R^{d})$}\\
 -u &  \text{if $c= (u, y)\in\partial_{\infty,\R^d}C_2(\R^d)= \s^{d-1}\times \R^d$}\\
u & \text{if $c=(x, u)\in\partial_{\R^d,\infty}C_2(\R^d)=\R^d \times \s^{d-1}$}\\
\frac{u}{||u||}& \text{if  $c= [u]_x\in \unitaire_x\R^d\subset\unitaire\R^d\cong \partial_{\Delta}C_2(\R^d)$}
\end{array}\right.$$

\end{lm}

This map $G$ exists only when $\punct{X}=\R^d$, but, if $(\ambientspace, \tau)$ is a parallelized asymptotic homology $\R^{n+2}$, it is possible to define an analogue $G_\tau$ of $G$ on the boundary of $\configM$, as in \cite[Proposition 2.3]{[Lescop]}.

\begin{df}\label{configM}Let $(\ambientspace, \tau)$ be a parallelized asymptotic homology $\R^{n+2}$.
Note that the face $\partial_{\infty, \infty}\configM$ is canonically identified with $\partial_{\infty, \infty} C_2(\R^{n+2})$.

Then, we can define a smooth map $G_\tau\colon\partial\configM \rightarrow\s^{n+1}$ by the following formula: $$G_\tau(c) =\left\{\begin{array}{lll} 
G(c) &\text{if $c\in \partial_{\infty,\infty}\configM\cong \partial_{\infty,\infty}C_2(\R^{n+2})$}\\
 -u &  \text{if $c= (u, y)\in\partial_{\infty,\punct M}\configM = \s^{n+1}\times \ambientspace$}\\
u & \text{if $c=(x, u)\in\partial_{\punct M,\infty}\configM=\ambientspace \times \s^{n+1}$}\\
\frac{\tau_x^{-1}(u)}{||\tau_x^{-1}(u)||} & \text{if  $c= [u]_x\in \unitaire_x\ambientspace\subset\unitaire\ambientspace\cong \partial_{\Delta}\configM$}
\end{array}\right.$$

\end{df}
One can think of this map as a limit of the Gauss map when one or both points approach infinity (where everything is standard), or when they are close to each other. In the latter case, the limit is defined by the parallelization.

\subsection{Configuration spaces}\label{S2.4}

Let $\Gamma$  be a BCR diagram, let $(\ambientspace,\tau)$ be a parallelized asymptotic homology $\R^{n+2}$, and let $\psi\colon \R^n \hookrightarrow \espaceambient$ be a long knot. Let $\confignoeud[\Gamma][\psi][^0]$ denote the \emph{open configuration space} \[\confignoeud[\Gamma][\psi][^0]= \{c\colon \sommets \hookrightarrow \ambientspace \mid \mbox{there exists } c_i \colon \sommetsinternes\hookrightarrow \R^n\mbox{ satisfying } c_{|\sommetsinternes} = \psi\circ c_i
\}.\]

An element $c$ of $\confignoeud[\Gamma][\psi][^0]$ is called a \emph{configuration}. By definition, the images of the vertices under a configuration are distinct, and the images of internal vertices are on the knot.

This configuration space is a non-compact smooth manifold. It admits a compactification $\confignoeud$, which is defined in \cite[Section 2.4]{[Rossi]}, and which is the closure of the image of the map $c\in \confignoeud[\Gamma][\psi][^0] \mapsto c^* \in C_{\sommets\cup\{*\}}( M)$, where ${c^*}_{|\sommets }= c$ and $c^*(*)=\infty$, and where $C_{\sommets\cup\{*\}}(M)$ is the compact configuration space defined in \cite{[Sinha]}.

\begin{theo}[Rossi, Sinha]The manifold $\confignoeud$ is a compact manifold with corners, such that: \begin{itemize}
\item The interior of $\confignoeud$ is canonically diffeomorphic to $\confignoeud[\Gamma][\psi][^0]$.
\item For any two internal vertices $v$ and $w$, the map $c\in \confignoeud[\Gamma][\psi][^0] \mapsto (c_i(v), c_i(w))\in \configR$ extends to a smooth map $p^{\psi,i}_{v,w}\colon \confignoeud\rightarrow \configR$.
\item For any two vertices $v$ and $w$, the map $c\in \confignoeud[\Gamma][\psi][^0] \mapsto (c(v), c(w))\in \configM$ extends to a smooth map $p^{\psi}_{v,w}\colon \confignoeud\rightarrow \configM$.
\end{itemize}
\end{theo}

\begin{df}\label{defpe}
The manifold $\confignoeud$ is called the \emph{(compact) configuration space associated to $\Gamma$ and $\psi$}.
For any edge $e$ of $\Gamma$ going from a vertex $v$ to a vertex $w$, $C_e$ denotes the configuration space $\configR$ if $e$ is internal, and $\configM$ if $e$ is external, and $p_e^{\psi}\colon\confignoeud\rightarrow C_e$ denotes the map $p_{v,w}^{\psi,i}$ if $e$ is internal, and the map $p_{v,w}^{\psi}$ if $e$ is external. When there is no ambiguity on the knot $\psi$, $p_e^\psi$ is simply denoted by $p_e$.
\end{df}
Orient $C_\Gamma^0(\psi)$ as follows.

Let $\d Y_i^{v}$ denote the $i$-th coordinate form of the internal vertex $v$ (parametrized by $\R^n$) and 
let$ \d X_i^{v}$ denote the $i$-th coordinate form of the external vertex $v$ (in an oriented chart of $\ambientspace$).

Split each external edge $e$ in two halves: the tail $e_-$ and the head $e_+$. Define a form $\orientationarete{e_\pm}$ for these external half-edges as follows: \begin{itemize}
\item for the head $e_+$ of a leg going to an external vertex $v$, $\Omega_{e_+} = \d X_v^1$,
\item for the head $e_+$ of an edge that is not a leg, going to an external vertex $v$, $\Omega_{e_+} = \d X_v^2$,
\item for the tail $e_-$ of an edge coming from an external vertex $v$, $\Omega_{e_-} = \d X_v^3\wedge\cdots\wedge\d X_v^{n+2}$,
\item for any (external) half-edge $e_\pm$ adjacent to an internal vertex $v$, $\Omega_{e_\pm} = \d Y_v^1\wedge\ldots\wedge\d Y_v^n$.
\end{itemize}
Note that this distributes the coordinates of each vertex on the half-edges that are adjacent to it, as in Figure \ref{orfig}.
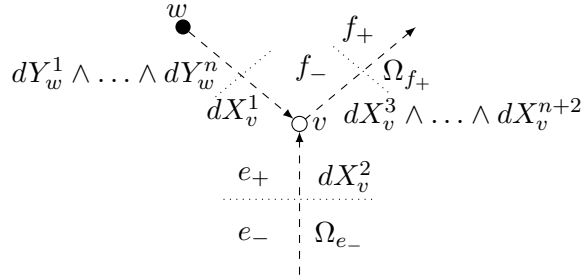
\begin{figure}[H]\centering
\begin{tikzpicture}
\draw (0.265,0) node {$v$};
\draw (0, 0) circle (0.1) ;
\draw (137: 2.2) node {$w$};
\draw [->, >= latex, dashed](-90: 2)-- (-90: 0.1) ;
\draw [<-, >= latex, dashed](40: 2)-- (40: 0.1) ;
\draw [dotted] (-1,-1) -- (1, -1) ;
\draw [dotted] (40: 1) ++(130:0.5) -- ++(-40: 1) ;
\draw [dotted] (140: 1)  ++(40:0.5) -- ++(-140:1);
\draw [->, >= latex, dashed](140: 1.9)-- (140: 0.1) ;
\draw (0.6,-0.7) node {$dX_v^2$};
\draw (-0.595,-0.7) node {$e_+$};
\draw (-0.595,-1.5) node {$e_-$};
\draw (0.395,-1.5) node {\ \  $\orientationarete{e_-}$};
\draw (40: 0.2) ++ (2,0) node {$ dX_v^3 \wedge \ldots \wedge dX_v^{n+2}$};
\draw (40: 1.4) ++ (-0.3,0.35) node {$f_+ $};
\draw (40: 1.4) ++ (0.35,-0.25) node {$\orientationarete{f_+} $};
\draw (40: 0.6) ++ (-0.3,0.35) node {$f_- $};
\draw (140: 0.7) ++ (220:0.43) node {$ dX^1_v$};
\draw (140: 1.8) ++ (205:1.15) node {$dY_w^1\wedge\ldots\wedge dY_w^n$};
\fill (140:2) circle (0.1);
\end{tikzpicture}
\caption{The forms associated to some external half-edges}
\label{orfig}
\end{figure}
Let $N_{T,i}(\Gamma)$ denote the number of internal trivalent vertices, and define the \emph{sign} of a BCR diagram as $\epsilon(\Gamma) = (-1)^{N_{T,i}(\Gamma)+\Card(\aretesexternes)}$. The orientation of $\confignoeud[\Gamma][\psi][0]$ is given by the form $\orientationconfig=\epsilon(\Gamma)\bigwedge\limits_{e\in\aretesexternes}\orientationarete{e}$, where $\orientationarete{e}= \orientationarete{e_-}\wedge\orientationarete{e_+}$ for any external edge $e$.

\subsection{Propagating forms}

Here we define the notion of propagating forms, which allows us to extend the definition of the BCR invariants to all parallelizable asymptotic homology $\R^{n+2}$. 

For any even integer $d$, an \emph{antisymmetric form} on $\s^d$ is a form $\omega$ such that $(-\mathrm{Id}_{\s^d})^*(\omega) = - \omega$, where $-\mathrm{Id}_{\s^d}$ is the antipodal map of the sphere.

\begin{df}
An \emph{internal propagating form} (or \emph{internal propagator}) is a closed $(n-1)$-form $\propagateurinterne$ on $\configR$ such that $\propagateurinterne_{| \partial\configR} = ({G}_{|\partial \configR})^*(\omega_{\propagateurinterne})$ where $\omega_{\propagateurinterne}$ is an antisymmetric volume form on $\s^{n-1}$  such that $\int_{\s^{n-1}} \omega_{\propagateurinterne} = 1$, and where $G\colon \configR \rightarrow \s^{n-1}$ is the map defined in Lemma \ref{Gauss}.

An \emph{external propagating form} (or \emph{external propagator}) of $(\ambientspace, \tau)$ is a closed $(n+1)$-form $\propagateurexterne$ on $\configM$ such that $\propagateurexterne_{| \partial\configM} = {G_\tau}^*(\omega_{\propagateurexterne})$ where $\omega_{\propagateurexterne}$ is an antisymmetric volume form on $\s^{n+1}$  such that $\int_{\s^{n+1}} \omega_{\propagateurexterne} = 1$, and where $G_\tau$ is the map of Definition \ref{configM}.

For a given integer $k$, a \emph{family $F=\familleformes$ of propagating forms of $(\ambientspace,\tau)$} is the data of $2k$ internal propagating forms $(\alpha_i)_{1\leq i \leq 2k}$ and $2k$ external propagating forms $(\beta_i)_{1\leq i \leq 2k}$ of $(\ambientspace, \tau)$.

Given such a family and a degree $k$ numbered BCR diagram $(\Gamma,\sigma)$, for each edge $e$ of $\Gamma$, set \[\formearete=\left\{
\begin{array}{lll}{p_e}^*(\alpha_{\sigma(e)}) & \text{if $e$ is internal,}\\
{p_e}^*(\beta_{\sigma(e)}) & \text{if $e$ is external.}
\end{array}
\right.\]

\end{df}

For any edge $e$, $n(e)$ denotes the integer $n-1$ if $e$ is internal, and $n+1$ if $e$ is external, so that $\formearete$ is an $n(e)$-form on $\confignoeud$. We will see in Corollary \ref{lm-hom2} that families of propagating forms exist.
 
\subsection[Definition and properties of generalized BCR invariants]{Definition and properties of generalized BCR invariants of long knots}\label{26}

Fix an integer $k\geq 2$, and a family $F=\familleformes$ of propagating forms of $(\ambientspace,\tau)$. 

Let $\psi$ be a long knot.

For any numbered BCR diagram $(\Gamma,\sigma)$ of degree $k$, define\footnote{The order of the forms inside the wedge product is not important since they have even degree.} the form $\forme$ on $\confignoeud$ as $\forme = \bigwedge\limits_{e\in \aretes}\formearete$, and set $ \I=\int_{\confignoeud}\forme$. This integral is a real number because of the following lemma.

\begin{lm}\label{dim=deg}
For any BCR diagram $\Gamma$, $\dim(\confignoeud) = \deg(\forme)$.
\end{lm}

\begin{proof}
Split any edge $e$ of $\Gamma$ in two halves $e_-$ (the tail) and $e_+$ (the head), and let $v(e_\pm)$ denote the vertex adjacent to the half-edge $e_\pm$.
Assign an integer $d(e_\pm)$ to each half-edge as follows: \begin{itemize}
\item If $e$ is external, $d(e_+)=1$ and $d(e_-)=n$ as in $\vcenter{\begin{tikzpicture} \draw [>=latex, dashed, ->] (0,0) -- (2,0); \draw [thick,densely dotted] (1,  -0.2) -- (1,0.2) ;
\draw (0.4, 0.2) node {$n$} (1.5,0.2) node{$1$} (0.4, -0.3) node {$e_-$} (1.6, -0.3) node {$e_+$};\end{tikzpicture}}$.
\item If $e$ is internal, $d(e_+)=0$ and $d(e_-)= n-1$ as in $\vcenter{\begin{tikzpicture} \draw [>=latex, ->] (0,0) -- (2,0); \draw [thick,densely dotted] (1,  -0.2) -- (1,0.2) ;
\draw (0.4, 0.2) node {$n-1$} (1.5,0.2) node{$0$} (0.4, -0.3) node {$e_-$} (1.6, -0.3) node {$e_+$};\end{tikzpicture}}$.
\end{itemize}
Note that, with these notations: \begin{itemize}
\item for any edge $e\in\aretes$, $d(e_+) + d(e_-) = n(e)$.
\item for any vertex $v\in\sommets$, as it can be checked in Figure \ref{somme-d}, \[\sum\limits_{e_\pm, v(e_\pm) = v} d(e_\pm) = \left\{\begin{array}{lll} n & \text{if $v$ is internal,}\\ n+2 & \text{if $v$ is external.}\end{array}\right. \]
\end{itemize}
\begin{figure}[H]
\centering
\begin{tikzpicture}
\draw (0,0.27) node {$v$};
\draw (0, 0) circle (0.1) ;
\draw [->, >= latex, dashed](-90: 1)-- (-90: 0.1) ;
\draw [<-, >= latex, dashed](40: 1)-- (40: 0.1) ;
\draw [->, >= latex, dashed](140: 1)-- (140: 0.1) ;
\fill (140:1.1) circle (0.1);
\draw (140: .5) ++(50:.2)-- ++(50:-.4);
\draw (-90: .5) ++(0:.2)-- ++(0:-.4);
\draw (40: .5) ++(130:.2)-- ++(130:-.4);
\draw (-90: .25) ++ (.14, -0.1) node {\tiny $1$};
\draw (40: .25) ++ (.15, -0.1) node {\tiny $n$};
\draw (140: .25) ++ (-.18, -.15) node {\tiny $1$};
\end{tikzpicture} \ \ \ \ \ \ \ \ 
\begin{tikzpicture}
\fill (0, 0) circle (0.1) ;
\fill (140:1.1) circle (0.1);
\draw (0,0.27) node {$v$};
\draw [->, >= latex, thick](-90: 1)-- (-90: 0.1) ;
\draw [<-, >= latex, thick](40: 1)-- (40: 0.1) ;
\draw [->, >= latex, dashed](140: 1)-- (140: 0.1) ;
\draw (140: .5) ++(50:.2)-- ++(50:-.4);
\draw (-90: .5) ++(0:.2)-- ++(0:-.4);
\draw (40: .5) ++(130:.2)-- ++(130:-.4);
\draw (-90: .25) ++ (.14, -0.1) node {\tiny $0$};
\draw (40: .25) ++ (.35, -0.05) node {\tiny $n-1$};
\draw (140: .25) ++ (-.18, -.15) node {\tiny $1$};
\end{tikzpicture} \ \ \ \ \ \ \ \ 
\begin{tikzpicture}
\fill (0, 0) circle (0.1) ;
\draw (0,0.27) node {$v$};
\draw [<-, >= latex, dashed](-90: 1)-- (-90: 0.1) ;
\draw (-90: .5) ++ (-0.2,0)-- ++(0.4, 0);
\draw (-90: .25) ++(.2, 0) node {\tiny $n$};
\end{tikzpicture} \ \ \ \ \ \ \ \ 
\begin{tikzpicture}
\fill (0, 0) circle (0.1) ;
\draw (0,0.27) node {$v$};
\draw [->, >= latex, dashed](-90: 1)-- (-90: 0.1) ;
\draw [<-, >= latex, thick](40: 1)-- (30: 0.1) ;
\draw (-90: .5) ++ (-0.2,0)-- ++(0.4, 0);
\draw (40: .5) ++(130:.2)-- ++(130:-.4);
\draw (-90: .25)++(.2, 0) node {\tiny $1$};
\draw (40: .25) ++ (.33, -0.06) node {\tiny $n-1$};
\end{tikzpicture}  \ \ \ \ \ \ \ 
\begin{tikzpicture}
\fill (0, 0) circle (0.1) ;
\draw (0,0.27) node {$v$};
\draw [->, >= latex, thick](-90: 1)-- (-90: 0.1) ;
\draw [<-, >= latex, dashed](40: 1)-- (30: 0.1) ;
\draw (-90: .5) ++ (-0.2,0)--++ (0.4, 0);
\draw (40: .5) ++(130:.2)-- ++(130:-.4);
\draw (-90: .25)++(.2, 0) node {\tiny $0$};
\draw (40: .25) ++ (.15, -0.05) node {\tiny $n$};
\end{tikzpicture}
\caption{ }\label{somme-d}
\end{figure}

Then, \begin{eqnarray*}\deg(\forme) &=& \sum\limits_{e\in\aretes}( d(e_+)+d(e_-) )= \sum\limits_{v\in \sommets}\sum\limits_{e_\pm, v(e_\pm) = v} d(e_\pm)\\& = &\sum\limits_{v \in \sommetsinternes} n+\sum\limits_{v\in\sommetsexternes} (n+2) = \dim(\confignoeud).\qedhere\end{eqnarray*}

\end{proof}

\begin{theo}\label{th1}
Set \[\Z(\psi) =\frac1{(2k)!} \sum\limits_{(\Gamma,\sigma)\in \graphesnum} \I.\]
The following properties hold: 
\begin{enumerate}
\item The value of $\Z(\psi)$ does not depend on the choice of the family $F$ of propagating forms of $(\ambientspace, \tau)$.
\item The value of $Z_k(\psi)= \Z(\psi)$ does not depend on the choice of the parallelization $\tau$ of the ambient manifold $\ambientspace$.
\item For any $\phi\in \mathrm{Diffeo}^+(\ambientspace)$ that fixes $\voisinageinfini$ pointwise, and for any long knot $\psi$ of $\ambientspace$, $Z_k(\psi)= Z_k(\phi\circ \psi)$. In particular, $Z_k$ is a long knot isotopy invariant.
\item The invariant $Z_k$ takes only rational values.

\item If $k$ is odd, $Z_k$ is always zero.
\end{enumerate}
The obtained invariant $Z_k$ is called the \emph{generalized BCR invariant of degree $k$}.

\end{theo}
When $\ambientspace = \R^{n+2}$, and when all the propagators are pullbacks of the homogoneous unit volume form on $\s^{n-1}$ and $\s^{n+1}$ with total volume one, our definition matches the definition of the invariants\footnote{Only $\Theta_2$ and $\Theta_3$ are explicitly defined in \cite{Cattaneo2005}, but the definition for higher $k$ is mentioned.} $(\Theta_k)_{k\geq 2}$ of \cite[Section 6]{Cattaneo2005} and of the invariants $2z_k$ of \cite[Section 2.4]{[Watanabe]} (we have $Z_k  = \Theta_k = 2z_k$). Our definition allows more flexibility on the choice of the forms. It extends the invariant to an invariant for long knots in any parallelized asymptotic homology $\R^{n+2}$. In \cite[Theorem 4.1]{[Watanabe]}, Watanabe proved that $z_k$ is not trivial when $k$ is even and $\punct M = \R^{n+2}$, and he related $z_k$ to Alexander polynomial for long ribbon knots.

\subsection{Propagating chains}
Let us first fix some notations on the chains used in this article.
\begin{df}
A \emph{rational $k$-chain} $A$ of a manifold $X$ is a finite rational combination $\sum\limits_{i=1}^r q_iY_i$ of compact oriented $k$-submanifolds with corners $(Y_i)_{1\leq i \leq r}$ of $X$. The boundary $\partial A$ of $A$ is the rational $(k-1)$-chain $\partial A= \sum\limits_{i=1}^r q_i \partial Y_i$, up to the usual algebraic cancellations\footnote{These cancellations allow us to write $1.(-Y)=(-1).Y$ for a submanifold $Y$, where $-Y$ denotes the manifold $Y$ with the opposite orientation, and $1.(Y\sqcup Z) =1.Y+1.Z$ for disjoint submanifolds $Y$ and $Z$, for example.}.

If the $(Y_i)_{1\leq i \leq r}$ have pairwise disjoint interiors, $A$ is called an \emph{embedded rational $k$-chain.}\footnote{Note that any rational chain is homologous to an embedded one.} If $A$ is an embedded rational $k$-chain, $\mathrm{Supp}(A) = \bigcup\limits_{i=1}^r Y_i$ denotes the \emph{support} of $A$, $A^{(k-1)}=\bigcup\limits_{i=1}^r \partial Y_i$ denotes its $(k-1)$-skeleton, and $\Int(A)=\mathrm{Supp}(A) \setminus A^{(k-1)}$ its \emph{interior}.
\end{df}
Let us now define the notion of \emph{propagating chains}, which will give us another way of computing the invariant $Z_k$, and help us to prove the fourth assertion of Theorem~\ref{th1}.
\begin{df}
An \emph{internal propagating chain} (or internal propagator) is an embedded rational $(n+1)$-chain $\Propint$ of $\configR$ such that there exists $x_\Propint\in\s^{n-1}$ such that $\partial \Propint = \frac12(G_{|\partial\configR})^{-1}(\{-x_\Propint, x_\Propint\})$.

An \emph{external propagating chain} (or external propagator) of $(\ambientspace, \tau)$ is an embedded rational $(n+3)$-chain $\Propext$ of $\configM$ such that there exists $x_\Propext\in\s^{n+1}$ such that $\partial \Propext = \frac12 G_\tau^{-1}(\{- x_\Propext, x_\Propext\})$.

A \emph{family $F_*=(\Propint_i,\Propext_i)_{1\leq i \leq 2k}$ of propagating chains of $(\ambientspace, \tau)$} is the data of $2k$ internal propagating chains $(\Propint_i)_{1\leq i \leq 2k}$ and $2k$ external propagating chains $(\Propext_i)_{1\leq i \leq 2k}$ of $(\ambientspace, \tau)$.
\end{df}

Consider a family $F_*=(\Propint_i,\Propext_i)_{1\leq i \leq 2k}$ of propagating chains of $(\ambientspace,\tau)$. For any BCR diagram $\Gamma$, set \[ \begin{array}{llll}P_\Gamma\colon& \confignoeud &\rightarrow &\prod\limits_{e\in\aretesinternes}\configR\times\prod\limits_{e\in\aretesexternes}\configM=\prod\limits_{e\in\aretes}C_e\\
&c & \mapsto&  (p_e(c))_{e\in\aretes}.
\end{array} \]

The family $F_*$  is in \emph{general position} if, for any numbered BCR diagram $(\Gamma,\sigma)\in \graphesnum$, and for any $c\in\confignoeud$ such that $P_\Gamma(c) \in\left(\prod\limits_{e\in\aretesinternes}\mathrm{Supp}(\Propint_{\sigma(e)})\right)\times\left(
\prod\limits_{e\in\aretesexternes}\mathrm{Supp}(\Propext_{\sigma(e)})\right)$: \begin{itemize}
\item For any internal edge $e$ of $\Gamma$, $p_e(c) \in \Int(A_{\sigma(e)})$.
\item For any external edge $e$ of $\Gamma$, $p_e(c) \in \Int(B_{\sigma(e)})$.
\item We have the transversality property
\begin{align*}
&T_{P_\Gamma(c)}\left( \prod\limits_{e\in\aretes}C_e\right)\\
= \ & T_cP_\Gamma(T_c\confignoeud) + \left(\left(\prod\limits_{e\in\aretesinternes}T_{p_e(c)}\Int(\Propint_{\sigma(e)})\right)\times\left(
\prod\limits_{e\in\aretesexternes}T_{p_e(c)}\Int(\Propext_{\sigma(e)})\right)
\right).
\end{align*}
\end{itemize}
In the following, $D_{e,\sigma}^{F_*}$ denotes the chain $p_e^{-1}(\Propint_{\sigma(e)})$ if $e$ is internal, and the chain $p_e^{-1}(\Propext_{\sigma(e)})$ if $e$ is external. This is a chain of codimension $n(e)$ of $\confignoeud$.

\subsection{Computation of $Z_k$ in terms of propagating chains}
We can now give a discrete definition of our generalized BCR invariants.

\begin{theo}\label{the}
Let $F_*=(\Propint_i,\Propext_i)_{1\leq i\leq 2k}$ be a family of propagating chains of $(\ambientspace, \tau)$ in general position. 

The algebraic intersection number $I^{F_*}(\Gamma,\sigma,\psi)$ of the chains $(D_{e, \sigma}^{F_*})_{e\in\aretes}$ inside $\confignoeud$ makes sense and$$\Z[](\psi) =\frac1{(2k)!} \sum\limits_{(\Gamma,\sigma)\in\graphesnum} I^{F_*}(\Gamma,\sigma,\psi).$$

\end{theo}
This theorem is proved in Section \ref{proof-th*}, where a more precise definition of this intersection number is given.
The existence of families of propagating chains in general position is proved in Section \ref{sec_ex_ch}

\subsection{Additivity of $Z_k$ under connected sum}\label{S29}

Let $\punct{M_1}$ and $\punct{M_2}$ be two asymptotic homology $\R^{n+2}$. 
Let us define the connected sum $\punct{M_1}\sharp \punct{M_2}$. Let $\voisinageinfinideux$ be the complement in $\R^{n+2}$ of the two open balls $\mathring B_1$ and $\mathring B_2$ of radius $\frac14$ and with respective centers $\Omega_1=(0, 0, \ldots, 0, -\frac12)$ and $\Omega_2=(0, 0, \ldots, 0, \frac12)$.
For $i\in\{1,2\}$ and $x$ in $\partial B(M_i)\subset \R^{n+2}$, define the map $\phi_i(x) = \frac14x +\Omega_i $, which is a diffeomorphism from $\partial B(M_i)$ to $\partial B_i$.

Set $\punct{M_1}\sharp \punct{M_2}= \voisinageinfinideux\cup B(M_1) \cup B(M_2)$, where $B(M_i)$ is glued to $\voisinageinfinideux$ along $\partial B_i$ using the map $\phi_i$, and set $B(\punct{M_1}\sharp\punct{M_2}) = \overline{(\punct{M_1}\sharp \punct{M_2} )\setminus \voisinageinfini}$, where $\voisinageinfini$ is defined in Section \ref{S11}. 
The manifold $\punct{M_1}\sharp\punct{M_2}$ with the decomposition $\punct{M_1}\sharp\punct{M_2}= B(\punct{M_1}\sharp \punct{M_2}) \cup \voisinageinfini$ is called the \emph{connected sum} of $\punct{M_1}$ and $\punct{M_2}$.
\begin{prop}
The obtained manifold $\punct{M_1}\sharp \punct{M_2}$ is an asymptotic homology $\R^{n+2}$ with two canonical injections $\iota_i \colon B(M_i) \hookrightarrow B(\punct{M_1}\sharp \punct{M_2})\subset\punct {M_1}\sharp \punct{M_2}$ for $i\in\{1,2\}$. 

If $\punct{M_1}$ and $\punct{M_2}$ are parallelized, $\punct{M_1}\sharp \punct{M_2}$ inherits a natural parallelization, up to homotopy.

\end{prop}
\begin{proof}
This is immediate.
\end{proof}

\begin{df}
Let $\punct{M_1}$ and $\punct{M_2}$ be two asymptotic homology $\R^{n+2}$.

Let $\psi_1 \colon \R^n\hookrightarrow \punct{M_1}$ and $\psi_2 \colon \R^n\hookrightarrow \punct{M_2}$ be two long knots. The formula $$(\psi_1\sharp \psi_2)(x)= \left\{\begin{array}{lll}
\iota_2(\psi_2(4.x_1, \ldots, 4.x_{n-1}, 4.x_n-2)) & \text { if $||x-(0, \ldots, 0, \frac12)|| \leq \frac14$,}\\
\iota_1(\psi_1(4.x_1, \ldots, 4.x_{n-1}, 4.x_n+2)) & \text { if $||x-(0, \ldots, 0, -\frac12)|| \leq \frac14$,}\\
(0, 0, x)\in \voisinageinfinideux &\text{otherwise,}
\end{array}
\right.
$$
defines a long knot $\psi_1\sharp \psi_2 \hookrightarrow \punct{M_1}\sharp\punct{M_2}$, which is called the \emph{connected sum} of $\psi_1$ and $\psi_2$.
\end{df}

Let us assert the following immediate result about connected sum.
\begin{lm}\label{neutre}
Set $\psi_{triv}\colon x\in \R^n\mapsto (0,0,x)\in\R^{n+2}$. The embedding $\psi_{triv}$ is called the \emph{trivial knot}. 

For any parallelizable asymptotic homology $\R^{n+2}$ $\punct M$ and for any long knot $\psi$ in $\punct M$, there exist two diffeomorphisms $\mathcal T^{(1)}_{\punct M, \psi}\colon\R^{n+2}\sharp \punct M\rightarrow\punct M$ and $\mathcal T^{(2)}_{\punct M,\psi}\colon\punct M\sharp \R^{n+2}\rightarrow \punct M$ such that $\mathcal T^{(1)}_{\punct M,\psi}\circ(\psi_{triv}\sharp\psi)=\psi=\mathcal T^{(2)}_{\punct M,\psi} \circ (\psi\sharp \psi_{triv})$.

Similarly, the connected sum is associative and commutative up to ambient diffeomorphisms.
\end{lm}

In Section \ref{S8}, we prove the following theorem.

\begin{theo}\label{additivity}
Let $\punct{M_1}$ and $\punct{M_2}$ be two parallelizable asymptotic homology $\R^{n+2}$ and let $\psi_1 \colon \R^n\hookrightarrow \punct{M_1}$ and $\psi_2 \colon \R^n\hookrightarrow \punct{M_2}$ be two long knots.
Then, for any $k\geq 2$,
\[ Z_k(\psi_1\sharp \psi_2) = Z_k(\psi_1)+Z_k(\psi_2).\]
\end{theo}

\subsection{Extension of $Z_k$ to any asymptotic homology $\R^{n+2}$}

We prove the following proposition at the end of Section \ref{annex2}.
\begin{prop}\label{prop210}
For any odd $n\geq 1$, the connected sum of any asymptotic homology $\R^{n+2}$ with itself is parallelizable in the sense of Definition \ref{paral-def}.

\end{prop}
Theorem \ref{th1}, Proposition \ref{prop210} and the additivity of $Z_k$ under connected sum of Theorem \ref{additivity} show that the following definition is consistent.

\begin{df}\label{Zkext} Let $\psi$ be a long knot in a (possibly non-parallelizable) asymptotic homology $\R^{n+2}$ with $n$ odd $\geq3$.
Define $Z_k(\psi)$ as $\frac12 Z_k(\psi\sharp \psi) $.

 \end{df}
 By construction, $Z_k$ still satisfies the three last points of Theorem \ref{th1}: it is invariant under ambient diffeomorphisms, takes rational values, and is trivial when $k$ is even.
 The associativity and commutativity of connected sum up to ambient diffeomorphisms of Lemma \ref{neutre} and Theorem \ref{additivity} show the following proposition, which extends Theorem \ref{additivity}.
 
 \begin{prop}\label{additivity2}
Let $\punct{M_1}$ and $\punct{M_2}$ be two asymptotic homology $\R^{n+2}$ and let $\psi_1 \colon \R^n\hookrightarrow \punct{M_1}$ and $\psi_2 \colon \R^n\hookrightarrow \punct{M_2}$ be two long knots.
Then, for any $k\geq 2$,
\[ Z_k(\psi_1\sharp \psi_2) = Z_k(\psi_1)+Z_k(\psi_2).\]
 \end{prop} 

\section{Independence of the propagating forms}\label{S2}
In this section, we study the effect on $Z_k$ of a change in the family of propagating forms. Without loss of generality, it suffices to study how $Z_k$ changes when $\propint_1$ and $\propext_1$ change.
\subsection{Expression of the dependence in terms of boundary integrals}\label{S21}
For later purposes, we allow a more general context: as previously, we suppose that a family $F=(\propint_i,\propext_i)_{1\leq i\leq 2k}$ of propagating forms is given, but we allow the forms $\propext_i$ to be compatible with different parallelizations $\tau_i$ of $\ambientspace$ (which means that $(\propext_i)_{|\partial\configM} = G_{\tau_i}^*(\omega_{\propext_i})$). This will allow us to use the results of this section in the proof of the independence of the parallelization in Section \ref{s52}. For simplicity, we set $\omega_i^{n-1} = \omega_{\alpha_i}$ and $\omega_i^{n+1} = \omega_{\beta_i}$.

Let $\tau'_1$ be a parallelization of $\punct M$.
Let $F'= (\alpha'_i, \beta'_i)_{1\leq i\leq 2k}$ be a family of propagating forms such that for any $i\geq 2$, $(\alpha'_i, \beta'_i)=(\alpha_i, \beta_i)$, and such that $\beta'_1$ is an external propagating form for $\tau'_1$, and $\alpha'_1-\alpha_1$ and $\beta'_1 -\beta_1$ are exact forms. We set $(\omega^{n-1}_1)'= \omega_{\alpha'_1}$ and $(\omega^{n+1}_1)'= \omega_{\beta'_1}$.

Let $\ecartinterne$ be an $(n-2)$-form on $\configR$ and let $\ecartexterne$ be an $n$-form on $\configM$ such that $\propint'_1=\propint_1+\d\ecartinterne$ and $\propext'_1=\propext_1+\d\ecartexterne$.

We say that \emph{$(\propint'_1- \propint_1, \propext'_1-\propext_1)$ has the sphere factorization property} if we can choose the forms $(\ecartinterne,\ecartexterne)$ such that there exists an antisymmetric $(n-2)$-form $\volint$ on $\s^{n-1}$ such that $\ecartinterne_{|\partial\configR} = {G_{|\partial\configR}}^*(\volint)$ and an antisymmetric $n$-form $\volext$ on $\s^{n+1}$ such that $\ecartexterne_{|\partial\configM} = {G_{\tau_1}}^*(\volext)$. In the following, when this property is assumed, we always choose such primitives.
 
For any $(\Gamma,\sigma)\in\graphesnum$, and for any edge $e$ of $\Gamma$, define the form $$\formeareteT =\left\{\begin{array}{ll} \formearete & \text{if $\sigma(e)\neq 1$,}\\
p_e^*(\ecartinterne) & \text{if $\sigma(e)=1$ and $e$ is internal,}\\
p_e^*(\ecartexterne) & \text{if $\sigma(e)=1$ and $e$ is external,}\end{array}\right.$$ and set  $\formeT=  \bigwedge\limits_{e\in\aretes }\formeareteT$, where the order of the forms is not important since all of them except one have even degree.

\begin{lm}\label{lm31}
With these notations,
\[\displaystyle\Z[F'](\psi)-\Z(\psi) = \frac1{(2k)!}\sum\limits_{(\Gamma,\sigma)\in\graphesnum}\int_{\partial\confignoeud} \formeT.\]
\end{lm}

\begin{proof}From the Stokes formula, it directly follows that \[\I[F'] -\I=\int_{\confignoeud}\d\formeT= \int_{\partial\confignoeud}\formeT.\qedhere\]

\end{proof}
\subsection{Codimension $1$ faces of $\confignoeud$}\label{ds}

The codimension $1$ open faces of $\confignoeud$ are in bijection with the subsets $S$ of cardinality at least two of $\sommets[\Gamma][*]=\sommets\sqcup\{*\}$. Let $\partial_S\confignoeud$ denote the face associated to such an $S$ and let $\faces$ denote the set of the codimension $1$ faces. There are four types of faces in $\faces$: \begin{itemize}
\item If $S$ contains $*$, $\face$ is called an \emph{infinite face}, and its elements are configurations of $\confignoeud$ that map the vertices of $S\setminus\{*\}$ to infinity, and all the other vertices to pairwise distinct points of $\ambientspace$. 
\item If $S= \sommets$, $\face$ is called the \emph{anomalous face}. Its elements are configurations that map all the vertices to one point, which is necessarily on the knot.

\item If $S$ has exactly two points, which are connected by exactly one edge, $\face$ is called a \emph{principal face} and its elements are configurations that map the two vertices of $S$ to one point $x_S$ and all the other ones to pairwise distinct vertices of $\ambientspace \setminus\{x_S\}$.
\item Otherwise, $\face$ is called a \emph{hidden face}, and its elements are configurations that map all the vertices of $S$ to one point $x_S$, and all the other ones to pairwise distinct points of $\ambientspace \setminus\{x_S\}$. 
\end{itemize}

One can find precise descriptions of these faces in Section \ref{annex} or in \cite[pp 61-62]{[Rossi]}.

A \emph{numbered (codimension $1$) face} of $\confignoeud$ is a face $\partial_S\confignoeud$ as above, together with a numbering $\sigma$ of $\Gamma$.

For any numbered face $(\partial_S\confignoeud, \sigma)$, set $\delta_S\I[]= \int_{\face}\formeT$, so that $\displaystyle\Z[F'](\psi)-\Z(\psi)= \frac1{(2k)!}\sum\limits_{(\Gamma,\sigma)\in\graphesnum}\sum\limits_{S\in\faces}\delta_S\I[]$.

\subsection{Vanishing lemma for the face contributions}

\begin{restatable}{lm}{vth}\label{vanishing theorem}
If $S\subset \sommets$, $\Gamma_S$ denotes the subgraph of $\Gamma$ whose vertices are the elements of $S$ and whose edges are the edges of $\Gamma$ that connect two vertices of $S$.

\begin{itemize}
\item For any numbered infinite face $(\face, \sigma)$, such that no end of $\sigma^{-1}(1)$ is in $S$, $\delta_S\I[]=0$.
\item The set of hidden faces splits into two sets $\hfaces$ and $\Hfaces$, such that: \begin{itemize}
\item For any hidden face $\face$ of $\hfaces$ and any numbering $\sigma$, $\delta_S\I[]=0$.
\item For any hidden face $\face$ of $\Hfaces$, we have an involution $\sigma\mapsto \sigma^*$ of the numberings of $\Gamma$ such that $\delta_SI(\Gamma,\sigma^*,\psi) = -\delta_SI(\Gamma,\sigma,\psi)$.
\end{itemize}
\item Represent the principal faces by pairs $(\Gamma,e)$ where $\Gamma\in\graphes$ and $e\in\aretes$. For any numbering $\sigma$, let $\delta_e\I[]$ denote the integral $\delta_S\I[]$ where $S$ is the set of the two ends of $e$. Let $\mathcal N_{\neq1}(\Gamma, e)$ denote the set of the numberings of $\Gamma$ such that $\sigma(e)\neq1$, and let $\mathcal N(\Gamma )$ denote the set of all the numberings of $\Gamma$. Then: 
\begin{itemize}\item There exists an involution $s\colon(\Gamma,e)\mapsto(\Gamma^*, e^*)$ of the set of principal faces such that, for any $(\Gamma, e)$, there exists a canonical map $s_{\Gamma,e}\colon \sigma\in \mathcal N_{\neq1}(\Gamma, e)\mapsto \sigma^*\in\mathcal N_{\neq1}(\Gamma^*, e^*)$, such that $\delta_{e^*}I(\Gamma^*,\sigma^*,\psi) = - \delta_eI(\Gamma,\sigma,\psi)$ and such that $s_{\Gamma,e} \circ s_{\Gamma^*, e^*}= \mathrm{Id}$.
\item If $\sigma$ is a numbering of $\Gamma$ such that $\sigma(e) =1$, and, if $e$ is internal, or if $e$ is external with at least one external end, then $\delta_e\I[]=0$.
\end{itemize}
\end{itemize}
Furthermore, if $(\propint'_1- \propint_1, \propext'_1-\propext_1)$ has the sphere factorization property:
\begin{itemize}
\item For any infinite face $\face$ such that $S$ contains an end of $\sigma^{-1}(1)$, $\delta_S\I[]=0$.
\item The anomalous faces do not contribute: for any $(\Gamma, \sigma)\in\graphesnum$, $\delta_{\sommets}\I[]=0$.
\item For the principal faces $(\Gamma, e)$ associated to an edge $e$ where $e$ is external with internal ends, the map $s_{\Gamma,e}$ above can be extended to a map $\mathcal N(\Gamma)\rightarrow \mathcal N(\Gamma^*)$ such that $\delta_{e^*}I(\Gamma^*,\sigma^*,\psi) = - \delta_eI(\Gamma,\sigma,\psi)$ and $s_{\Gamma,e} \circ s_{\Gamma^*, e^*}= \mathrm{Id}$.
\end{itemize}
\end{restatable}

The proof of this lemma is given in Section \ref{annex}.

\subsection{Cohomology groups of two-point configuration spaces}

In this section, we study the cohomology of configuration spaces. This allows us to prove the existence of families of propagating forms and the independence of $Z_k^F$ of the propagating forms (first point of Theorem \ref{th1}) up to Lemma \ref{vanishing theorem} in the next subsection.

\begin{lm}\label{homologie}
Let $(\ambientspace, \tau)$ be a parallelized asymptotic homology $\R^{n+2}$.
The relative cohomology of $\configM$ is \[H^*(\configM, \partial \configM;\mathbb R) =\begin{cases}
\R & \mbox{if $* = n+3$ or $*= 2n+4$,}\\
0 & \mbox{otherwise.}
\end{cases}\]

\end{lm}

\begin{proof}

Since $\configM$ is a compact oriented $(2n+4)$-manifold, $$H^{2n+4}(\configM, \partial \configM; \R) = \R.$$ 

Fix $0\leq \ell \leq 2n+3$. The Poincaré-Lefschetz duality yields $$H^{\ell}(\configM,\partial\configM) = H_{2n+4-\ell}(\configM)= H_{2n+4-\ell}(C_2^0(\ambientspace)).$$ Furthermore, we have a long exact sequence \[ H_{2n+5-\ell}((\ambientspace)^2) \rightarrow H_{2n+5-\ell}((\ambientspace)^2, C_2^0(\ambientspace))\rightarrow H_{2n+4-\ell}(C_2^0(\ambientspace))\rightarrow H_{2n+4-\ell}((\ambientspace)^2)\] where $H_*((\ambientspace)^2)= H_*(\textrm{pt})$ by the Künneth formula. Then, we have an isomorphism $H_{2n+4-\ell}(C_2^0(\ambientspace)) \cong H_{2n+5-\ell}((\ambientspace)^2,C_2^0(\ambientspace))$.
The excision theorem yields \begin{eqnarray*}H_{2n+5-\ell}((\ambientspace)^2,C_2^0(\ambientspace))&=& H_{2n+5-\ell} (\mathcal N(\Delta_{\ambientspace}) , \mathcal N(\Delta_{\ambientspace})\setminus\Delta_{\ambientspace})\\&=& H_{2n+5-\ell}(\Delta_{\ambientspace}\times \mathbb D^{n+2}, \Delta_{\ambientspace}\times(\mathbb D^{n+2}\setminus\{0\})),\end{eqnarray*} where $\mathcal N(\Delta_{\ambientspace})$ is a tubular neighborhood of $\Delta_{\ambientspace}$, which can be identified with $\Delta_{\ambientspace}\times \mathbb D^{n+2}$ using the parallelization.
By Künneth's formula, \begin{eqnarray*}H_{2n+5-\ell}(\Delta_{\ambientspace} \times \mathbb D^{n+2}, \Delta_{\ambientspace} \times (\mathbb D^{n+2}\setminus\{0\}))&=& \bigoplus\limits_{i+j=2n+5-\ell}H_i(\Delta_{\ambientspace})\otimes H_j(\mathbb D^{n+2}, \s^{n+1})\\&=&H_{n+3-\ell}(\ambientspace)\otimes \R.\end{eqnarray*}
Therefore, $H^{\ell}(\configM,\partial \configM)\cong H_{n+3-\ell}(\ambientspace)$.
\end{proof}

\subsection{Existence of propagating forms. Independence of $Z_k^F$ of a choice of propagating forms}\label{sec_ex_ch}
The results of this section are applications of Lemma \ref{homologie}.

\begin{cor}\label{lm-hom2}
For any parallelized asymptotic homology $\R^{n+2}$ $(\ambientspace, \tau)$, there exist external propagating forms for $(\ambientspace, \tau)$.
\end{cor}

\begin{proof}The triviality of the cohomology group $H^{n+2}(\configM, \partial \configM)$ follows from the lemma. 
The restriction map $H^{n+1}(\configM) \rightarrow H^{n+1}(\partial \configM)$ is therefore surjective. Thus, given an antisymmetric closed $(n+1)$-form $\omega^{n+1}$ on $\s^{n+1}$, there exists a closed form $\beta_0^{n+1}$ on $\configM$ such that $[(\beta_0^{n+1})_{| \partial \configM}] = [{G_\tau}^*(\omega^{n+1})]$ in $H^{n+1}(\partial \configM)$. Then, there exists a form $\rho_0^n$ on $\partial\configM$ such that $(\beta^{n+1}_0)_{| \partial \configM} = {G_\tau}^*(\omega^{n+1}) + \d \rho_0^n$. Extend $\rho_0^n$ to a form $\rho^n$ on $\configM$, and set $\beta^{n+1}=\beta^{n+1}_0 -\d\rho^n$. The form $\beta^{n+1}$ is closed, and $(\beta^{n+1})_{| \partial \configM} = {G_\tau}^*(\omega^{n+1})$. The corollary is proved.
\end{proof}

Let us now prove the first point of Theorem \ref{th1}, i. e. that $Z_k^F$ does not depend on the choice of the family $F$ of propagating forms of $(\ambientspace, \tau)$.
Fix $(\punct M, \tau)$, and choose two families $F=\familleformes$  and $F'= (\alpha'_i, \beta'_i)_{1\leq i \leq 2k}$ of propagating forms of $(\ambientspace, \tau)$. 

As previously said, it suffices to show that $Z_k^F$ does not change if $\propint_1$ and $\propext_1$ change. Therefore, we assume that for any $i\geq 2$, $(\alpha'_i, \beta'_i) = (\alpha_i,\beta_i)$, without loss of generality, and we use the notations of Section \ref{S21}.
\begin{lm}\label{lmhom}
%
The pair $(\propint'_1-\propint_1, \propext'_1-\propext_1)$ has the sphere factorisation property.
\end{lm}\begin{proof}
By construction, $(\propagateurexterne'_1-\propagateurexterne_1)_{|\partial\configM} = G_\tau^*((\omega^{n+1}_{1})'-\omega_{1}^{n+1})$. Since $\int_{\s^{n+1}}(\omega_1^{n+1})' = \int_{\s^{n+1}}\omega_1^{n+1}$, there exists an $n$-form $\volext$ on $\s^{n+1}$ such that $\d\volext = (\omega^{n+1}_{1})'-\omega^{n+1}_{1}$. Since $\omega^{n+1}_{1}$ and $(\omega^{n+1}_{1})'$ are antisymmetric, $\volext$ can be assumed to be antisymmetric. Extend the form ${G_\tau}^*(\volext)$ to a form $\rho_1^n$ on $\configM$. Then, $\propagateurexterne'_1-\propagateurexterne_1-\d\rho_1^n$ is a closed form on $\configM$, whose restriction to $\partial\configM$ vanishes. Since $H^{n+1}(\configM, \partial\configM)=0$, 
there exists an $n$-form $\rho_2^n$ on $\configM$, 
which vanishes on $\partial\configM$, 
such that $ \propagateurexterne'_1-\propagateurexterne_1-\d\rho_1^n=\d\rho_2^n$. Set $\ecartexterne=\rho_1^n+\rho_2^n$, so that $\propagateurexterne'_1 -\propagateurexterne_1 = \mathrm{d} \ecartexterne$, $\ecartexterne_{|\partial\configM} = {G_{\tau}}^*(\volext)$ and $\volext$ is antisymmetric.

The same argument on $\propint'_1-\propint_1$ concludes the proof of Lemma \ref{lmhom}.
\end{proof}

From the previous lemma and Lemma \ref{vanishing theorem}, it follows that $Z_k^{F'}-Z_k^{F} = 0$. This proves the independence of $Z_k^F(\psi)$ of the family $F$ of propagating forms of $(\ambientspace, \tau)$. This is the first point of Theorem \ref{th1}.

\section[Rationality of Z_k]{Rationality of $Z_k$}\label{S4}

\subsection{Proof of Theorem \ref{the}    }\label{proof-th*}

Fix a family $F_*= (\Propint_i,\Propext_i)_{1\leq i \leq 2k}$ of propagating chains of $(\ambientspace, \tau)$ in general position. In order to prove that $Z_k$ can be computed with these propagators, we are going to define forms dual to them, and use the definition of $Z_k$. Fix Riemannian metrics on the configuration spaces $\configM$, $\configR$, and $\confignoeud$, and denote by $N_\epsilon(X)=\{c \mid d(c,X)\leq\epsilon\}$ the closed $\epsilon$-neighborhood of a subset $X$ of any of these spaces. 
Define \[\begin{array}{lll}\D &= \{c \in \confignoeud \mid (p_e(c))_{e\in\aretes} \in \prod\limits_{e\in \aretesinternes} \mathrm{Supp}(A_{\sigma(e)}) \times \prod\limits_{e\in \aretesexternes}\mathrm{Supp}(B_{\sigma(e)})\}\\ &= \bigcap\limits_{e\in\aretes}D_{e,\sigma}^{F_*}.\end{array}\]

Let $\epsilon >0$ be such that for any internal edge $e$, $p_e(\D)\subset \mathrm{Supp}(A_{\sigma(e)})\setminus N_\epsilon(A^{(n)}_{\sigma(e)})$, and such that for any external edge $e$, $p_e(\D)\subset \mathrm{Supp}(B_{\sigma(e)})\setminus N_\epsilon(B^{(n+2)}_{\sigma(e)})$.

Set $A_i^0=A_i\setminus N_\epsilon(A_i^{(n)})$, $N_\epsilon(A_i) = N_\epsilon(\mathrm{Supp}(A_i))$,$B_i^0=B_i\setminus N_\epsilon (B_i^{(n+2)})$, and $N_\epsilon(B_i) = N_\epsilon(\mathrm{Supp}(B_i))$.
For $\epsilon$ small enough, for any $x$ in $A_i^0$, there exists an open neighborhood $V_x\subset N_\epsilon (A_i)$ of $x$ in $\configR$, which can be thought of as a tubular neighborhood of an open neighborhood $W_x$ of $x$ in $A_i^0$, so that there is a local (orientation-preserving) trivialization $V_x\rightarrow W_x\times \mathbb D^{n-1} $. This induces a local fiber projection map $p_x\colon  V_x  \rightarrow \mathbb D^{n-1}$. This construction can be made so that if $V_x\cap V_{x'}\neq \emptyset$, there exists a rotation $r_{x,x'}\in SO(\R^{n-1})$ such that $(p_x)_{|V_x \cap V_{x'}}=(r_{x,x'}\circ p_{x'})_{|V_x \cap V_{x'}}$. For any $x\in B_i^0$, similarly define an open neighborhood $V_x\subset N_\epsilon(B_i)$ of $x$ in $\configM$, and a local fiber projection map $p_x\colon V_x\rightarrow \mathbb D^{n+1}$.

Some use of linear algebra and inverse function theorem proves the following lemma.
\begin{lm}\label{lm1}
For any $c\in \D$, there exists a neighborhood $U_c$ of $c$ in $\confignoeud$ such that for any $e\in \aretes$, $p_e(U_c)\subset V_{p_e(c)}$ and \[\begin{array}{rlll} \phi_c\colon& U_c &\rightarrow &\prod\limits_{e\in E(\Gamma)} \mathbb D^{n(e)}\\&y & \mapsto & (p_{p_{e}(c)}(p_{e}(y)))_{e\in E(\Gamma)}\end{array}\] is a diffeomorphism onto its image.
\end{lm}

Lemma \ref{lm1} implies that $\D$ is discrete in the compact space $\confignoeud$, so it is a finite set. 

Since $n(e)$ is even for any edge, $\prod\limits_{e\in E(\Gamma)} \mathbb D^{n(e)}$ is naturally oriented, and we can define $\mathrm{sgn}(\det(\d\phi_c))$ as the sign of the Jacobian $\det(\d\phi_c)$. 
For $c\in \D$, set $i(c) = \mathrm{sgn}(\det(\d\phi_c))\prod\limits_{e\in\aretes}q(p_e(c))$, where $q(p_e(c))$ is the coefficient $q_j$ of the submanifold $Y_j$ in which $p_e(c)$ lies in the rational chain $A_{\sigma(e)}$ (if $e$ is internal) or $B_{\sigma(e)}$ (if $e$ is external), which reads $ \sum\limits_i q_iY_i$. Then, the \emph{intersection number} $I^{F_*}(\Gamma,\sigma,\psi)$ is defined as $I^{F_*}(\Gamma,\sigma,\psi) = \sum\limits_{c\in\D} i(c)$.

The following lemma, which can be obtained as in \cite[Section 11.4, Lemma 11.13]{[Lescop2]} connects this intersection number to a configuration space integral, thus to the $Z_k$ invariant. 
\begin{lm}\label{lm2}
There exists a family $F=\familleformes$ of propagating forms of $(\ambientspace,\tau)$ such that for any $(\Gamma,\sigma)\in\graphesnum$: 
\begin{itemize}
\item The support of $\forme$ is a disjoint union of some neighborhoods $U_c$ of $c\in \D$ as in Lemma \ref{lm1}.
\item For any $c\in \D$, $\int_{U_c} \forme = i(c)$.

\end{itemize}
\end{lm}\begin{proof}[Sketch of proof]The main idea is to define the form $\alpha_i$ supported on $N_\epsilon(A_i)$, such that for any $x\in A_i^0$ $(\alpha_i)_{| V_x}= q(x).p_x^*(\omega^{n-1})$ where $\omega^{n-1}$ is a volum form of total volume one on $\mathbb D^{n-1}$ supported in the interior of $\mathbb D^{n-1}$ and $q(x)$ is the coefficient of the submanifold $Y_j$ in which $x$ lies in $A_i=\sum\limits_{k} q_{k}Y_{k}$. The proof of \cite[Section 11.4, Lemma 11.13]{[Lescop2]} explains how these forms can be "glued" along $N_\epsilon( A_i^{(n)})$ in order to get a closed form and how they can be defined on a collar of the boundary to get an internal propagating form. The construction of $\beta_i$ is similar. \qedhere

\end{proof}
Lemma \ref{lm2} implies Theorem \ref{the}. Indeed, with the family $F$ of propagating forms of the lemma, the integrals $\I$ of the definition of $Z_k$ in Theorem \ref{th1} are exactly the rational numbers $I^{F_*}(\Gamma,\sigma,\psi)$ of Theorem \ref{the}.

\subsection{Existence of propagating chains in general position}

Lemma \ref{homologie} and the Poincaré-Lefschetz duality imply that $H_{n}(\configR)$ and $H_{n+2}(\configM)$ are trivial groups. Therefore, propagating chains exist. 

As stated in the following theorem, these propagating chains can also be assumed to be in general position.
\begin{theo}\label{posgen}
For any family $(\Propint_i, \Propext_i)_{1\leq i \leq 2k}$ of propagating chains of $(\ambientspace, \tau)$, and any $\epsilon >0$, there exists a family $(\Propint'_i, \Propext'_i)_{1\leq i \leq 2k}$ of propagating chains of $(\ambientspace, \tau)$ in general position such that for any $1\leq i\leq 2k$, $\mathrm{Supp}(\Propint'_i)\subset N_\epsilon(\Propint_i)$ and $\mathrm{Supp}(\Propext'_i)\subset N_\epsilon(\Propext_i)$.
\end{theo}
\begin{proof}[Sketch of proof]
This theorem could be proved as in \cite[Section 11.3, Lemma 11.11]{[Lescop2]}. The main idea is to look at families of diffeomorphisms $(\phi_i, \phi_i')$ isotopic to the identity of the tubular neighborhoods $N_\epsilon(A_i), N_\epsilon(B_i)$ that act only fiberwise. In the space of such diffeomorphisms, the condition of general position on $(\phi_i(A_i), \phi_i'(B_i))$ can be proved to correspond to an open dense (so non empty) subset. Therefore, there exist some diffeomorphisms such that these perturbed chains are in general position.
\end{proof}

In particular, $Z_k$ can be computed with such propagating chains. By construction, this gives a rational number. This proves the fourth assertion of Theorem \ref{th1}.

\section{Nullity of $Z_k$ when $k$ is odd}

In this section, we prove the fifth assertion of Theorem \ref{th1}. The method is the same as in \cite[Section 2.5]{[Watanabe]}, but we have to deal with some more general propagating forms, and our orientations are not the same\footnote{The orientation of our configuration spaces is $w_k(\Gamma). \Omega_{Wat}(\Gamma)$ with the notations of \cite{[Watanabe]}.}.

Let $(\punct M, \tau)$ be a parallelized asymptotic homology $\R^{n+2}$. 

Let us fix an integer $k\geq 1$, a long knot $\psi$, and a family $F= (\alpha_i,\beta_i)_{1\leq i \leq 2k}$ of propagating forms of $(\ambientspace, \tau)$.

Let  $T_\alpha\colon C_2(\R^n) \rightarrow C_2(\R^n)$ denote the extension of the map $(x,y)\in C_2^0(\R^n) \mapsto (y,x) \in C_2(\R^n)$ to $C_2(\R^n)$. Similarly define $T_\beta \colon \configM\rightarrow\configM$.

Set $F'=(\alpha'_i, \beta'_i)_{1\leq i\leq 2k}= (\frac12(\alpha_i - T_\alpha^*(\alpha_i)), \frac12(\beta_i - T_\beta^*(\beta_i)))_{1\leq i \leq 2k}$. Since the forms $\omega_{\alpha_i}$ and $\omega_{\beta_i}$ are antisymmetric for any $1\leq i\leq 2k$, $F'$ is again a family of propagating forms of $(\ambientspace, \tau)$. For any $1\leq i \leq 2k$,  $T_\alpha^*(\alpha'_i) = - \alpha'_i$ and $T_\beta^*(\beta'_i) = - \beta'_i$.
\begin{prop}\label{zimpair}
For any $(\Gamma,\sigma)\in\graphesnum$, let $(\Gamma^*, \sigma^*)$ denote the numbered BCR diagram obtained from $(\Gamma,\sigma)$ by reversing all the edges of the cycle. Then, \[ I^{F'}(\Gamma^*,\sigma^*, \psi)= (-1)^{k} I^{F'}(\Gamma,\sigma,\psi).\]

\end{prop}
\begin{proof}
Since the vertices and their natures are the same for $\Gamma$ and $\Gamma^*$, we have a canonical diffeomorphism $C_\Gamma(\psi) \cong C_{\Gamma^*}(\psi)$, but it may change the orientation.

It follows from the definition of the orientation of configuration spaces in Section \ref{S2.4} that the orientation $\Omega(\Gamma^*)$ can be obtained from $\Omega(\Gamma)$ as follows: first, exchange the coordinate forms $\d X_v^2$ and $\d X_v^3\wedge \cdots \wedge \d X_v^{n+2}$ for any external vertex $v$ ; next, for any external edge $e$ of the cycle, exchange the forms $\Omega_{e_-}$ and $\Omega_{e_+}$.

Set $r=0$ if there is no internal edge in $\Gamma$. In this case, there are $k$ external vertices and $k$ external edges in the cycle, so $\Omega(\Gamma^*) = (-1)^{k+k}\Omega(\Gamma)=\Omega(\Gamma)$. Otherwise, any external edge of the cycle is contained in one maximal sequence of consecutive external edges of the cycle. If such a sequence has $d$ edges, it has $d-1$ external vertices. Let us denote by $(d_1, \ldots, d_r)$ the lengths of the $r$ maximal sequences of consecutive external edges of the cycle. Then, the previous analysis yields $\Omega(\Gamma^*) = (-1)^{\sum\limits_{i=1}^r(d_i + (d_i-1))} \Omega(\Gamma)= (-1)^r \Omega(\Gamma)$.

Let $L$ denote the number of edges of the cycle of $\Gamma$. Since $F'$ is such that for any $1\leq i \leq 2k$,  $T_\alpha^*(\alpha'_i) = - \alpha'_i$ and $T_\beta^*(\beta'_i) = - \beta'_i$, we have\footnote{We consider $\omega^{F'}(\Gamma^*, \sigma^*,\psi)$ as a form on $\confignoeud$ via the canonical identification $C_\Gamma(\psi) \cong C_{\Gamma^*}(\psi)$.} $\omega^{F'}(\Gamma^*, \sigma^*,\psi)=(-1)^{L} \omega^{F'}(\Gamma, \sigma,\psi)$.
Then, $I^{F'}(\Gamma^*,\sigma^*, \psi)= (-1)^{L+r} I^{F'}(\Gamma,\sigma,\psi)$.

It remains to check that $L+ r \equiv k\mod 2$. It is direct when there is no internal edge. 

Otherwise, let $u$ (resp. $b$, resp. $t$) denote the number of univalent (resp. bivalent, resp. trivalent) vertices of $\Gamma$. By definition of the BCR diagrams, $u= t$, and $2k= u + b + t= b + 2t$.

Note that there is a bijection between maximal sequences of consecutive external edges of the cycle and bivalent vertices with an external outgoing edge. This bijection is defined by taking the source of the first edge of a sequence. Taking the head of the last edge of a sequence also gives a bijection between the maximal sequences of consecutive external edges of the cycle and the bivalent vertices with an internal outgoing edge. Then, $r = \frac{b}2= k - t$.

The cycle is composed of all the bivalent and trivalent vertices, and has as many vertices and edges. Then, $L = b+ t= 2k- t$.

Eventually, $L + r = 3k - 2t \equiv k \mod 2$. This concludes the proof of Proposition \ref{zimpair}.\qedhere

\end{proof}

Proposition \ref{zimpair} directly implies that $Z_k(\psi) = 0$ when $k$ is odd.
\section{Independence of the parallelization, invariance under ambient diffeomorphisms}\label{S5}
In this section, we prove the second and third assertions of Theorem \ref{th1}.
\subsection{Homotopy classes of parallelizations of $\ambientspace$}

Let $\ambientspace$ be a fixed parallelizable asymptotic homology $\R^{n+2}$. Let $Z_k^{\tau}$ denote the value of the invariant $Z_k$ when computed with a family of propagating forms of $(\ambientspace,\tau)$.

We recall that two parallelizations $\tau$ and $\tau'$ are homotopic if there exists a smooth family $(\tau_t)_{0\leq t \leq 1}$ of parallelizations such that $\tau_0=\tau$ and $\tau_1=\tau'$, as in Definition \ref{paral-def}.

Denote by $\Par$ the set of homotopy classes of parallelizations of $\ambientspace$.

\begin{lm}
If $\tau$ and $\tau'$ are homotopic, then $Z_k^\tau= Z_k^{\tau'}$.
\end{lm}

\begin{proof} Let $(\tau_t)_{0\leq t \leq 1}$ be a smooth homotopy of parallelizations. Assume without loss of generality that there exists $\epsilon >0$ such that $\tau_t= \tau_0$ for any $t\in [0, \epsilon]$. Let $(\propint_i, \propext_i)$ be a family of propagating forms of $(\ambientspace,\tau_0)$. For any $1\leq i\leq 2k$, there exists a form $\omega_{\propext_i}$ such that $(\propext_i)_{|\partial\configM} = G_{\tau_0}^*(\omega_{\propext_i})$.

For any $1\leq i \leq 2k$, we define a smooth family $(\beta_i^s)_{0\leq s\leq 1}$ of external propagating forms such that $(\propext_i^s)_{|\partial\configM}= G_{\tau_s}^*(\omega_{\propext_i})$ as follows. 

Let $[-1,0]\times U\ambientspace$ be a collar of $U\ambientspace=\partial_\Delta\configM$ such that $\{0\}\times U\ambientspace$ corresponds to $\partial_\Delta\configM$. Let $N(\partial \configM)$ be a regular neighborhood of $\partial\configM$ that contains $[-1,0]\times U\bouleambiante$.
Extend $G_{\tau_0}$ to a smooth map $\overline{G_{\tau_0}}$ on $N(\partial \configM)$ such that for any $(t,x)\in [-1,0]\times U\bouleambiante$, $\overline{G_{\tau_0}}(t,x) = G_{\tau_0}(x)$. 
Assume that $(\beta_i)_{|N(\partial \configM)}= \overline{G_{\tau_0}}^*(\omega_{\beta_i})$. Since $(G_{\tau_s})_{|\partial UB(M)}= (G_{\tau_0})_{|\partial UB(M)}$ for any $s\in [0,1]$, the map $$p_s\colon (t,x) \in [-1,0]\times UB(M)\mapsto G_{\tau_{(1+t)s}}(x)\in\s^{n+1}$$ coincide with $\overline{G_{\tau_0}}$ on $ ([-1,0]\times\partial UB(M))\cup (\{-1\}\times UB(M))$. The forms ${p_s}^*(\omega_{\propext_i})$ and $\propext_i$ coincide on $ ([-1,0]\times\partial UB(M))\cup (\{-1\}\times UB(M))$. This allows us to define a closed form $\beta_i^s$ such that $(\beta_i^s)_{|[-1,0]\times UB(M)} = {p_s}^*(\omega_{\propext_i})$ and $(\beta_i^s)_{|\configM\setminus ([-1,0]\times UB(M))} = \beta_i$.


Therefore, $F_s=(\propint_i, \propext_i^{s})_{1\leq i\leq 2k}$ is a family of propagating forms of $(\ambientspace,\tau_s)$, and $$Z_k^{\tau_s} =\frac1{(2k)!} \sum\limits_{(\Gamma,\sigma)\in\graphesnum}\int_{\confignoeud}\omega^{F_s}(\Gamma,\sigma,\psi).$$ By construction, $\omega^{F_s}(\Gamma,\sigma,\psi)$ depends continuously on $s$, and then, $Z_k^{\tau_s}$ depends continuously on $s$. Since it takes only rational values, it is constant, and $Z_k^{\tau_0} = Z_k^{\tau_1}$.
\end{proof}

The following theorem will allow us to obtain the independence of $Z_k$ of the parallelization in the next subsection. It is proved in Section \ref{annex2}.

\begin{theo}\label{parallelization theorem}

Let $\ambientspace$ be an asymptotic homology $\R^{n+2}$, and let $\B\subset \bouleambiante$ be a standard $(n+2)$-ball.
Let $[\tau]$ and $[\tau']$ be two homotopy classes of parallelizations of $\ambientspace$ as defined in Definition \ref{paral-def}. 

It is possible to choose representatives $\tau$ and $\tau'$ of the classes $[\tau]$ and $[\tau']$, such that $\tau$ and $\tau'$ coincide on $(\ambientspace \setminus \mathbb B) \times \R^{n+2}$.
\end{theo}

\subsection{Proof of the independence of the parallelization}\label{s52}

Let $[\tau_0]$ and $[\tau_1]$ be two homotopy classes of parallelizations of $\ambientspace$. Let $\B$ be a ball of $\bouleambiante$ such that $\B\cap \psi(\R^n)= \emptyset$. Theorem \ref{parallelization theorem} allows us to pick representatives $\tau_0$ and $\tau_1$ that coincide outside $\B$.

Fix a family $F=\familleformes$ of propagating forms  of $(\ambientspace,\tau_0)$. The following lemma defines a family of propagating forms of $(\ambientspace,\tau_1)$.
\begin{lm}
There exists a family of $n$-forms $(\xi_i^n)_{1\leq i \leq 2k}$ on $\configM$ such that: \begin{itemize}
\item The family of forms $F'=(\propint_i, \propext'_i)_{1\leq i \leq 2k}$ obtained by setting $\propext'_i = \propext_i +\d\xi_i^n$ is a family of propagating forms of $(\ambientspace,\tau_1)$.
\item For any index $1\leq i\leq 2k$, the form ${\xi_i^n}_{|\partial\configM}$ is supported on $\unitaire \mathbb B\subset \unitaire \ambientspace\cong \partial_\Delta\configM$ (with the notations of Notation \ref{desc-bord}).

\end{itemize}
\end{lm}

\begin{proof}Fix the index $i\in \{1,\ldots, 2k\}$. 
First note that $G_{\tau_1}$ and $G_{\tau_0}$ coincide outside $\unitaire\B$. 
The form $G_{\tau_1}^*(\omega_{\propext_1}) - G_{\tau_0}^*(\omega_{\propext_1})$ defines a class in $H^{n+1}(\unitaire\B, \partial\unitaire\B)$ but $$H^{n+1}(\unitaire\B, \partial\unitaire\B)= H^{n+1}(\mathbb D^{n+2}\times\s^{n+1}, \s^{n+1}\times \s^{n+1})=H_{n+2}(\mathbb D^{n+2}\times\s^{n+1})= 0.$$ 
Therefore, there exists an $n$-form $(\xi_i^n)^0$ on $\unitaire\B$, which vanishes on $\partial U\B$, and is such that $(G_{\tau_1}-G_{\tau_0})^*(\omega_{\propext_1})= \d(\xi_i^n)^0$.
It remains to extend this form $(\xi_i^n)^0$. Since $(\xi_i^n)^0$ is zero on the boundary of $\unitaire\B$, we can extend it by $0$ to a form $(\xi_i^n)^1$ on $\partial\configM$. Then, pull this form $(\xi_i^n)^1$ back on a collar $N$ of $\partial\configM$, and multiply it by a smooth function, which sends $\partial \configM$ to $1$ and the other component of $\partial N$ to $0$. Eventually, extend this last form to $\configM$ by $0$ outside $N$. This gives an $n$-form $\xi_i^n$ as in the statement.
\end{proof}

Let $F_j$ denote the family of propagating forms with internal forms $(\propint_i)_{1\leq i \leq 2k}$ and external forms $(\propext'_1,\ldots,\propext'_{j}, \propext_{j+1},\ldots,\propext_{2k})$, so that $F_0= F$ and $F_{2k}=F'$.
For any $1\leq j \leq 2k$, set $\Delta_jZ_k(\psi)=\left( Z_k^{F_j}(\psi)-Z_k^{F_{j-1}}(\psi)\right)$, so that $$Z_k^{\tau_1}(\psi)-Z_k^{\tau_0}(\psi) = Z_k^{F'}(\psi)-Z_k^{F}(\psi) = \sum\limits_{1\leq j \leq 2k}\Delta_jZ_k(\psi).$$

Let us prove that $\Delta_1Z_k(\psi) = 0$.
Since $j=1$, with the notations of Section \ref{ds}, Lemma \ref{lm31} reads $\Delta_1Z_k(\psi)=\frac1{(2k)!}\sum\limits_{(\Gamma,\sigma)\in\graphesnum}\sum\limits_{S\in\faces}\delta_S\I[]$.
Since the internal forms are the same for $F_1$ and $F_2$, the numbered faces such that $e_0=\sigma^{-1}(1)$ is internal do not contribute.
According to Lemma \ref{vanishing theorem}, the only possibly contributing codimension $1$ numbered faces are:\begin{itemize}
\item numbered infinite faces $(\face, \sigma)$, such that $S$ contains at least one end of $e_0$, where $e_0$ is external,
\item numbered principal faces $(\face, \sigma)$, such that $S$ is composed of the ends of $e_0$, and such that $e_0$ is external with internal ends,
\item all the numbered anomalous faces $(\face[\sommets], \sigma)$ such that $e_0$ is external.
\end{itemize}

In these three cases, the map $p_{e_0}$ maps the face to $\partial \configM$. Infinite faces are sent to configurations with at least one of the two points at infinity. Anomalous and principal faces are sent to configurations where points of $S$ coincide, but since there exists at least one internal vertex, these points are necessarily on the knot, which does not meet $\B$. Then, $p_{e_0}$ maps the face outside the support of $\xi_1$, and the restriction of the form $\formeT$ to the face vanishes.

This proves that $\Delta_1Z_k(\psi)=0$. Similarly $\Delta_iZ_k(\psi)=0$ for any $2\leq i\leq 2k$. The independence of $Z_k$ of the parallelization follows.

\subsection{Invariance of $Z_k$ under ambient diffeomorphisms}
In this section, we prove the third assertion of Theorem \ref{th1}.

Fix a knot $\psi_0$ inside a parallelized asymptotic homology $\R^{n+2}$ denoted by $(\ambientspace,\tau)$, and fix a family $F=\familleformes$ of propagating forms of $(\ambientspace,\tau)$ .

Let $\phi\in\mathrm{Diffeo}(\ambientspace)$ be a diffeomorphism that fixes $\voisinageinfini$ pointwise, and let $\psi_1$ denote the knot $\phi\circ \psi_0$. In this section, for any $i\in \{0,1\}$ and for any edge $e$ of a BCR diagram $\Gamma$, $p_{e, i}$ denotes the map $p_e^{\psi_i}\colon C_\Gamma(\psi_i) \rightarrow \s^{n(e)}$ of Definition \ref{defpe}.

With these notations, $\phi$ induces a diffeomorphism $\Phi\colon C_\Gamma(\psi_0)\rightarrow C_\Gamma(\psi_1)$, and a diffeomorphism $\Phi_\propext \colon \configM\rightarrow\configM$. These diffeomorphisms extend the maps given by the formula $c\mapsto\phi\circ c$ on the interiors of these configuration spaces.
Then, \begin{eqnarray*} 
Z_k^\tau(\psi_1) & = &\frac1{(2k)!}\sum_{(\Gamma,\sigma)\in\graphesnum} \int_{C_\Gamma(\psi_1)} \omega^{F}(\Gamma,\sigma, \psi_1) \\
& = & \frac1{(2k)!}\sum_{(\Gamma,\sigma)\in\graphesnum} \int_{C_\Gamma(\psi_0)} \Phi^*(\omega^{F}(\Gamma,\sigma, \psi_1))\\&=&\frac1{(2k)!}
\sum_{(\Gamma,\sigma)\in\graphesnum} \int_{C_\Gamma(\psi_0)} \bigwedge\limits_{e\in E(\Gamma)}\Phi^*(\omega_e^{F}(\Gamma,\sigma, \psi_1))\\
&=&\frac1{(2k)!}\sum_{(\Gamma,\sigma)\in\graphesnum} \int_{C_\Gamma(\psi_0)} \bigwedge\limits_{e\in {\aretesinternes}}\Phi^*(p_{ e,1}^*({\propint}_{\sigma(e)}))\wedge\bigwedge\limits_{e\in {\aretesexternes}}\Phi^*(p_{ e,1}^*({\propext}_{\sigma(e)})).
\end{eqnarray*}

Note that, by construction, if $e\in\aretesinternes$, we have $p_{e,1}\circ \Phi = p_{e,0}$, and if $e\in\aretesexternes$, we have $p_{e,1}\circ \Phi = \Phi_\propext\circ p_{e,0}$. Define the family $F'=(\propint_i, \Phi_\propext^*(\propext_i))_{1\leq i \leq 2k}$ of propagating forms of $(\ambientspace, \tau')$, where $\tau'$ is the parallelization defined for any $x$ by the formula $\tau'_x = T_{\phi(x)} \phi^{-1}\circ \tau_{\phi(x)}$. The previous equation becomes \[Z_k^\tau(\psi_1) =\frac1{(2k)!} \sum_{(\Gamma,\sigma)\in\graphesnum} \int_{C_\Gamma(\psi_0)} \bigwedge\limits_{e\in E(\Gamma)}\omega_e^{F'}(\Gamma,\sigma, \psi_0) = Z_k^{\tau'}(\psi_0).\]

Since $Z_k$ does not depend on the parallelization, this reads $Z_k(\phi\circ\psi_0)=Z_k(\psi_0)$. This proves the third assertion of Theorem \ref{th1}.
\section{Proof of Lemma \ref{vanishing theorem}}\label{annex}
In this section, we analyse the variations of the integral $I^F(\Gamma,\sigma, \psi)$ under a change of the forms $(\alpha_1,\beta_1)$. These variations can be expressed as the sum of the integrals $\delta_S I(\Gamma,\sigma, \psi)$ over the numbered codimension $1$ faces $(\face, \sigma)$ of $\confignoeud$ described in Section \ref{ds}. Here, we study all these integrals in order to obtain the face cancellations precisely described in Lemma \ref{vanishing theorem}.

Recall that for any edge $e$, $n(e) =n-1$ if $e$ is internal, and $n(e)=n+1$ if $e$ is external.
In all this section, for any edge $e$ that has at least one end in $S$, let $\pes \colon\face\rightarrow \s^{n(e)} $ be the map $G\circ (p_e)_{|\face}$ if $e$ is internal, and the map $G_{\tau_{\sigma(e)}}\circ (p_e)_{|\face}$ if $e$ is external.

\subsection{Infinite faces}

In this section, we prove that the infinite face contributions vanish. As in Section \ref{ds}, $\sommets[\Gamma][*]= \sommets \sqcup \{*\}$. When $S\subsetneq \sommets[\Gamma][*]$, our proof is inspired from the proof of \cite[Lemma 6.5.9]{[Rossi]}. Let $S'=S\setminus\{*\}$, so that $S= S'\sqcup\{\infty\}\subset\sommets[\Gamma][*]$.

For infinite faces, the open face $\face$ is diffeomorphic to the product $\faceinfiniC\times\faceinfiniU$ where: \begin{itemize}
\item The manifold $\faceinfiniC$ is the set of configurations $c\colon \sommets\setminus S'\hookrightarrow \ambientspace$ such that $c(\sommetsinternes\setminus (S'\cap\sommetsinternes)) \subset \psi(\R^n)$.
\item The manifold $\faceinfiniU$ is the quotient set of maps $u_{S'}\colon S'\hookrightarrow \R^{n+2}\setminus\{0\}$ such that $u_{S'}(S'\cap\sommetsinternes)\subset \{0\}^2\times\R^n$ by dilations. 
\end{itemize}
Denote by $(c,[u_{S'}])$ a generic element of the infinite face $\face$. Such a configuration can be seen as the limit of the map \[c_\lambda \colon v\in \sommets \mapsto \left\{ \begin{array}{ll} c(v) & \text{if $v\not\in S'$,}\\ \frac{\lambda u_{S'}(v)}{||\lambda u_{S'}(v)||^2} &\text{if $v\in S'$,}\end{array}\right.\in \punct M\] when $\lambda$ approaches zero ($c_\lambda$ is well-defined for $\lambda$ sufficiently close to $0$).

\textbf{First case: $S= {\sommets[\Gamma][*]}$ and $(\propint'_1- \propint_1, \propext'_1-\propext_1)$ has the sphere factorization property.} 
In this case, $\face[{\sommets[\Gamma][*]}]$ is diffeomorphic to $\faceinfiniU[{\sommets}]$.
The following lemma directly implies that $\delta_{\sommets[\Gamma][*]}\I[]=0$.
\begin{lm}\label{infini-*}
The form $(\formeT)_{|\face[{\sommets[\Gamma][*]}]}$ is zero.
\end{lm}
\begin{proof}

Define the equivalence relation on $\faceinfiniU[{\sommets}]$ such that $[u_{\sommets}]$ and $[{u'}_{\sommets}]$ are equivalent if and only if there exist representatives $u_{\sommets}$ and $u'_{\sommets}$, and a vector $x\in \{0\}^2\times \R^n$, such that, for any $v \in \sommets $, \[ \frac{u'_{\sommets} (v)}{||u'_{\sommets}(v)||^2}=
\frac{ u_{\sommets} (v) }{ ||u_{\sommets} (v)||^2 } + x.\]

Let $\phi \colon \face[{\sommets[\Gamma][*]}]=\faceinfiniU[{\sommets}]\rightarrow Q$ denote the induced quotient map.
Then, for any $e\in\aretes$, the map $\pes[e][{\sommets[\Gamma][*]}]$ factors through $\phi$.
Since $(\propint'_1- \propint_1, \propext'_1-\propext_1)$ has the sphere factorization property, ${(\formeareteT)}_{|\face[{\sommets[\Gamma][*]}]}
 $ is the pullback of a form on the sphere by the map $\pes[e][{\sommets[\Gamma][*]}]$, for any edge $e$, including $\sigma^{-1}(1)$. Then, $(\formeareteT)_{|\face[{\sommets[\Gamma][*]}]}=\phi^*(\theta_{e,\sigma})$ where $\theta_{e,\sigma}$ is a form on $Q$, and $\formeT_{|\face[{\sommets[\Gamma][*]}]}$ therefore reads ${\formeT}_{|\face[{\sommets[\Gamma][*]}]} = \phi^*(\theta_\sigma)$ where $\theta_\sigma= \bigwedge\limits_{e\in\aretes} \theta_{e,\sigma}$.
Since $\deg(\theta_\sigma) = \deg(\formeT) = \dim(\partial\confignoeud)>\dim(Q)$, we have $\theta_\sigma=0$, so $(\formeT)_{|\face[{\sommets[\Gamma][*]}]}=0$. \end{proof}

\textbf{Second case: $S\subsetneq \sommets[\Gamma][*]$ and either $\sigma^{-1}(1)$ has no end in $S'$, or $\sigma^{-1}(1)$ has at least one end in $S'$ and $(\propint'_1- \propint_1, \propext'_1-\propext_1)$ has the sphere factorization property.}
In this case, let $\Esi$ (resp. $\Ese$) denote the set of internal (resp. external) edges with at least one end in $S'$, and set $\Es= \Esi \sqcup \Ese$.

\begin{lm}\label{inégalité-s*}
For any $S= S' \sqcup \{*\} \subsetneq \sommets[\Gamma][*]$, \[n.\Card(S'\cap\sommetsinternes) + (n+2).\Card(S'\cap\sommetsexternes) < (n-1).\Card(\Esi) +(n+1).\Card(\Ese).\]

\end{lm}

\begin{proof}
Split any edge $e$ of $\Gamma$ in two halves $e_-$ (the tail) and $e_+$ (the head), and let $v(e_\pm)$ denote the vertex adjacent to the half-edge $e_\pm$, as in the proof of Lemma \ref{dim=deg}.

Recall the definition of the integers $d(e_\pm)$ from the proof of Lemma \ref{dim=deg}:\begin{itemize}
\item If $e$ is external, $d(e_+)=1$ and $d(e_-)=n$ as in $\vcenter{\begin{tikzpicture} \draw [>=latex, dashed, ->] (0,0) -- (2,0); \draw [thick,densely dotted] (1,  -0.2) -- (1,0.2) ;
\draw (0.4, 0.2) node {$n$} (1.5,0.2) node{$1$} (0.4, -0.3) node {$e_-$} (1.6, -0.3) node {$e_+$};\end{tikzpicture}}$
\item If $e$ is internal, $d(e_+)=0$ and $d(e_-)= n-1$ as in  $\vcenter{\begin{tikzpicture} \draw [>=latex, ->] (0,0) -- (2,0); \draw [thick,densely dotted] (1,  -0.2) -- (1,0.2) ;
\draw (0.4, 0.2) node {$n-1$} (1.5,0.2) node{$0$} (0.4, -0.3) node {$e_-$} (1.6, -0.3) node {$e_+$};\end{tikzpicture}}$
\end{itemize}

As in the proof of Lemma \ref{dim=deg} and Figure \ref{somme-d}, these integers satisfy: \begin{itemize}
\item for any vertex $v\in\sommets$, $\sum\limits_{e_\pm, v(e_\pm) = v} d(e_\pm) = \left\{\begin{array}{lll} n & \text{if $v$ is internal,}\\ n+2 & \text{if $v$ is external.}\end{array}\right.$
\item for any edge $e\in\aretes$, $d(e_+) + d(e_-) = n(e)$.
\end{itemize}

Since $S\subsetneq \sommets[\Gamma][*]$, $S'\subsetneq \sommets$, and one of the following behaviors happens: \begin{itemize}
\item $S'$ contains only univalent vertices, and there exists an external edge going from $S'$ to $\sommets\setminus S'$.
\item $S'$ contains at least one vertex of the cycle of $\Gamma$, and there exists an edge going from $\sommets\setminus S'$ to $S'$.
\end{itemize}

In both cases, there exists a half-edge $e_\pm$ such that $v(e_\pm) \in S'$, $v(e_\mp) \not\in S'$, and $d(e_\mp)\neq 0$ ($n-1$ is indeed positive since $n\neq 1$). Therefore: \begin{eqnarray*}
n.\Card(S'\cap\sommetsinternes)& +  &(n+2).\Card(S'\cap\sommetsexternes) \\ &=& \sum\limits_{e_\pm, v(e_\pm)\in S'} d(e_\pm)\\
&<& \sum\limits_{e\in \Es} (d(e_+)+d(e_-))\\
&=& (n-1).\Card(\Esi) +(n+1).\Card(\Ese)\qedhere
\end{eqnarray*}

\end{proof}

Since the edges of $\Es$ have at least one vertex at infinity, their directions do not depend on the position of the points that are not at infinity, and we have the following.

\begin{lm}\label{vanishing-s*}
For any edge $e\in \Es$, the map $\pes$ factors through $\phi\colon\face\rightarrow \faceinfiniU$.
\end{lm}

From this lemma, and since either $(\propint'_1-\propint_1,\propext'_1-\propext_1)$ has the sphere factorization property, or $\sigma^{-1}(1)\not\in\Es$, we can write $\formeareteT = \phi^*(\theta_{e,\sigma})$ for any $e\in\Es$, where \[\deg(\theta_{e,\sigma}) =\left\{ \begin{array}{lll}
n(e) & \text{if $\sigma(e)\neq 1$,}\\
n(e)-1 & \text{if  $\sigma(e)= 1$.}
\end{array}\right.\] Then, $\deg\left(\bigwedge\limits_{e\in\Es} \theta_{e,\sigma} \right) \geq \left(\sum\limits_{e\in \Es} n(e)\right) -1$. 

Since $\dim(\faceinfiniU) = n.\Card(S'\cap\sommetsinternes) + (n+2).\Card(S'\cap\sommetsexternes) -1$, Lemma \ref{inégalité-s*} implies that $ \left(\sum\limits_{e\in \Es} n(e)\right) -1> \dim(\faceinfiniU)$, and $\deg\left(\bigwedge\limits_{e\in\Es} \theta_{e,\sigma} \right)  > \dim(\faceinfiniU)$. Therefore, $\bigwedge\limits_{e\in\Es} \theta_{e,\sigma}$ and $\formeT_{|\face}$ are zero. Then, $\delta_S \I[]=0$, as expected.

\subsection{Finite faces}

In this section, we study the contribution of the anomalous, hidden and principal faces. Our analysis resembles the analysis in \cite[Appendix A]{[Watanabe]}, but we have to take care of the fact that the propagating forms are not the same on each edge, and that they may not be pullbacks of forms on the spheres. 

\subsubsection{Description and restriction to the connected case}
Let $S$ be a subset of $\sommets$ of cardinality at least two, and let $\delta_S\Gamma$ be the graph obtained from $\Gamma$ by collapsing all the vertices in $S$ to a unique vertex $*_S$, internal if at least one of the vertices of $S$ is, external otherwise.
 Let $\inj$ denote the space of linear injections of $\R^n$ in $\R^{n+2}$. Define the following spaces: \begin{itemize}
\item The space $\facefinieC$ is composed of the injective maps $c\colon \sommets[\delta_S\Gamma] \hookrightarrow \ambientspace$ such that there exists $c_i \colon \sommetsinternes[\delta_S\Gamma]\hookrightarrow \R^n$ such that $c_{|\sommetsinternes[\delta_S\Gamma]}= \psi \circ c_i $.

\item If $S$ contains internal vertices, the space $\facefinieU$ is the quotient of the set $\facefinieU^0$ of pairs $(u, \iota)$ where $\iota$ is a linear injection of $\R^n$ inside $\R^{n+2}$ and $u$ is an injective map $u\colon S \hookrightarrow \R^{n+2}$ such that $u(S\cap\sommetsinternes) \subset \iota(\R^n)$, by dilations and by translations of $u$ along $\iota(\R^n)$.

If $S$ contains only external vertices, $\facefinieU$ is the quotient of the set of injective maps $u\colon S\hookrightarrow \R^{n+2}$ by dilations and translations along $\R^{n+2}$.
\end{itemize}

Then: \begin{itemize}
\item If $S$ contains an internal vertex, \[\face = \facefinieC\times_{\inj}\facefinieU = \{(c, [u, \iota ] ) \in \facefinieC\times\facefinieU\mid \iota= \tau_{c(*_S)}^{-1}\circ T\psi_{c_i(*_S)}\} .\]
\item If $S$ contains only external vertices, $\face=\facefinieC\times\facefinieU$. 
\end{itemize}
Keep the notation $\face = \facefinieC\times_{\inj}\facefinieU$ in both cases. In the following, an element of $\face$ will be represented by $(c, [u])$ since $\iota$ can be deduced from $c$ when $S\cap \sommetsinternes \neq \emptyset$.

\begin{lm}\label{connected}
Let $\Gamma_{S}$ be the graph defined in Lemma \ref{vanishing theorem}.
If $\Gamma_{S}$ is not connected, then $\delta_S \I[]=0$.

\end{lm}
\begin{proof}

Suppose that $\Gamma_{S}$ is not connected. Then, there exists a partition $S=S_1\sqcup S_2$ such that no edge connects $S_1$ and $S_2$, and where $S_1$ and $S_2$ are non-empty sets. 

Suppose that $S$ contains at least one internal vertex, and set $\iota(c) =\tau_{c(*_S)}^{-1}\circ T\psi_{c_i(*_S)}$.
Define the equivalence relation on $\face$ such that $(c, [u])$ and $(c', [u'])$ are equivalent if and only if $c=c'$ and there exist representatives $u$ and $u'$ and a vector $x\in \R^n$ such that, for any $v\in S$, \[u' (v) = \begin{cases} u(v) + \iota(c)(x)&\text{if $v\in S_1$,}\\u(v) & \text{if $v\in S_2$.}\end{cases}\]
Let $\phi \colon \face\rightarrow Q$ denote the quotient map.
With these notations, for any edge $e$, the map $p_e$ factors through $\phi$. We conclude as in the proof of Lemma \ref{infini-*}, since $\deg(\formeT)>\dim( Q)$.

If $S$ contains only external vertices, we proceed similarly with the equivalence relation such that $(c,[u])$ and $(c', [u'])$ are equivalent if and only if $c=c'$ and there exist representatives $u$ and $u'$ and a vector $x\in\R^{n+2}$ such that, for any $v\in S$, \[u' (v) = \begin{cases} u(v) + x&\text{if $v\in S_1$,}\\u(v) & \text{if $v\in S_2$.}\end{cases}\qedhere\]
\end{proof}

\subsubsection{Anomalous face}
In this section, according to the hypotheses of Lemma \ref{vanishing theorem}, assume that $(\propint'_1-\propint_1,\propext'_1-\propext_1)$ has the sphere factorization property.

\begin{lm}\label{vanishing-anomalous}
There exists an orientation-reversing diffeomorphism of the anomalous face $T\colon \face[\sommets]\rightarrow \face[\sommets]$, such that, for any edge $e\in\aretes$,
 $$\pes\circ T = (-\mathrm{Id}_{\s^{n(e)}}) \circ \pes,$$ where $-\mathrm{Id}_{\s^{n(e)}}$ is the antipodal map.

\end{lm}
\begin{proof}
Here, since $\delta_{\sommets}\Gamma$ is a graph with only one internal vertex $*_{\sommets}$, the face is diffeomorphic to $\psi(\R^n)\times_{\inj}\facefinieU[\sommets]$.
Choose an internal vertex $v$ of $\Gamma$. For $[u]\in \facefinieU[\sommets]$, define $[u']\in \facefinieU[\sommets]$ as the class of the map $u'$ such that, for any vertex $w$, $u'(w) = 2u(v)-u(w)$.
Then the map $T\colon (c,[u])\in \face[\sommets]\mapsto (c, [u'])\in\face[\sommets]$ is a diffeomorphism. The sign of its Jacobian determinant is $(-1)^{(2k-1)(n+2)}= -1$, since $n$ is odd.
It is easy to check that $\pes\circ T = (-\mathrm{Id}_{\s^{n(e)}}) \circ \pes$ for any edge $e$.
\end{proof}
Since $(\propint'_1-\propint_1,\propext'_1-\propext_1)$ has the sphere factorization property, $(\formeareteT)_{|\face[\sommets]}$ reads $\pes ^*(\theta_e)$, where, for any edge $e$, $\theta_e$ is an antisymmetric form on the sphere. Then, Lemma \ref{vanishing-anomalous} yields $T^*((\formeareteT)_{|\face[\sommets]}) = \pes^*( (-\mathrm{Id}_{\s^{n(e)}})^*(\theta_e)) = - (\formeareteT)_{|\face[\sommets]}$. Then,  \begin{eqnarray*}\delta_{\sommets}\I[] &= &\int_{\face[\sommets]} \formeT_{|\face[\sommets]}\\& =& - \int_{\face[\sommets]} T^*\left(\formeT_{|\face[\sommets]}\right) \\&=& - \int_{\face[\sommets]} (-1)^{\Card(\aretes)} \formeT_{|\face[\sommets]} \\&=& - \delta_{\sommets}\I[]  \text{\ \ \ \ since $\Card(\aretes)=2k$}.\end{eqnarray*}

Eventually, this implies that $\delta_{\sommets}\I[]=0$.
\subsubsection{Hidden faces}

\begin{lm}\label{vanishing-caché-univalent}
Let $\hfaces$ be the set of hidden faces such that at least one of the following properties hold:\begin{itemize} \item $\Gamma_S$ is non connected.
\item $\Gamma_S$ has at least three vertices, $\Gamma_S$ has a univalent vertex $v_0$, and, if this vertex is internal, then its only adjacent edge $e_0$ in $\Gamma_S$ is internal (as in Figure \ref{cachees1}).\end{itemize}
For any face $\face$ in $\hfaces$, $\delta_S\I[]=0$.
\end{lm}

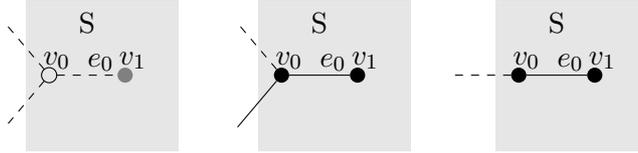
\begin{figure}[H]
\centering
\begin{tikzpicture} 
\fill [color=gray!20] (1.7, 1) rectangle (-0.3, -1) ;
\draw (0,0) circle (0.1) ;
\draw [dashed] (0:0.1) -- (0:0.9) ;
\fill [gray] (0: 1) circle (0.1);
\draw[dashed] (130:0.1)--(130:0.9);
\draw[dashed] (230:0.1)--(230:0.9);
\draw (0.5, 0.7) node {S};
\draw (0.1, 0.2) node {$v_0$} ++(1, 0) node {$v_1$};
\draw (14: 0.7) node {$e_0$};
\end{tikzpicture}
\ \ \ \ \ 
\begin{tikzpicture} 
\fill [color=gray!20] (1.7, 1) rectangle (-0.3, -1) ;
\fill (0,0) circle (0.1) ;
\draw  (0:0.1) -- (0:0.9) ;
\fill (0: 1) circle (0.1);
\draw[dashed] (130:0.1)--(130:0.9);
\draw (230:0.1)--(230:0.9);
\draw (0.5, 0.7) node {S};
\draw (0.1, 0.2) node {$v_0$} ++(1, 0) node {$v_1$};
\draw (14: 0.7) node {$e_0$};
\end{tikzpicture}
\ \ \ 
\begin{tikzpicture} 
\fill [color=gray!20] (1.7, 1) rectangle (-0.3, -1) ;
\fill (0,0) circle (0.1) ;
\draw  (0:0.1) -- (0:0.9) ;
\draw (14: 0.7 ) node {$e_0$};
\fill (0: 1) circle (0.1);
\draw[dashed] (180:0.1)--(180:0.9);
\draw (0.5, 0.7) node {S};
\draw (0.1, 0.2) node {$v_0$} ++(1, 0) node {$v_1$};
\end{tikzpicture}
\caption{The second property in the definition of $\hfaces$}\label{cachees1}
\end{figure}

\begin{proof}
If $\Gamma_S$ is not connected, this is Lemma \ref{connected}.

If $\Gamma_S$ is connected, we have a univalent vertex $v_0$ as in Figure \ref{cachees1}.
There is a natural map $\face \rightarrow \facefinieC\times_{\inj}\times \facefinieU[S\setminus\{v_0\}]$. Let $$\phi\colon \face \rightarrow Q=(\facefinieC\times_{\inj}\times \facefinieU[S\setminus\{v_0\}])\times \s^{n(e_0)} $$ denote the product of this map and the Gauss map $G_{e_0,S}$.
As in the similar lemmas of the previous subsection, for any edge whose ends are both in $S$, $\pes$ factors through $\phi$, and for any other edge, $p_e$ factors through $\phi$. Then, all the forms $\formeareteT$ are pullbacks of forms on $Q$ by $\phi$, and $\formeT$ also is. 
The hypotheses of the lemma imply that $\dim(Q) < \dim(\face)$, so $\delta_S\I[]=0$.\qedhere

\end{proof}

\begin{lm}\label{vanishing-h-2}
Let $\HfaceA$ denote the set of hidden faces that are not in $\hfaces$ and such that $\Gamma_S$ contains a bivalent vertex $v$, which is trivalent in $\Gamma$, and which has one incoming and one outgoing edge in $\Gamma_S$, which are both internal if $v$ is internal.

For any face $\face$ in $\HfaceA$, $\delta_SI(\Gamma,\sigma, \psi) = -\delta_SI(\Gamma,\sigma\circ\rho, \psi)$, where $\rho$ denotes the transposition of $e$ and $f$.\end{lm}

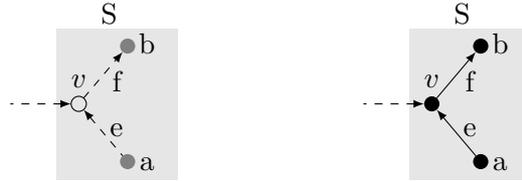
\begin{figure}[H]
 \centering
\begin{tikzpicture} 
\fill [color=gray!20] (1.3, 1) rectangle (-0.3, -1) ;
\draw (0,0) circle (0.1) ;
\draw [dashed,<-, >= latex] (180:0.1) -- (180:0.9) ;
\draw[dashed,->, >= latex] (50:0.1)--(50:0.9);
\draw[dashed,<-, >= latex] (-50:0.1)--(-50:0.9);
\fill[gray]  (50: 1) circle (0.1);
\fill[gray]  (-50: 1) circle (0.1);
\draw (0.4, 1.2) node {S};
\draw (0, 0.3) node{$v$};
\draw (0.5, -0.35) node {e};
\draw (0.5, 0.3) node {f};
\draw (0.9, -0.8) node {a};
\draw (0.9, 0.8) node {b};
\end{tikzpicture}
 \ \ \ \ \ \ \ \ \ \ \ \ \ \ \ \ \ \ 
\begin{tikzpicture} 
\fill [color=gray!20] (1.3, 1) rectangle (-0.3, -1) ;
\fill (0,0) circle (0.1) ;
\draw [dashed,<-, >= latex] (180:0.1) -- (180:0.9) ;
\draw [->, >= latex](50:0.1)--(50:0.9);
\draw [<-, >= latex] (-50:0.1)--(-50:0.9);
\fill  (50: 1) circle (0.1);
\fill  (-50: 1) circle (0.1);
\draw (0.4, 1.2) node {S};
\draw (0, 0.3) node{$v$};
\draw (0.5, -0.35) node {e};
\draw (0.5, 0.3) node {f};
\draw (0.9, -0.8) node {a};
\draw (0.9, 0.8) node {b};
\end{tikzpicture}\caption{Hypotheses of Lemma \ref{vanishing-h-2}.}
\label{Fig8}
\end{figure}

\begin{proof}
Let $e$ and $f$ denote the incoming and the outgoing edge of $v$ in $\Gamma_S$, and let $a$ and $b$ denote the other ends of $e$ and $f$, as in Figure \ref{Fig8}.\footnote{Note that we may have $a=b$.}

Let $T$ be the orientation-reversing diffeomorphism of $\face$ defined as follows: if $(c,[u])\in\face$, let $u'\colon S\rightarrow \R^{n+2}$ be the map such that, for any $w\in S$, \[u'(w) = \left\{\begin{array}{ll}
u(w) & \text{if $w\neq v$,} \\
u(b)+u(a)-u(v) &\text{if $w=v$,}
\end{array}\right.\] and set $T(c,[u]) = (c, [u'])$. 

For any $g\in\aretes$, $p_g\circ T = p_{\rho(g)}$, so $T^*((\formeT)_{|\face[S]})=(\formeT[\Gamma][\sigma\circ\rho])_{|\face[S]}$, and thus $\delta_SI(\Gamma,\sigma,\psi) =- \delta_SI(\Gamma,\sigma\circ\rho, \psi)$.\end{proof}

\begin{lm}
Let $\HfaceB$ be the set of hidden faces that are neither in $\hfaces$ nor in $\HfaceA$.
For any face $\face$ in $\HfaceB$, we have the following properties : \begin{itemize}
\item If $S$ contains the head of an external edge, then it contains its tail.
\item If $S$ contains a univalent vertex, then it contains its only adjacent vertex.
\end{itemize}
In particular, $S$ necessarily contains at least one vertex of the cycle, but cannot contain all of them, since it would imply $S=\sommets$.
\end{lm}
\begin{proof}
Let $\face$ be a face in $\HfaceB$.
The second point directly follows from the connectedness of $\Gamma_S$. Let us prove the first point. 
Let $e=(v,w)$ be an external edge with $w$ in $S$. \begin{itemize}
\item If $e$ is a leg, we have three possible cases:
\begin{itemize}
\item If the two neighbors of $w$ in the cycle are in $S$, then $v$ is in $S$. Indeed, otherwise $S$ would contain a piece as in Figure \ref{Fig8}.
\item If $S$ contains one of the neighbors of $w$ in the cycle, then $v\in S$. Indeed, otherwise $S$ would contain a piece such as in the two first pieces of Figure \ref{cachees1}.
\item  If none of the neighbors of $w$ in the cycle are in $S$, then $\Gamma_S$ is not connected, which is impossible.
\end{itemize}
\item Otherwise, $e$ is an external edge of the cycle, and we have two possible cases:
\begin{itemize}
\item If $w$ is bivalent, then it has two neighbors $v$ and $w'$. \begin{itemize}
\item If $w'\in S$, then $v\in S$: otherwise, we would have a piece as the third one of Figure \ref{cachees1}. 
\item If $w'\not \in S$, then $v\in S$ because of the connectedness of $\Gamma_S$.
\end{itemize}
\item Otherwise $w$ is trivalent, and external.\begin{itemize}
\item If its two other neighbors than $v$ are in $S$, $v$ is in $S$: otherwise, we would have a piece as the first one of Figure \ref{Fig8}. 
\item If $S$ contains one of these two neighbors, $v$ is in $S$: otherwise, we would have a piece as the first one of Figure \ref{cachees1}. 
 \item Eventually, if none of these two neighbors are in $S$, $v$ is in $S$ because of the connectedness of $\Gamma_S$.\qedhere
\end{itemize}
\end{itemize}
\end{itemize}
\end{proof}
\begin{lm}\label{vanishing-h-3}
Suppose that $\face$ is a face of $\HfaceB$ and that $S$ contains at least one external vertex.

Then, there exists a transposition $\rho$ of two edges such that $$\delta_SI(\Gamma,\sigma,\psi) = -\delta_SI(\Gamma,\sigma\circ\rho, \psi).$$
\end{lm}

\begin{proof}

Choose an external vertex of $S$ and follow the cycle backwards until getting out of $S$. Let $d$ be the last met vertex in $S$. It follows from the previous lemma that $d$ is an internal vertex, with an incoming internal edge coming from $\sommets\setminus S$. From $d$, move forward along the cycle, and let $v_0$ be the first seen external vertex. There are two incoming edges in $v_0$, one coming from a univalent vertex $b$, denoted by $f$, and one coming from a bivalent vertex $a$, denoted by $e$ (we may have $a= d$). Let $S_0$ be the set of vertices of the cycle between $d$ and $a$, with their univalent adjacent vertices. The obtained situation is like in Figure \ref{Fig9}. 

\begin{figure}[H]
\centering
\begin{tikzpicture}\fill [color=gray!20] (-3, 1.5) rectangle (4, -1.5) ;
\fill (-1, 1) circle (0.1) (0,0) circle (0.1) (1,0) circle (0.1) (1,1) circle (0.1) (3,0) circle (0.1) (3,1) circle (0.1) ;
\draw (-1, 0) circle (0.1);
\draw[ ->, >= latex, dashed] (-0.1, 0) -- (-0.9, 0);
\draw[ ->, >= latex, dashed] (-1.1, 0) -- (-1.9, 0);
\draw[ ->, >= latex, dashed] (-1, 0.9) -- (-1, 0.1); 
\draw[->, >= latex]  (4.5, 0)-- (3.1, 0);
\draw[->, >= latex]  (1.5,0)--(1.1,0); 
\draw[->, >= latex]  (0.9, 0)--(0.1,0); 
\draw[->, >= latex, dashed]  (3, 0.9) -- (3, 0.1);
\draw[->, >= latex, dashed]  (1, 0.9) -- (1, 0.1);
\draw[dotted] (1.5, 0)--(2.5,0);
\draw (2.9,0)--(2.5,0);
\draw[dotted] (1.7, 0) circle (2);
\draw (-1, -0.3) node {$v_0$} (-0.8, 0.45) node {$f$} (-0.8, 1.1) node {$b$};
\draw (-0.5, -0.3) node {$e$} (0.2, 0.2) node {$a$};
\draw (3.3, 0.2) node {$d$};
\draw (3.75,-0.75) node {$S_0$};
\draw (-2.8, -1.3) node {$S$};
\end{tikzpicture}
\caption{Notations for the proof of Lemma \ref{vanishing-h-3}.}\label{Fig9}
\end{figure}
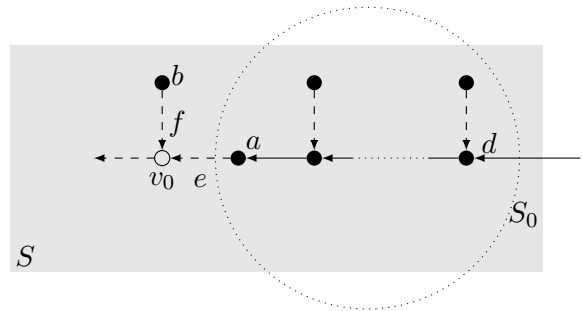
Then, if $(c,[u])\in\face$, let $u'$ denote the map such that, for any $w\in S$, \[u'(w) = \left\{
\begin{array}{lll}  u(a) & \text{if $w=b$,}\\ u(w) + u(b)-u(a) & \text{if $w\in S_0$,}\\ u(w) & \text{otherwise,}
\end{array}\right.\]and define an orientation-reversing diffeomorphism $T\colon \face \rightarrow \face$ by the formula $T(c,[u]) = (c,[u'])$. 
Thus, if $\rho$ denotes the transposition of $e$ and $f$, we have $p_g\circ T = p_{\rho(g)}$ for any $g$, and we conclude as in Lemma \ref{vanishing-h-2}.
\end{proof}

\begin{lm}\label{vanishing-h-4}
Suppose that the face $\face$ is in $\HfaceB$, and that $\Gamma_S$ contains no external vertex. Then, it contains at least one of the following pieces: \begin{itemize}
\item Two non adjacent external edges with their sources $a$ and $b$ univalent in $\Gamma_S$ (not necessarily in $\Gamma$).
\item A sequence of one external, one internal and one external edge, as in the second part of Figure \ref{Fig10}.
\item A trivalent internal vertex with all its neighbors.
\end{itemize}

In all of the above cases, we have a transposition of two edges $\rho$ such that $$\delta_SI(\Gamma,\sigma,\psi) =- \delta_SI(\Gamma,\sigma\circ\rho, \psi).$$

\end{lm}
\begin{proof}
Figure \ref{Fig10} describes the three possible cases of the lemma, and we use its notations.

 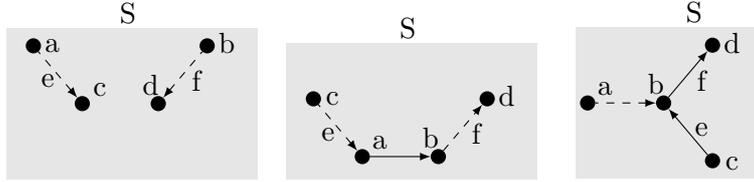
\begin{figure}[H]
 \centering
\begin{tikzpicture} 
\fill [color=gray!20] (1.3, 1) rectangle (-2, -1) ;
\fill (0,0) circle (0.1) ; 
\draw[dashed,<-, >= latex] (50:0.1)--(50:0.9);
\fill (180:1) circle (0.1);
\draw (180:1) ++(40: 0.3) node {c};
\draw[dashed, <-, >=latex] (180:1) ++(130:0.1) -- ++(130:0.9);
\fill (180:1) ++(130:1) circle (0.1);
\fill  (50: 1) circle (0.1);
\draw (-0.4, 1.2) node {S};
\draw (-0.1, 0.25) node{d};
\draw (180: 1.45) ++ (0,0.3) node {e};
\draw (0.5, 0.3) node {f};
\draw (180:0.75) ++ (130:1) node {a};
\draw (0.9, 0.8) node {b};
\end{tikzpicture}
\ \
\begin{tikzpicture} 
\fill [color=gray!20] (1.3, 1.5) rectangle (-2, -0.3) ;
\fill (0,0) circle (0.1) ;
\draw [<-, >= latex] (180:0.1) -- (180:0.9) ;
\draw[dashed,->, >= latex] (50:0.1)--(50:0.9);
\fill (180:1) circle (0.1);
\draw (180:1) ++(40: 0.3) node {a};
\draw[dashed, <-, >=latex] (180:1) ++(130:0.1) -- ++(130:0.9);
\fill (180:1) ++(130:1) circle (0.1);
\fill  (50: 1) circle (0.1);
\draw (-0.4, 1.7) node {S};
\draw (-0.1, 0.25) node{b};
\draw (180: 1.45) ++ (0,0.3) node {e};
\draw (0.5, 0.3) node {f};
\draw (180:0.75) ++ (130:1) node {c};
\draw (0.9, 0.8) node {d};
\end{tikzpicture}
\ \ \
\begin{tikzpicture} 
\fill [color=gray!20] (1.3, 1) rectangle (-1.15, -1) ;
\fill (0,0) circle (0.1) ;
\draw [dashed,<-, >= latex] (180:0.1) -- (180:0.9) ;
\draw [->, >= latex](50:0.1)--(50:0.9);
\draw [<-, >= latex] (-50:0.1)--(-50:0.9);
\fill (180:1) circle (0.1);
\fill  (50: 1) circle (0.1);
\draw (180:1) ++(40: 0.3) node {a};
\fill  (-50: 1) circle (0.1);
\draw (0.4, 1.2) node {S};
\draw (-0.1, 0.25) node{b};
\draw (0.5, -0.35) node {e};
\draw (0.5, 0.3) node {f};
\draw (0.9, -0.8) node {c};
\draw (0.9, 0.8) node {d};
\end{tikzpicture}
\caption{The three behaviors of Lemma \ref{vanishing-h-4}}
\label{Fig10}
\end{figure}

Let $\rho$ be the transposition swapping $e$ and $f$. The involution $T$ is defined by $T(c,[u]) = (c,[u'])$ where, for any vertex $w\in S$ :\begin{itemize}
\item In the first case, \[u'(w) =
 \begin{cases}
 u(c) + u(b)-u(d) & \text{if $w=a$,}\\
 u(d)+u(a)-u(c) & \text{if $w=b$,}\\
 u(w) & \text{otherwise.}
\end{cases} \]

\item In the second case, \[u'(w) =
 \begin{cases}
 u(c) + u(d)-u(b) & \text{if $w=a$,}\\
 u(c)+u(d)-u(a) & \text{if $w=b$,}\\
 u(w) & \text{otherwise.}
\end{cases} \]

\item In the third case, \[u'(w) =
 \begin{cases}
 u(c) + u(d)+u(a)-2u(b) & \text{if $w=a$,}\\
 u(c)+u(d)-u(b) & \text{if $w=b$,}\\
 u(w) & \text{otherwise.}
\end{cases} \]

\end{itemize}
As in the previous proofs, $T$ reverses the orientation, and $p_g\circ T = p_{\rho(g)}$ for any edge $g$ of $\Gamma_S$.
\end{proof}

For a given $\Gamma$, set $\Hfaces=\HfaceA\cup\HfaceB$. For any $\face$ in $\Hfaces$, define the involution $\sigma \mapsto \sigma^*$ of Lemma \ref{vanishing theorem} as follows: put a total order on non-ordered pairs of $\{1,\ldots, 2k\}$. If there is a $v$ as in Lemma \ref{vanishing-h-2}, choose the one minimizing $\{\sigma(e), \sigma(f)\}$, and set $\sigma^* = \sigma \circ \rho$ as in the lemma. Otherwise, if there is an external vertex in $S$, choose one such that the outgoing edge is of minimal $\sigma$, and proceed as in Lemma \ref{vanishing-h-3}, setting $\sigma^*= \sigma\circ \rho$. Otherwise, if there are two edges $e$ and $f$ as in the first case of Lemma \ref{vanishing-h-4}, choose the pair that minimizes $\{\sigma(e), \sigma(f)\}$. If not, and if there is a piece as in the second case, choose the one with minimal $\{\sigma(e), \sigma(f)\}$, and otherwise, there is a piece as in the third case: take the one of minimal $\{\sigma(e), \sigma(f)\}$. In these last three cases, set $\sigma^* = \sigma \circ \rho$ where $\rho$ is the transposition of $e$ and $f$.

\subsubsection{Principal faces}\label{section-principal}

It only remains to study the principal faces, which are the faces such that the ends of an edge $e$ collide, and where this edge is the only edge between its two ends. Then, $\face[e]\cong\facepC\times \s^{m(e)}$, where \[m(e)= \left\{\begin{array}{ll} n-1 &\text{if the ends of $e$ are both internal,} \\ n+1 &\text{otherwise.}\end{array}\right.\] Choose this diffeomorphism in such a way that the Gauss map reads as the second projection map $\mathrm{pr}_2$ in the product, and orient $\facepC$ in such a way that this diffeomorphism preserves the orientation.

\begin{lm}\label{vanishing-p}
If $\sigma(e)=1$, and if $e$ is either an internal edge or an external edge with at least one external end, then $\delta_e\I[]=0$.
\end{lm}\begin{proof}
For any edge $f\neq e$, the map $p_f$ factors through $\mathrm{pr}_1\colon\face[e]\rightarrow \facepC$. Then $\formeT_{|\face[e]} = \formeareteT[e][\sigma]\wedge \bigwedge\limits_{f\in\aretes, f\neq e} \mathrm{pr}_1^*(\theta_{f,\sigma})$, where $\theta_{f,\sigma}$ are forms on $\facepC$. 

But we have \[\deg\left(\bigwedge\limits_{f\in\aretes,f\neq e} \theta_{f,\sigma}\right) = \dim(\face) - (n(e)-1) = \dim(\facepC) +1,\] since $m(e)=n(e)$ under the hypotheses of the lemma.

Then, $\deg\left(\bigwedge\limits_{f\in\aretes,f\neq e} \theta_{f,\sigma}\right)  > \dim(\facepC)$, and $\delta_e\I[]=0$.
\end{proof}

Lemmas \ref{p2} to \ref{p5} are proved after the statement of Lemma \ref{p5}.
\begin{lm}\label{p2}
Suppose that $\Gamma$ looks as in Figure \ref{p21} around $e$.

\begin{figure}[H]
\centering
\begin{tikzpicture} 
\fill (0,0) circle (0.1) ;
\draw [<-, >= latex] (180:0.1) -- (180:0.9) ;
\draw [->, >=latex]  (0:0.1) -- (0:0.9) ;
\draw [<-,>=latex, dashed] (90:0.1)--(90:0.9);
\draw (50: 0.3) node {w};
\draw (0:0.36) ++(-65: 0.3) node {g};
\draw (180:0.66) ++(-60: 0.3) node {f};
\draw (90:0.99) ++(150: 0.3) node {v};
\draw (90:0.41) ++(150: 0.3) node {e};
\fill (90:1) circle (0.1);
\end{tikzpicture}
\caption{}
\label{p21}
\end{figure}
\noindent Let $\Gamma^*$ denote the BCR diagram where this part of $\Gamma$ is replaced as in Figure \ref{p22}.

\begin{figure}[H]
\centering
\begin{tikzpicture} 
\fill(0,0) circle (0.1) ;
\draw [<-, dashed,>= latex] (180:0.1) -- (180:0.9) ;
\draw [<-, >=latex]  (180:1.1) -- (180:1.9) ;
\draw [->,>=latex]  (0:0.1) -- (0:0.9) ;
\fill(180:1) circle (0.1);
\draw (180:1) ++(50: 0.3) node {v};
\draw (50: 0.3) node {w};
\draw (180:0.66) ++(-50: 0.3) node {$e^*$};
\draw (180:1.66) ++(-50: 0.3) node {f};
\draw (0:0.36) ++(-60: 0.3) node {g};
\end{tikzpicture}
\caption{}
\label{p22}
\end{figure}
If $(\propint'_1-\propint_1,\propext'_1-\propext_1)$ has the sphere factorization property, or if $\sigma(e)\neq 1$, then $\delta_eI(\Gamma,\sigma,\psi) =-\delta_{e^*} I(\Gamma^*,\sigma^*,\psi)$, where $\sigma^*$ is naturally induced by $\sigma$.

\end{lm}

In all the remaining cases, $\sigma(e)\neq1$, since the numbered faces with $\sigma(e)=1$ are all studied by Lemmas \ref{vanishing-p} and \ref{p2}.
\begin{lm}\label{p1}
Suppose that $\Gamma$ looks as in Figure \ref{p11} around $e$. 

\begin{figure}[H]
\centering
\begin{tikzpicture} 
\fill (0,0) circle (0.1) ;
\draw [<-, >= latex] (180:0.1) -- (180:0.9) ;
\draw [<-, >=latex, dashed]  (180:1.1) -- (180:1.9) ;
\draw [->, >=latex, dashed]  (0:0.1) -- (0:0.9) ;
\fill (180:1) circle (0.1);
\draw (180:1) ++(50: 0.3) node {v};
\draw (50: 0.3) node {w};
\draw (180:0.66) ++(-50: 0.3) node {e};
\draw (180:1.69) ++(-50: 0.3) node {f};
\draw (0:0.46) ++(-60: 0.3) node {g};
\end{tikzpicture} 
\caption{}
\label{p11}
\end{figure}

\noindent Let $\Gamma^*$ denote the BCR diagram where this part of $\Gamma$ is replaced as in Figure \ref{p12}.

\begin{figure}[H]
\centering
\begin{tikzpicture} 
\draw (0,0) circle (0.1) ;
\draw [<-, >= latex, dashed] (90:0.1) -- (90:0.9) ;
\draw [<-, >=latex, dashed]  (180:0.1) -- (180:0.9) ;
\draw [->, >=latex, dashed]  (0:0.1) -- (0:0.9) ;
\fill (90:1) circle (0.1);
\draw (90:1) ++(50: 0.3) node {v};
\draw (50: 0.3) node {w};
\draw (90:0.41) ++(150: 0.3) node {$e^*$};
\draw (0:0.46) ++(-60: 0.3) node {g};
\draw (180:0.66) ++(-60: 0.3) node {f};
\end{tikzpicture}
\caption{}
\label{p12}
\end{figure}

If $\sigma(e)\neq 1$, then $\delta_e I(\Gamma,\sigma,\psi) =-\delta_{e^*} I(\Gamma^*,\sigma^*,\psi)$, where $\sigma^*$ is naturally induced by $\sigma$.

\end{lm}

\begin{lm}\label{p3}
Suppose that $\Gamma$ looks as in Figure \ref{p31} around $e$.  

\begin{figure}[H]
\centering
\begin{tikzpicture} 
\fill (0,0) circle (0.1) ;
\draw [<-, >= latex] (180:0.1) -- (180:0.9) ;
\draw [<-, >=latex, dashed]  (180:1.1) -- (180:1.9) ;
\draw [->, >=latex ]  (0:0.1) -- (0:0.9) ;
\draw [<-,>=latex, dashed] (90:0.1)--(90:0.9);
\fill (180:1) circle (0.1);
\draw (180:1) ++(50: 0.3) node {v};
\draw (50: 0.3) node {w};
\draw (180:0.66) ++(-50: 0.3) node {e};
\draw (180:1.69) ++(-50: 0.3) node {f};
\draw (90:0.41) ++(150: 0.3) node {g};
\draw (90:0.99) ++(150: 0.3) node {x};
\fill (90:1) circle (0.1);
\end{tikzpicture}
\caption{}
\label{p31}
\end{figure}

\noindent Let $\Gamma^*$ denote the BCR diagram where this part of $\Gamma$ is replaced as in Figure \ref{p32}.

\begin{figure}[H]
\centering
\begin{tikzpicture} 
\fill (0,0) circle (0.1) ;
\draw [<-, >= latex, dashed] (180:0.1) -- (180:0.9) ;
\draw [<-, >=latex, dashed]  (180:1.1) -- (180:1.9) ;
\draw [->, >=latex ]  (0:0.1) -- (0:0.9) ;
\draw [<-,>=latex, dashed] (180:1) ++(90:0.1)--++(90:0.9);
\draw (180:1) circle (0.1);
\draw (180:1) ++(50: 0.3) node {v};
\draw (50: 0.3) node {w};
\draw (180:0.66) ++(-50: 0.3) node {$e^*$};
\draw (180:1.69) ++(-50: 0.3) node {f};
\draw (180:1) ++(90:0.41) ++(150: 0.3) node {g};
\draw (180:1) ++(90:0.99) ++(150: 0.3) node {x};
\fill (180:1) ++ (90:1) circle (0.1);
\end{tikzpicture}
\caption{}
\label{p32}
\end{figure}

If $\sigma(e)\neq 1$, then $\delta_eI(\Gamma,\sigma,\psi) =-\delta_{e^*} I(\Gamma^*,\sigma^*,\psi)$, where $\sigma^*$ is naturally induced by $\sigma$.

\end{lm}

\begin{lm}\label{p4}
Suppose that $\Gamma$ looks as in Figure \ref{p41} around $e$. 

\begin{figure}[H]
\centering
\begin{tikzpicture} 
\fill (0,0) circle (0.1) ;
\draw [<-, >= latex] (180:0.1) -- (180:0.9) ;
\draw [<-, >=latex]  (180:1.1) -- (180:1.9) ;
\draw [->, >=latex, dashed]  (0:0.1) -- (0:0.9) ;
\draw [<-,>=latex, dashed] (180:1) ++(90:0.1)--++(90:0.9);
\fill (180:1) circle (0.1);
\draw (180:1) ++(50: 0.3) node {v};
\draw (50: 0.3) node {w};
\draw (180:0.66) ++(-50: 0.3) node {e};
\draw (0:0.19) ++(-50: 0.3) node {f};
\draw (180:1) ++(90:0.41) ++(150: 0.3) node {g};
\draw (180:1) ++(90:0.99) ++(150: 0.3) node {x};
\fill (180:1) ++ (90:1) circle (0.1);
\end{tikzpicture}
\caption{}
\label{p41}
\end{figure}

\noindent Let $\Gamma^*$ denote the BCR diagram where this part of $\Gamma$ is replaced as in Figure \ref{p42}. 

\begin{figure}[H]
\centering
\begin{tikzpicture} 
\draw (0,0) circle (0.1) ;
\draw [<-, >= latex, dashed] (180:0.1) -- (180:0.9) ;
\draw [<-, >=latex]  (180:1.1) -- (180:1.9) ;
\draw [->, >=latex, dashed ]  (0:0.1) -- (0:0.9) ;
\draw [<-,>=latex, dashed] (90:0.1)--(90:0.9);
\fill (180:1) circle (0.1);
\draw (180:1) ++(50: 0.3) node {v};
\draw (50: 0.3) node {w};
\draw (180:0.66) ++(-50: 0.3) node {$e^*$};
\draw (0:0.19) ++(-50: 0.3) node {f};
\draw (90:0.41) ++(150: 0.3) node {g};
\draw (90:0.99) ++(150: 0.3) node {x};
\fill (90:1) circle (0.1);
\end{tikzpicture}
\caption{}
\label{p42}
\end{figure}

If $\sigma(e)\neq 1$, then $\delta_e I(\Gamma,\sigma,\psi)=-\delta_{e^*} I(\Gamma^*,\sigma^*,\psi)$, where $\sigma^*$ is naturally induced by $\sigma$.

\end{lm}

\begin{lm}\label{p5}
Suppose that $\Gamma$ looks as in Figure \ref{p51} around $e$.

 \begin{figure}[H]
\centering
\begin{tikzpicture} 
\fill (0,0) circle (0.1) ;
\draw [<-, >= latex] (180:0.1) -- (180:0.9) ;
\draw [<-, >=latex]  (180:1.1) -- (180:1.9) ;
\draw [->, >=latex]  (0:0.1) -- (0:0.9) ;
\draw [<-,>=latex, dashed] (180:1) ++(90:0.1)--++(90:0.9);
\draw [<-,>=latex, dashed] (90:0.1)--(90:0.9);
\fill (180:1) circle (0.1);
\draw (180:1) ++(50: 0.3) node {v};
\draw (50: 0.3) node {w};
\draw (180:0.66) ++(-50: 0.3) node {e};
\draw (180:1) ++(90:0.41) ++(150: 0.3) node {f};
\draw (180:1) ++(90:0.99) ++(150: 0.3) node {x};
\draw (90:0.99) ++(150: 0.3) node {y};
\draw (90:0.41) ++(150: 0.3) node {g};
\fill (90:1) circle (0.1);
\fill (180:1) ++ (90:1) circle (0.1);
\end{tikzpicture}\ \ \ \ 
\begin{tikzpicture} 
\draw (0,0) circle (0.1) ;
\draw [<-, dashed,>= latex] (180:0.1) -- (180:0.9) ;
\draw [<-,dashed, >=latex]  (180:1.1) -- (180:1.9) ;
\draw [->, dashed,>=latex]  (0:0.1) -- (0:0.9) ;
\draw [<-,>=latex, dashed] (180:1) ++(90:0.1)--++(90:0.9);
\draw [<-,>=latex, dashed] (90:0.1)--(90:0.9);
\draw(180:1) circle (0.1);
\draw (180:1) ++(50: 0.3) node {v};
\draw (50: 0.3) node {w};
\draw (180:0.66) ++(-50: 0.3) node {e};
\draw (180:1) ++(90:0.41) ++(150: 0.3) node {f};
\draw (180:1) ++(90:0.99) ++(150: 0.3) node {x};
\draw (90:0.99) ++(150: 0.3) node {y};
\draw (90:0.41) ++(150: 0.3) node {g};
\fill (90:1) circle (0.1);
\fill (180:1) ++ (90:1) circle (0.1);
\end{tikzpicture}
\caption{}
\label{p51}
\end{figure}

If $\sigma(e)\neq 1$, then $\delta_e I(\Gamma,\sigma,\psi) =-\delta_{e} I(\Gamma,\sigma\circ\rho,\psi)$, where $\rho$ is the transposition of $f$ and $g$.

\end{lm}
\begin{proof}Let us prove Lemma \ref{p1}, and explain why it is possible to deal with $\sigma(e)=1$ in Lemma \ref{p2}. Lemmas \ref{p3} and \ref{p4} are proved similarly. Lemma \ref{p5} is proved as Lemma \ref{vanishing-h-2} (for example), using the orientation-reversing diffeomorphism that exchanges $x$ and $y$.

In Lemma \ref{p1}, we have $\face[e]= \facepC\times\s^{n-1}$ and $\partial_{e^*}C_{\Gamma^*}(\psi)= -\facepC\times\s^{n+1}$ since the graphs $\delta_e\Gamma$ and $\delta_{e^*}\Gamma^*$ are identical, and one can check by computation that the orientations are different, as in the second row of Figure \ref{table}.
For any edge $h\neq e$ of $\Gamma$ or $\Gamma^*$ the maps $p_h\colon \face[e]\rightarrow C_e$ and $ p_h^*\colon \partial_{e^*}C_{\Gamma^*}(\psi)\rightarrow C_e$ factor through the maps $\mathrm{pr}_1\colon \face[e]\rightarrow \facepC$ and $\mathrm{pr}_{1,*}\colon \partial_{e^*}C_{\Gamma^*}(\psi)\rightarrow \facepC$. 
The maps $G\circ p_e$ and $G_{\tau_{\sigma(e)}}\circ p_{e^*}$ are exactly the maps $\mathrm{pr}_2\colon \face[e]\rightarrow \s^{n-1}$ and $\mathrm{pr}_{2,*}\colon\partial_{e^*}C_{\Gamma^*}(\psi)\rightarrow\s^{n+1}$. 
Then, one can write $\formeT = \mathrm{pr}_1^*(\lambda)\wedge\mathrm{pr}_2^*(\omega_{\propint_{\sigma(e)}})$ and $\formeT[\Gamma^*][\sigma^*] =  \mathrm{pr}_{1,*}^*(\lambda)\wedge\mathrm{pr}_{2,*}^*(\omega_{\propext_{\sigma(e)}})$  for some form $\lambda$ on $\facepC$.
This implies that \[\delta_e I(\Gamma,\sigma, \psi) = \int_{\facepC} (\lambda \int_{\s^{n-1}} \omega_{\propint_{\sigma(e)}} )=  \int_{\facepC} \lambda = \int_{\facepC}( \lambda \int_{\s^{n+1}} \omega_{\propext_{\sigma(e)}}) =- \delta_{e^*} I(\Gamma^*,\sigma^*,\psi),\] where the minus sign comes from the identification $\partial_{e^*}C_{\Gamma^*}(\psi)= -\facepC\times\s^{n+1}$. This proves Lemma \ref{p1}.

In the proof of Lemma \ref{p2}, we can similarly prove that $\formeT = \mathrm{pr}_1^*(\lambda)\wedge\mathrm{pr}_2^*(\mu_e)$ and $\formeT[\Gamma^*][\sigma^*] = \mathrm{pr}_{1,*}^*(\lambda)\wedge\mathrm{pr}_{2,*}^*(\mu_e)$ where $\lambda$ is a form on $\facepC$ and where \[\mu_e=\begin{cases}\omega_{\propext_{\sigma(e)}}& \text{if $\sigma(e)\neq 1$,} \\ \volext &\text{if $\sigma(e)=1$ and $(\propint'_1-\propint_1,\propext'_1-\propext_1)$ has the sphere factorization property,}\end{cases} \]so that $\formeT =  \formeT[\Gamma^*][\sigma^*]$. Since both faces are diffeomorphic with opposite orientations, $\delta_e I(\Gamma,\sigma, \psi) = - \delta_{e^*} I(\Gamma^*,\sigma^*, \psi)$.

Figure \ref{table} describes the different orientations used to check Lemmas \ref{p2} to \ref{p4}, where $\Omega'$ denotes the wedge products of the $\Omega_h$, on the external edges $h$ not named on the pictures, $\d Y_* = \bigwedge\limits_{i=1}^n \d Y_*^i$, and $\d X_* = \bigwedge\limits_{i=1}^{n+2} \d X_*^i$.\qedhere

\begin{figure}[H]
\centering
\begin{tabular}{|c|c|c|c|}
  \hline
  Lemma & $\epsilon(\Gamma^*)/\epsilon(\Gamma)$& $\Omega(C_{\delta_e\Gamma}) $&$ \Omega(C_{\delta_e\Gamma^*}) $ \\
  \hline
 \ref{p2} & $-1$ & $\epsilon(\Gamma)\d Y_v\Omega'$&$\epsilon(\Gamma^*)\d Y_v\Omega'$ \\
  \ref{p1} & $-1$ & $\epsilon(\Gamma)\d Y_v \Omega_{f_-}\Omega_{g_+}\Omega'$&  $\epsilon(\Gamma^*)\d Y_v \Omega_{f_-}\Omega_{g_+}\Omega'$\\
\ref{p3} & $+1$& $\epsilon(\Gamma)\d Y_v \d Y_x \Omega_{f_-} \Omega'$ &$-\epsilon(\Gamma^*)\d Y_v \d Y_x \Omega_{f_-} \Omega'$ \\
\ref{p4} &$+1$& $\epsilon(\Gamma)\d Y_v \d Y_x \Omega_{f_+} \Omega'$  &$-\epsilon(\Gamma^*)\d Y_v \d Y_x \Omega_{f_+} \Omega'$\\
  \hline
\end{tabular}\caption{Face orientations}\label{table}
\end{figure}

\end{proof}

\section{ Proofs of Theorem \ref{parallelization theorem} and Proposition \ref{prop210}}\label{annex2}
A topological pair $(X,A)$ is the data of a topological space $X$ and a subset $A\subset X$. A map $f\colon (X,A) \rightarrow (Y,B)$ between two such pairs is a continuous map $f\colon X \rightarrow Y$ such that $f(A)\subset B$.

If $(X,A)$ and $(Y,B)$ are two topological pairs, $[(X,A), (Y,B)]$ denotes the set of homotopy classes of maps from $(X,A)$ to $(Y,B)$.

\begin{lm}Let $\ambientspace$ be a parallelizable asymptotic homology $\R^{n+2}$, and fix a parallelization $\tau_0$ of $\ambientspace$.
For any map $g\colon\ambientspace\rightarrow SO(n+2)$ that sends $\voisinageinfini$ to the identity matrix $I_{n+2}$, define the map $\psi(g)\colon (x,v) \in \ambientspace\times\R^{n+2} \mapsto (x,g(x)(v))\in \ambientspace\times\R^{n+2}$.

The map $$\begin{array}{lll}[(\ambientspace, \voisinageinfini), (SO(n+2), I_{n+2})]& \rightarrow &\Par \\ \left[g\right]& \mapsto & \left[ \tau_0\circ \psi (g)\right] \end{array}$$ is well-defined and is a bijection.
\end{lm}
\begin{proof}
The lemma would be direct with $GL_{n+2}^+(\R)$ instead of $SO(n+2)$, and $SO(n+2)$ is a deformation retract of $GL_{n+2}^+(\R)$.
\end{proof}

A \emph{homology $(n+2)$-ball} is a compact smooth manifold that has the same integral homology as a point, and whose boundary is the $(n+1)$-sphere $\s^{n+1}$.

We are going to prove the following theorem, which implies Theorem \ref{parallelization theorem}.

\begin{theo}\label{71}
Let $\B$ be a standard $(n+2)$-ball inside the interior of a homology $(n+2)$-ball $\bouleambiante$.
For any map $f\colon (\B,\partial \B) \rightarrow (SO(n+2), I_{n+2})$, define the map $I(f)\colon (\bouleambiante,\partial\bouleambiante)\rightarrow (SO(n+2), I_{n+2})$ such that \[I(f)(x) = \left\{\begin{array}{lll} f(x) & \text{if $x\in \B$,}\\ I_{n+2} & \text{otherwise.}\end{array}\right. \]

Then, the induced map  $$\begin{array}{lll} [(\B, \partial \B) , (SO(n+2), I_{n+2})]&\rightarrow& [(\bouleambiante, \partial\bouleambiante), (SO(n+2), I_{n+2})]\\ \  [f] & \mapsto& [I(f)] \end{array}$$ is surjective.
%
%

\end{theo}
In order to prove this theorem, we are going to build a right inverse to this map. To a map $f\colon (\bouleambiante, \partial\bouleambiante)\rightarrow (SO(n+2), I_{n+2})$, we will associate a map $g$ homotopic to $f$, such that $g(M\setminus \B) = \{I_{n+2}\}$.

\begin{lm}\label{triangulation lemma}
Let $(Y, y_0)$ be a path-connected pointed space with abelian fundamental group, and let $\bouleambiante$ be a homology $(n+2)$-ball.
Let $f\colon(\bouleambiante, \partial \bouleambiante) \rightarrow (Y, y_0)$ be a continuous map.

Then, $f$ is homotopic to a map $g$ that sends the complement of $\B$ to $y_0$, among the maps that send $\partial\bouleambiante$ to $y_0$. 
\end{lm}
\begin{proof} In this proof, "homotopic" will always mean "homotopic among the maps that send $\partial \bouleambiante $ to $y_0$."

Fix a triangulation $T$ of $(\bouleambiante, \partial\bouleambiante)$, and denote by $T^{(k)}$ its $k$-skeleton. The first projection map $p\colon \mathcal B=\bouleambiante\times Y \rightarrow \bouleambiante$ defines a trivial bundle over $(\bouleambiante, \partial\bouleambiante)$.
Set $f_0\colon x\in \bouleambiante \mapsto (x, f(x))\in \mathcal B$ and $f_1\colon x\in \bouleambiante \mapsto (x, y_0)\in \mathcal B$. Since $H^q(\bouleambiante, \partial\bouleambiante, \mathbb Z)=0$ for any $0\leq q\leq n+1$, the groups $H^q(\bouleambiante, \partial\bouleambiante,\pi_q(Y, y_0))$ are also trivial. Obstruction theory defined by Steenrod in \cite{Steenrod}, or more precisely in Theorem 34.10, therefore guarantees the existence of a homotopy between $f_0$ and a map $f_2$ such that $(f_2)_{|T^{(n+1)} }= {(f_1)}_{|T^{(n+1)}}$ among maps from $\bouleambiante$ to $\mathcal B$ that coincide with $f_1$ on $\partial\bouleambiante$. This implies that $f$ is homotopic to a map $g$ that maps $T^{(n+1)}$ to $y_0$.

It remains to prove that $g$ is homotopic to a map that sends the complement of $\B$ to $y_0$.
Let $U$ be a regular neighborhood of $T^{(n+1)}$. Up to a homotopy, assume that $g$ sends $U$ to $y_0$. 
There exists a closed ball $V$ such that $U\cup V = B(M)$ and such that $V$ contains $\B$:
indeed, it suffices to take a closed regular neighborhood $V_0$ of a tree with exactly one vertex in each $(n+2)$-cell of $T$ ; up to homotopy, we can assume that $\B$ is contained in one $(n+2)$-cell of $T$ and that $\B$ does not meet the tree neither its neighborhood $V_0$, and we let $V$ be the union of $V_0$, $\B$ and a cylinder between these two disjoint closed balls. 

 Then, $g$ maps the complement of $V$ to $y_0$. Since $V$ is a ball, and $\B$ a ball inside $V$, $g_{|V}$ is homotopic to a map that sends $V\setminus \B$ to $y_0$, among the maps that send $\partial V$ to $y_0$. This implies Lemma \ref{triangulation lemma}.\qedhere

\end{proof}

Then, any element of $[(\bouleambiante, \partial\bouleambiante), (SO(n+2), I_{n+2})]$ can be represented by a map $f\colon \ambientspace\rightarrow SO(n+2)$, such that $f(\bouleambiante\setminus \B) = \{I_{n+2}\}$. This proves Theorem \ref{71}, and therefore Theorem \ref{parallelization theorem}.

\begin{proof}[Proof of Proposition \ref{prop210}]\label{pf210}
We are going to prove that the connected sum of any asymptotic homology $\R^{n+2}$ with itself is parallelizable in the sense of Definition \ref{paral-def}.

As in the previous proof, obstruction theory shows that for any ball $\B$ inside the interior of $B(M)$, there exists a parallelization on $B(M)\setminus \B$ that coincides with the standard one on $\partial B(M)=\partial \voisinageinfini\subset \R^{n+2}$, and that the obstruction to extending it to a parallelization as in Definition \ref{paral-def} lies in $H^{n+2}(B(M), \partial B(M), \pi_{n+1}(SO(n+2), I_{n+2}))\cong \pi_{n+1}(SO(n+2),I_{n+2})$.

This group is known (see for example \cite{[Kervaire]}) and, for any odd $n\geq 1$, any element of $\pi_{n+1}(SO(n+2), I_{n+2})$ is of order $1$ or $2$.
This proves Proposition \ref{prop210}.\qedhere
\end{proof}
\begin{rmq}
Any asymptotic homology $\R^3$ or $\R^7$ is parallelizable in the sense of Definition \ref{paral-def}.
\end{rmq}
\begin{proof}
This follows from the same arguments as above 
and from the fact that $\pi_2(SO(3), I_3)$ and $\pi_6(SO(7), I_7)$ are trivial (see again for example \cite{[Kervaire]}).
\end{proof}
\section{Proof of Theorem \ref{additivity}: additivity of $Z_k$}\label{S8}

Recall that $G$ is the Gauss map $\configR \rightarrow \s^{n-1}$. In this section, $G_{ext}$ denotes the Gauss map $C_2(\R^{n+2})\rightarrow \s^{n+1}$. 

The proof in this section is an adaptation to the higher dimensional case of the method developed in \cite[Sections 16.1-16.2]{[Lescop2]}. Important differences appear in Section \ref{S82}.

\subsection{Definition of extended BCR diagrams}\label{S82}

Fix an integer $k\geq 2$, and let $\psi_{triv}\colon x\in \R^n \hookrightarrow (0, 0, x)\in \R^{n+2}$ be the trivial knot.

For any $(\Gamma,\sigma)\in\graphesnum$, and any $S_1\sqcup S_2 \subsetneq \sommets$, define the graph $\Gamma_{S_1,S_2}$ as follows: remove the edges of $\Gamma$ between two vertices of $S_1$ or two vertices of $S_2$. Next, remove the isolated vertices. 
Eventually blow up the obtained graph at each vertex of $S_1\sqcup S_2$, by replacing such a vertex with a univalent vertex for each adjacent half-edge on the corresponding half-edge. 
Note that the corresponding half-edges do not meet anymore in $\Gamma_{S_1,S_2}$. 
Let $\overline S_i$ denote the set of all the vertices in $\Gamma_{S_1, S_2}$ coming from a (possibly blown-up) vertex of $S_i$ in $\Gamma$. 
The graph $\Gamma_{S_1,S_2}$ is endowed with a partition $\overline S_1\sqcup \overline S_2 \sqcup( \sommets\setminus (S_1\sqcup S_2) )$, and its edges are the edges of $\Gamma$ that do not have both ends in $S_1$ or both ends in $S_2$. We set $V_i(\Gamma_{S_1,S_2})= (V_i(\Gamma) \setminus(S_1\sqcup S_2))\sqcup (\overline S_1\sqcup \overline S_2)$ so that $V_e(\Gamma_{S_1,S_2}) = \sommetsexternes\setminus(S_1\sqcup S_2)$.

Figure \ref{FigS} shows an example of the obtained graph $\Gamma_{S_1,S_2}$. This graph can be thought of as a modified version of $\Gamma$ where we only look at the directions between vertices which are not in the same $S_i$.
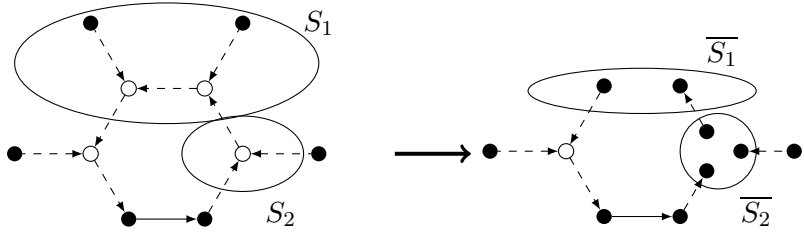
\begin{figure}[H]\centering
\begin{tikzpicture}
\draw (60:1) circle (0.1) (120:1) circle (0.1) (180:1) circle (0.1) (0:1) circle (0.1);
\fill (60:2) circle (0.1) (120:2) circle (0.1) (180:2) circle (0.1) (240:1) circle (0.1) (300:1) circle (0.1) (0:2) circle (0.1);
\draw [dashed, >= latex, ->] (0:1) ++(120:0.1) -- ++(120:0.8);
\draw [dashed, >= latex, ->] (60:1) ++(180:0.1) -- ++(180:0.8);
\draw [dashed, >= latex, ->] (120:1) ++(240:0.1) -- ++(240:0.8);
\draw [dashed, >= latex, ->] (180:1) ++(300:0.1) -- ++(300:0.8);
\draw [ >= latex, ->] (240:1) ++(0:0.1) -- ++(0:0.8);
\draw [dashed, >= latex, ->] (300:1) ++(60:0.1) -- ++(60:0.8);
\draw [dashed, >= latex, <-] (0:1) ++(0:0.1) -- ++(0:0.8);
\draw [dashed, >= latex, <-] (60:1) ++(60:0.1) -- ++(60:0.8);
\draw [dashed, >= latex, <-] (120:1) ++(120:0.1) -- ++(120:0.8);
\draw [dashed, >= latex, <-] (180:1) ++(180:0.1) -- ++(180:0.8);
\draw  (90: 1.2) ellipse (2 and 0.8);
\draw (0: 1) ellipse (0.8 and 0.5);
\draw (60:2) ++(1, 0) node {$S_1$};
\draw (1.5, -0.8) node {$S_2$};
\draw[line width = 1.5,->] (3, 0) --(4,0) ;
\end{tikzpicture}
\begin{tikzpicture}
\draw  (180:1) circle (0.1);
\fill (300:1) circle (0.1) ;
\fill  (180:2) circle (0.1) (240:1) circle (0.1) ;
\fill (60:1) circle (0.1) (120:1) circle (0.1) ;
\fill (2, 0) circle (0.1);
\fill (0:1) ++ (0: 0.3) circle (0.1);
\fill (0:1) ++ (120: 0.3) circle (0.1);
\fill  (0:1) ++ (240: 0.3) circle (0.1);
\draw [dashed, >= latex, ->] (0:1) ++(120:0.4) -- ++(120:0.5);
\draw [dashed, >= latex, ->] (120:1) ++(240:0.1) -- ++(240:0.8);
\draw [dashed, >= latex, ->] (180:1) ++(300:0.1) -- ++(300:0.8);
\draw [>= latex, ->] (240:1) ++(0:0.1) -- ++(0:0.8);
\draw [dashed, >= latex, ->] (300:1) ++(60:0.1) -- ++(60:0.5);
\draw [dashed, >= latex, <-] (0:1) ++(0:0.4) -- ++(0:0.5);
\draw [dashed, >= latex, <-] (180:1) ++(180:0.1) -- ++(180:0.8);
\draw  (90: 0.8) ellipse (1.5 and 0.3);
\draw (0: 1) ellipse (0.5 and 0.5);
\draw (60:1.3) ++(0.4, 0.2) node {$\overline{S_1}$};
\draw (1.5, -0.8) node {$\overline{S_2}$};
\end{tikzpicture}
\caption{Construction of $\Gamma_{S_1,S_2}$ for some degree $5$ BCR diagram}
\label{FigS}
\end{figure}
Since any edge of $\Gamma_{S_1, S_2}$ comes from an edge of $\Gamma$, the numbering $\sigma$ induces a map $\sigma_{S_1,S_2} \colon E(\Gamma_{S_1,S_2})\hookrightarrow \{1,\ldots, 2k\}$. 
Set $\Omega_i = (0, 0, \ldots, 0, \frac{(-1)^i}2)$ as in Section \ref{S29}.
One can associate the configuration space \[C^0_{\Gamma_{S_1,S_2}}(\psi_{triv})=\left\{c\colon \sommets[\Gamma_{S_1,S_2}] \rightarrow \R^{n+2} \ \left| \ \begin{array}{l} c_{|\sommets[\Gamma_{S_1,S_2}]\setminus (\overline S_1\sqcup \overline S_2)} \text{ is injective}\\ \text{ and does not take the values $\Omega_1$ or $\Omega_2$, }\\ c(\sommetsinternes[\Gamma_{S_1,S_2}])\subset \psi_{triv}(\R^n),\\ c(\overline S_1)=\{\Omega_1\}, c(\overline S_2)=\{\Omega_2\}\end{array}\right.\right\}\] to the obtained graph $\Gamma_{S_1,S_2}$. As before, $c_i\colon \R^n \hookrightarrow \R^{n+2}$ denotes the map such that $c_{|\sommetsinternes[\Gamma_{S_1,S_2}]} = \psi_{triv} \circ c_i$.
This space admits a compactification $\confspaceAA$ as in Section \ref{S2.4} such that for any $e=(v,w)\in E(\Gamma_{S_1,S_2})$ the map \[G_e^0 \colon c\in C^0_{\Gamma_{S_1,S_2}}(\psi_{triv})\mapsto \begin{cases} \frac{c_i(w)-c_i(v)}{||c_i(w)-c_i(v)||} & \text{ if $e$ is internal,} \\ \frac{c(w)-c(v)}{||c(w)-c(v)||} & \text{if $e$ is external,}\end{cases} \in \s^{n(e)}\] extends to a smooth map  $G_e\colon \confspaceAA \rightarrow \s^{n(e)}$. For simplicity, we will simply denote this compact space by $\confspaceA$ in the following.

\begin{lm}\label{stab-conf}
For any $(\Gamma, \sigma)\in \graphesnum$ and any $S_1\sqcup S_2\subsetneq \sommets$, 
 \[\dim(\confspaceA)\leq \sum\limits_{e\in E(\Gamma_{S_1,S_2}) } n(e).\] 
Furthermore, this inequality is an equality if and only if $S_1=S_2=\emptyset$.
\end{lm}

\begin{proof}
We use the same method as in the proof of Lemmas \ref{dim=deg} and \ref{inégalité-s*}.
Split any edge of $\Gamma_{S_1, S_2}$ into two halves $e_-$ and $e_+$, and assign an integer $\tilde{d}(e_\pm)$ to each half-edge $e_\pm$ as follows: \begin{itemize}
\item if $e_\pm$ is adjacent to a vertex of $\overline S_1\sqcup \overline S_2$, $\tilde{d}(e_\pm)=0$,
\item otherwise, $\tilde{d}(e_\pm)$ is the integer $d(e_\pm)$ of Lemma \ref{dim=deg}.
\end{itemize}

Note that for any vertex $v$, \[\sum\limits_{e_\pm\text{ adjacent to } v} \tilde{d}(e_\pm) = \begin{cases} 0 & \text {if $ v \in \overline S_1\cup \overline S_2$},\\
n & \text{if $v$ is internal and $ v \not\in \overline S_1\cup \overline S_2$,}\\
n+2 & \text{if $v$ is external and $ v \not\in \overline S_1\cup \overline S_2$.}
\end{cases}\] This implies that $\sum\limits_{e\in E(\Gamma_{S_1,S_2})} (\tilde{d}(e_+) + \tilde{d}(e_-)) = \dim(\confspaceA) $. This construction also ensures that $\tilde{d}(e_-) + \tilde{d}(e_+) \leq n(e)$ for any edge $e=(v,w)$, with equality if and only if $(v,w)\in (V(\Gamma_{S_1,S_2})\setminus (\overline S_1\sqcup\overline S_2))^2$ or if $e$ is an internal edge coming from $V(\Gamma_{S_1,S_2})\setminus (\overline S_1\sqcup\overline S_2)$ and going to $\overline S_1\sqcup\overline S_2$. This proves the inequality of the lemma.

 Let us prove that the inequality is strict when $S_1\sqcup S_2\neq \emptyset$. In this case, $\overline S_1\sqcup \overline S_2\neq \emptyset$, so there exists an edge $e$ with one end in $\overline S_1\sqcup \overline S_2$ and the other one in $V(\Gamma_{S_1,S_2})\setminus( \overline S_1\sqcup\overline S_2)$. If there exists such an edge that is not an internal edge going from $V(\Gamma_{S_1,S_2})\setminus( \overline S_1\sqcup\overline S_2)$ to $\overline S_1\sqcup \overline S_2$, it satisfies $\tilde{d}(e_-) + \tilde{d}(e_+) < n(e)$, and the inequality of the lemma is strict. 
But if there is an internal edge from $V(\Gamma_{S_1,S_2})\setminus( \overline S_1\sqcup\overline S_2)$ to $\overline S_1\sqcup \overline S_2$, neither $S_1\sqcup S_2$ nor $\sommets\setminus(S_1\sqcup S_2)$ contains the whole cycle of $\Gamma$. This implies that there is at least one edge from $\overline S_1\sqcup\overline S_2$ to $V(\Gamma_{S_1,S_2})\setminus( \overline S_1\sqcup\overline S_2)$, and concludes.

If $S_1\sqcup S_2= \emptyset$, the inequality of the lemma is an equality, since $\Gamma_{\emptyset, \emptyset} = \Gamma$.
\qedhere

\end{proof}

\begin{cor}\label{gras}
For any $(\Gamma, S_1, S_2)$ as in Lemma \ref{stab-conf} and any numbering $\sigma$ of $\Gamma$, define the maps \[\begin{array}{llll} G_{\Gamma_{S_1,S_2}}\colon & \confspaceA & \rightarrow &  \prod\limits_{e\in E(\Gamma_{S_1, S_2})} \s^{n(e)} \\  & c & \mapsto& (G_e(c))_{e\in E(\Gamma_{S_1,S_2})}\end{array}\] and \[\begin{array}{llll}\pi_{\Gamma_{S_1,S_2}, \sigma} \colon& (\s^{n-1}\times\s^{n+1})^{2k} & \rightarrow & \prod\limits_{e\in E(\Gamma_{S_1,S_2})}\s^{n(e)}\\ & (X_i^{n-1}, X_i^{n+1})_{1\leq i \leq 2k} & \mapsto & (X_{\sigma(e)}^{n(e)})_{e\in E(\Gamma_{S_1,S_2})}\end{array} \]

For any maps $\hat\epsilon, \hat\epsilon'\colon \{1,\ldots, 2k\} \rightarrow\{\pm1\} $, set \[\begin{array}{llll}T_{\hat\epsilon, \hat\epsilon'}  \colon& (\s^{n-1}\times\s^{n+1})^{2k} & \rightarrow & (\s^{n-1}\times\s^{n+1})^{2k} \\ & (X_i^{n-1}, X_i^{n+1})_{1\leq i \leq 2k} & \mapsto & (\hat\epsilon (i)X_i^{n-1},\hat\epsilon'(i) X_i^{n+1})_{1\leq i \leq 2k}\end{array} .\]
For any $(\Gamma, \sigma, S_1, S_2, \hat\epsilon, \hat\epsilon')$, the set ${T_{\hat\epsilon, \hat\epsilon'}}^{-1}\left( {\pi_{\Gamma_{S_1,S_2},\sigma}}^{-1}\left(G_{\Gamma_{S_1,S_2}}\left(\confspaceA \right)\right)\right)$ is a closed subset with empty interior of $(\s^{n-1}\times \s^{n+1})^{2k}$.

Then, $\mathcal O_k= \bigcap\limits_{\Gamma,S_1,S_2, \sigma,\hat\epsilon,\hat\epsilon'}\left( (\s^{n-1}\times\s^{n+1})^{2k}\setminus{T_{\hat\epsilon, \hat\epsilon'}}^{-1}\left({\pi_{\Gamma_{S_1,S_2},\sigma}}^{-1}\left(G_{\Gamma_{S_1,S_2}}\left(\confspaceA \right)\right)\right)\right)$ is an open dense set of $(\s^{n-1}\times \s^{n+1})^{2k}$.

\end{cor}

\begin{proof}Since $\confspaceA$ is compact, $G_{\Gamma_{S_1,S_2}}\left(\confspaceA \right)$ is compact and therefore closed. Let us prove that its interior is empty. 

If $S_1\sqcup S_2\neq \emptyset$, Lemma \ref{stab-conf} and the Morse-Sard theorem ensure that the image of $G_{\Gamma_{S_1,S_2}}$ has empty interior, since the target of this map has greater dimension than its source.

If $S_1\sqcup S_2= \emptyset$, $G_{\Gamma_{\emptyset,\emptyset}}$ is a map between two manifolds of same dimension. Let $\R^n$ act by translations along $\{0\}^2\times\R^n\subset \R^{n+2}$ on $C_{\Gamma_{\emptyset,\emptyset}}(\psi_{triv})$. The map $G_{\Gamma_{\emptyset,\emptyset}}$ factors through the quotient map of this action. Using the Morse-Sard theorem, this again implies that the image of $G_{\Gamma_{\emptyset,\emptyset}}$ has empty interior. 

Then, $G_{\Gamma_{S_1,S_2}}\left(\confspaceA \right)$ is always closed with empty interior. This implies that ${\pi_{\Gamma_{S_1,S_2},\sigma}}^{-1}\left(G_{\Gamma_{S_1,S_2}}\left(\confspaceA \right)\right)$ is also closed with empty interior since $\pi_{\Gamma_{S_1,S_2},\sigma}$ is an open map. Since $T_{\hat\epsilon, \hat\epsilon'} $ is a diffeomorphism, the first assertion of the lemma follows. Then, $\mathcal O_k$ is a finite intersection of open dense sets in the complete metric space $(\s^{n-1}\times \s^{n+1})^{2k}$. The lemma follows from the Baire category theorem.\qedhere

\end{proof}

Lemma \ref{gras}, which is used in Section \ref{S9.3} to prove Theorem \ref{additivity}, also yields a proof (but not the simplest one) of the following result. 

\begin{cor}\label{psio}
For the trivial knot $\psi_{triv}$, $Z_k(\psi_{triv})=0$.
\end{cor}

\begin{proof}Because of Corollary \ref{gras}, $\mathcal O_k$ is non empty. Fix $(X_i^{n-1},X_i^{n+1})_{1\leq i \leq 2k}\in \mathcal O_k$. 
Compute $Z_k$ with the propagating chains $A_i= \frac12G^{-1}(\{-X^{n-1}_i, + X^{n-1}_i\})$ and $B_i= \frac12G_{ext}^{-1}(\{-X^{n+1}_i, + X^{n+1}_i\})$. The definition of $\mathcal O_k$ implies that the intersection numbers in Theorem \ref{the} are all zero.
\end{proof}

\subsection{An extension of the Gauss map}\label{S91}

Let $(\punct{M_1}, \tau_1)$ and $(\punct{M_2}, \tau_2)$ be two parallelized asymptotic homology $\R^{n+2}$. Fix two knots $\psi_1\colon \R^n \hookrightarrow\punct{M_1}$ and $\psi_2\colon \R^n \hookrightarrow \punct{M_2}$, and an integer $k\geq 2$.

Fix $\eta\in (0, \frac 12)$, and let $\voisinageinfinideux[\eta]$ be the complement in $\R^{n+2}$ of the open balls $\mathring{B^1_\eta}$ and $\mathring{B^2_\eta}$ of respective centers $\Omega_1=(0, \ldots, 0,-\frac12)$ and $\Omega_2=(0, \ldots, 0, \frac12)$ and radius $\eta$.

Glue $\voisinageinfinideux[\eta]$ and the two closed balls $B(M_1)$ and $B(M_2)$ along $\partial B^1_\eta$ and $\partial B^2_\eta$. In this setting, $B_\eta(M_1)$ and $B_\eta(M_2)$ denote the images of $B(M_1)$ and $B(M_2)$, since they "replace" the balls $B^1_\eta$ and $B^2_\eta$. The obtained manifold $\ambientspace$ identifies with $\punct{M_1}\sharp \punct{M_2}$ and comes with a decomposition $\voisinageinfinideux[\eta]\cup B_\eta(M_1)\cup B_\eta(M_2)$ and a parallelization $\tau$ naturally induced by $\tau_1$, $\tau_2$, and the standard parallelization of $\voisinageinfinideux[\eta]\subset \R^{n+2}$ up to homotopy. For $\eta < r < \frac12$, $B_r(M_i)$ denotes the union of $B_\eta(M_i)$ with $\{x \in \voisinageinfinideux[\eta] \mid d(x, \Omega_i) \leq r\}$.

\begin{df}
Let $\chi_\pi\colon [0, 3\eta] \rightarrow \R_+$ be a smooth increasing map such that $\chi_\pi^{-1}(\{0\})=[0,\eta] $ and $\chi_\pi([2\eta,3\eta])=\{1\}$.
Let $\pi\colon\punct{M_1}\sharp \punct{M_2}\rightarrow \R^{n+2}$ be the smooth map such that, for any $x \in \punct{M_1}\sharp\punct{M_2}$,

\[\pi(x) = \left\{
\begin{array}{lll}
x & \text{if $x\in \voisinageinfinideux[2\eta]$, }\\
\Omega_1 & \text{if $x\in B_\eta(M_1)$,}\\
\Omega_2 & \text{if $x\in B_\eta(M_2)$,}\\
\Omega_1+ \chi_\pi(||x - \Omega_1||).(x-\Omega_1) & \text{if $x\in B_{2\eta}(M_1)\setminus B_{\eta}(M_1)$,} \\
\Omega_2+ \chi_\pi(||x - \Omega_2||).(x-\Omega_2) & \text{if $x\in B_{2\eta}(M_2)\setminus B_{\eta}(M_2)$.} 
\end{array}
\right. \]

 Set $C_2(B_{2\eta}(M_i)) = p_b^{-1}(B_{2\eta}(M_i)^2)$, and set $$D(G_{\tau,\eta})= ( \configM \setminus (C_2(B_{2\eta}(M_1)) \cup C_2(B_{2\eta}(M_2)) ) )\cup U\punct M.$$ Define the analogue $G_{\tau, \eta}\colon D(G_{\tau, \eta}) \rightarrow \s^{n+1}$ of the Gauss map as the map such that for any $c\in D(G_{\tau, \eta})$,
 \[  G_{\tau,\eta}(c) = \left\{\begin{array}{lll}
\frac{\pi(y) -\pi(x)}{||\pi(y)-\pi(x)||} & \text{if $c=(x,y)\not\in C_2(B_{2\eta}(M_1)) \cup C_2(B_{2\eta}(M_2))\cup U\punct M$,}\\ 
G_\tau(c) & \text{if $c\in U\punct M$.}
\end{array}
\right.\]
Note that $(G_{\tau, \eta})_{| C_2(B_{\infty, 2\eta})} = (G_{ext})_{| C_2(B_{\infty, 2\eta})}$ and $(G_{\tau,\eta})_{|\partial\configM} = G_\tau$.
\end{df}

\subsection{Proof of the additivity}\label{S9.3}

Define the distance on $(\s^{n-1}\times \s^{n+1})^{2k}$ given by the maximum of the Euclidean distances on each spherical factor.
For $d= n \pm 1$, set $\s_h^d = \{X \in \s^d \mid {X_{d+1}}^2 < \frac12\}$. Let $\mathcal O_k'$ denote the intersection $\mathcal O_k\cap (\s_h^{n-1}\times \s_h^{n+1})^{2k}$. Corollary \ref{gras} ensures that $\mathcal O_k'$ is a non-empty open set.

Fix $(X^{n-1}_i, X^{n+1}_i)\in \mathcal O'_k$, and $\frac14>\delta>0$ such that the ball of radius $9\delta$ and center $(X^{n-1}_i, X^{n+1}_i)$ in $(\s^{n-1}\times \s^{n+1})^{2k}$ is contained in $\mathcal O'_k$.
Choose $\eta>0$ in Section \ref{S91} such that $\eta < \frac18(\frac\delta2)^{2k}$.
\begin{prop}\label{suitable}For any $1\leq i \leq 2k$, fix a closed antisymmetric $(n+1)$-form $\omega_{\beta_i}$ on $\s^{n+1}$ with total mass one, and with support contained in the union of the two balls of center $\pm X^{n+1}_i$ and radius $\delta$.

For any $1\leq i \leq 2k$, there exists an external propagating form $\beta_i$ of $(\ambientspace, \tau)$ such that $(\beta_i)_{|D(G_{\tau,\eta})}= G_{\tau, \eta}^*(\omega_{\beta_i})$. 
Furthermore, ${\beta_i}_{| B_{\frac14}(M_1)\times B_{\frac14}(M_2)}=0$ and ${\beta_i}_{| B_{\frac14}(M_2)\times B_{\frac14}(M_1)}=0$.

For any $1\leq i \leq 2k$, fix a closed antisymmetric $(n-1)$-form $\omega_{\alpha_i}$ on $\s^{n-1}$ with total mass one, with support contained in the union of the two balls of center $\pm X^{n-1}_i$ and radius $\delta$, and set $\alpha_i= G^*(\omega_{\alpha_i})$. 
These forms satisfy ${\alpha_i}_{| \psi^{-1}(B_{\frac14}(M_1))\times\psi^{-1}( B_{\frac14}(M_2))}=0$ and ${\alpha_i}_{| \psi^{-1}(B_{\frac14}(M_2))\times \psi^{-1}(B_{\frac14}(M_1))}=0$, where $\psi= \psi_1\sharp \psi_2$.
\end{prop}
\begin{proof}

Let us first construct the forms $\beta_i$.
First note that the condition on the restriction is compatible with the property of being a propagating form since $(G_{\tau, \eta})_{|  \partial\configM} = G_\tau$.
It remains to prove that the closed form $G_{\tau, \eta}^*(\omega_{\beta_i})$ on $D(G_{\tau, \eta})$ extends to a closed form on $\configM$. It suffices to prove that the restrictions to $\partial C_2(B_{2\eta}(M_1))$ and to $\partial C_2(B_{2\eta}(M_2))$ extend to $C_2(B_{2\eta}(M_1))$ and to $C_2(B_{2\eta}(M_2))$ as closed $(n+1)$-forms.
Note that $C_2(B_{2\eta}(M_i))$ is diffeomorphic to $C_2(\punct{M_i})$. Then, Lemma \ref{homologie} yields $H^{n+2}(C_2(B_{2\eta}(M_i)), \partial C_2(B_{2\eta}(M_2)))=0$ and implies the existence of the form $\beta_i$.
Since the support of $\omega_{\beta_i}$ is contained in $ \s_h^{n+1}$, the restriction ${\beta_i}_{| B_{\frac14}(M_1)\times B_{\frac14}(M_2)}=0$ vanishes. The same argument proves the similar assertion about $\alpha_i$.
\qedhere

\end{proof}

We are going to prove the following proposition, which implies Theorem \ref{additivity}.
\begin{prop}\label{prop-add}

Fix propagating forms $(\alpha_i)_{1\leq i \leq 2k}$ and $(\beta_i)_{1\leq i \leq 2k}$ as in Proposition \ref{suitable}, and set $F = (\alpha_i, \beta_i)_{1\leq i \leq 2k}$. Then, for any $(\Gamma,\sigma)\in \graphesnum$,

$$I^{F}(\Gamma,\sigma,\psi_1\sharp \psi_2) = I^{F}(\Gamma,\sigma,\psi_1) + I^{F}(\Gamma,\sigma,\psi_2).$$
\end{prop}

\begin{proof}

Fix $(\Gamma,\sigma)\in\graphesnum$. 
For $1\leq j \leq 2k+1$, set $\ray= \frac14(\frac\delta2)^{2k+1-j}$, and note that $\ray[1]+\cdots +\ray< \frac{\delta}{2-\delta} \ray[j+1] < \delta\ray[j+1]$ and that $\ray[2k+1] = \frac14$.

A \emph{coloring} is a map $\chi\colon \sommets\rightarrow \{(1,1), \ldots, (1,2k)\} \cup \{(2,1), \ldots,(2,2k)\}\cup \{\infty\}$. For a given coloring $\chi$, define $U(\chi)$ as the set of configurations in $C_\Gamma(\psi_1\sharp \psi_2)$ such that:
\begin{itemize}
\item If $\chi(v) =(1, 1)$, then $c(v)$ is in $\mathring B_{2\ray[1]}(M_1)$, and if $\chi(v) = (2,1)$, $c(v)$ is in $\mathring B_{2\ray[1]}(M_2)$.
\item If $\chi(v)\not\in\{(1,1), (2,1)\}$, then $c(v)$ is neither in $B_{\ray[1]}(M_1)$ nor in $B_{\ray[1]}(M_2)$. In particular, since $2\eta<\ray[1]$, $c(v) \in \voisinageinfinideux[{\ray[1]} ]\subset\voisinageinfinideux[2\eta]$, and it makes sense to use the Euclidean norm of $\R^{n+2}$ for such vertices.
\item If $\chi(v) = (i, 2)$ (for some $i\in \{1,2\}$), then $c(v) \in \mathring B_{2\ray[2]}(M_i)$, and there exists a vertex $w$, adjacent\footnote{i. e. such that there is an edge that connects $v$ to $w$.} to $v$, such that $\chi(w)=(i,1)$.
\item If $\chi(v) = (i, j+1)$ for some $2\leq j \leq 2k-1$, then there exists a vertex $w$ adjacent to $v$, such that $\chi(w)=(i, j)$ and $|| c(v) - c(w) || < 2\ray[j+1]$.

\item If $\chi(v) = \infty$, and if there exists a vertex $w$ adjacent to $v$ such that $\chi(w) = (i, 1)$, then $||c(v) - \Omega_i || > r_2$.

\item If $\chi(v) = \infty$, and if there exists a vertex $w$ adjacent to $v$ such that $\chi(w)= (i, j)$ with $j>1$, then $|| c(v) - c(w) || > \ray[j+1]$.
\end{itemize}

Note that if $c\in U(\chi)$, and if $\chi(v)=(i, j)$, $c(v) \in \mathring B_{2\ray[1]+\ldots + 2\ray[j]}(M_i)\subset\mathring  B_{2\delta\ray[j+1]}(M_i)$. In the following, if $e$ is an edge which connects two vertices $v$ and $w$, such that $\chi(v), \chi(w)\not\in \{(1,1), (1,2)\}$, the distance $||c(v)-c(w)||$ is called the \emph{length} of $e$.

\begin{lm}
The family $(U(\chi))_{\chi \text{coloring}}$ defines an open cover of $C_\Gamma(\psi_1\sharp \psi_2)$.

\end{lm}
\begin{proof}
The fact that the $U(\chi)$ are open subsets is immediate. Let us prove that any configuration is in at least one of these sets. Fix a configuration $c$. 

First color all the vertices $v$ such that $c(v) \in \mathring B_{2\ray[1]}(M_i)$ with $\chi(v) = (i, 1)$. 

Next, for $i\in\{1,2\}$, color with $\chi(w)=(i, 2)$ the vertices $w$ adjacent to those of color $(i, 1)$ such that $c(w)\in \mathring B_{2\ray[2]}(M_i)$.

Next, for any $2\leq j \leq 2k-1$, define the vertices of color $(i,j+1)$ inductively: when the vertices of color $(i, j)$ are defined, color with $(i,j+1)$ the vertices $v$ which are not already colored, and such that there exists an edge of length less than $2\ray[j+1]$ between $v$ and a vertex $w$ colored by $(i, j)$.

With this method, no vertex can be simultaneously colored by $(1,j)$ and $(2,j')$. Indeed, the construction above ensures that any vertex colored by $(i, j)$ is in $B_{2\delta\ray[j+1]}(M_i)$. Since $2\delta\ray[j+1] = \delta\frac12(\frac\delta2)^{2k-j}\leq \frac12\delta <\frac14$, we have $B_{2\delta\ray[j+1]}(M_1)\cap B_{2\delta\ray[j'+1]}(M_2)= \emptyset$, which concludes.

Setting $\chi(v) = \infty$ for all the vertices that remain still uncolored after this induction gives a coloring such that $c\in U(\chi)$. \qedhere

\end{proof}
We are going to use the following two lemmas in the proof of Theorem \ref{additivity}.

\begin{lm}\label{86}
If $\chi$ is a coloring such that there exists an edge between a vertex colored by some $(1, j)$ and a vertex colored by some $(2,j')$, then $\omega^F(\Gamma, \sigma,\psi_1\sharp \psi_2)_{|U(\chi)}=0$.
\end{lm}

\begin{lm}\label{87}
If $\chi$ is a coloring such that at least one vertex is colored by $\infty$, then $\omega^F(\Gamma, \sigma,\psi_1\sharp \psi_2)_{|U(\chi)}=0$.
\end{lm}
\emph{Proof of Proposition \ref{prop-add} assuming Lemmas \ref{86} and \ref{87}.}

First note that these two lemmas imply that $I^F(\Gamma,\sigma,\psi_1\sharp \psi_2) = \int_{U}\omega^F(\Gamma,\sigma, \psi_1\sharp \psi_2)$ where $U$ is the union of all the $U(\chi)$ where $\chi$ is a coloring such that no vertex is colored by $\infty$, and no edge connects two vertices colored by some $(1, j)$ and $(2, j')$. By construction, since $\Gamma$ is connected, such a coloring $\chi$ takes only values of the form $(1, j)$ or only values of the form $(2, j)$. Let $U_1$ be the union of the $U(\chi)$ such that $\chi$ takes only values of the form $(1, j)$ and similarly define $U_2$, so that $U = U_1 \sqcup U_2$. This implies that \[ I^F(\Gamma,\sigma,\psi_1\sharp \psi_2) = \int_{U_1}\omega^F(\Gamma,\sigma, \psi_1\sharp \psi_2)+\int_{U_2}\omega^F(\Gamma,\sigma, \psi_1\sharp \psi_2).\]

Note that the form $(\omega^F(\Gamma,\sigma, \psi_1\sharp \psi_2))_{|U_i}$ does not depend on the knot $\psi_{3-i}$, since $U_i$ is composed of configurations which send all vertices in $B_{\frac12}(M_i)$.
This implies that $Z_k(\psi_1\sharp \psi_2) = F_1(\psi_1) + F_2(\psi_2)$ for some functions $F_1$ and $F_2$. For the trivial knot $\psi_{triv}$, Corollary \ref{psio} directly implies that $F_1(\psi_{triv})+F_2(\psi_{triv})=0$. Lemma \ref{neutre} implies that: $$Z_k(\psi_1)=Z_k(\psi_1\sharp \psi_{triv}) =  F_1(\psi_1) +F_2(\psi_{triv})$$
 $$Z_k(\psi_2)=Z_k(\psi_{triv}\sharp \psi_2) = F_1(\psi_{triv}) +F_2(\psi_2)$$
The sum of these two equalities gives $Z_k(\psi_1)+Z_k(\psi_2) = F_1(\psi_1)+F_2(\psi_2) = Z_k(\psi_1\sharp \psi_2)$. This concludes the proof of Proposition \ref{prop-add}, hence of Theorem \ref{additivity}.
\end{proof}

\begin{proof}[Proof of Lemma \ref{86}.]

Lemma \ref{86} directly follows from Proposition \ref{suitable}, since it implies that if $c$ is in the support of $\omega^F(\Gamma,\sigma, \psi_1\sharp \psi_2)$, no edge of $\Gamma$ can connect a vertex of $B_\frac14(M_1)$ and a vertex of $B_\frac14(M_2)$.\qedhere\end{proof}

\begin{proof}[Proof of Lemma \ref{87}.]

Fix a coloring $\chi$ that maps at least one vertex to $\infty$. 
For $j\in\{1,2\}$, let $S_j$ be the set of the vertices of $\Gamma$ colored by a color of $\{j\}\times\{1,\ldots, 2k\}$.

Take $c\in U(\chi)$ and suppose that $c$ is in the support of $\omega^F(\Gamma,\sigma,\psi_1\sharp \psi_2)$. For any external edge $e=(v,w)$ of $ \Gamma_{ S_1, S_2}$, since $p_e(c) \in D(G_{\tau, \eta})$, there exists a sign $\epsilon_\sigma(e)$ such that $|| G_{\tau, \eta}(c(v), c(w)) - \epsilon_{\sigma}(e) X^{n+1}_{\sigma(e)}|| <\delta $, and for any internal edge $e=(v,w)$, there exists a sign $\epsilon_\sigma(e)$ such that  $|| G( c_i(v), c_i(w) ) - \epsilon_\sigma(e) X^{n-1}_{\sigma(e)}|| < \delta$.

\begin{lm}\label{89}Endow the spheres $\s^{n(e)}$ with the usual distance coming from the Euclidean norms $|| \cdot || $ on $\R^{n(e)+1}$.

Let $\chi$ be a coloring that maps at least one vertex to $\infty$, and let $c\in U(\chi)$. Define a configuration $c_0$ of $\confspaceAA$ from $c$ as follows: 
\begin{itemize}
\item If $v$ is a vertex of $\overline S_1$ in $\Gamma_{S_1, S_2}$, $c_0(v) = \Omega_1=(0, 0, \ldots, -\frac12)$.
\item If $v$ is a vertex of $\overline S_2$ in $\Gamma_{S_1, S_2}$, $c_0(v) =\Omega_2= (0, 0, \ldots, \frac12)$.
\item If $v$ is a vertex of $V(\Gamma_{S_1, S_2})\setminus (\overline S_1\sqcup\overline  S_2)= \sommets\setminus(S_1\cup S_2)$, $c_0(v) = c(v)$.
\end{itemize}
Then, $d( G_e( c_0) ,\epsilon_{\sigma}(e)X^{n(e)}_{\sigma(e)}) <9 \delta$ for any edge $e$ of $\Gamma_{S_1,S_2}$.
\end{lm}\begin{proof}

The edges of $\Gamma_{S_1, S_2}$ are of four types: 
\begin{itemize}
\item Those joining two vertices $v$ and $w$ of $V(\Gamma_{S_1, S_2})\setminus (\overline S_1\sqcup\overline  S_2)$.
\item Those joining one vertex $v$ of $V(\Gamma_{S_1, S_2})\setminus (\overline S_1\sqcup\overline  S_2)$ and one vertex $w$ of $\overline S_1$.
\item Those joining one vertex $v$ of $V(\Gamma_{S_1, S_2})\setminus (\overline S_1\sqcup\overline  S_2)$ and one vertex $w$ of $\overline S_2$.
\item Those joining one vertex $v$ of $\overline S_1$ and one vertex $w$ of $\overline S_2$.
\end{itemize}

We have to check that in any of these four cases, the direction of the edge $e$ between $c_0(v)$ and $c_0(w)$ is at distance less than $9\delta$ from $\epsilon_{\sigma}(e)X_{\sigma(e)}^{n(e)}$. We prove this for external edges, the case of internal edges can be proved with the same method. Assume that $e$ goes from $v$ to $w$ (the proof is similar in the other case). In this case, the construction of $G_e$ implies that the direction to look at is $G_{ext}(c_0(v), c_0(w))$. Since $c$ is in the support of $\omega^F(\Gamma,\sigma, \psi_1\sharp \psi_2)$, \[|| G_{\tau, \eta}(c(v), c(w)) -\epsilon_{\sigma}(e) X_{\sigma(e)}^{n+1}||= \left|\left| \frac{\pi(c(w))-\pi(c(v))}{||\pi(c(w))-\pi(c(v))||} - \epsilon_\sigma(e)X_{\sigma(e)}^{n+1}\right|\right|  <\delta.\]

Note the following easy lemma.
\begin{lm}
For any $a$ and $h$ in $\R^{n+2}$ such that $a$ and $a+h$ are non zero vectors:
$$\left|\left| \frac{a}{||a||}-\frac{a+h}{||a+h||}\right|\right|\leq \frac{2||h||}{||a||}.$$
\end{lm}

Now, let us study the previous four cases:

\begin{itemize}

\item In the first case, $c(v)$ and $c(w)$ are in $\voisinageinfinideux[2\eta]$, then the direction of the edge is $G_{ext} (c_0(v), c_0(w) )=G_{ext} (c(v), c(w) )= G_{\tau, \eta}(c(v), c(w)) $. Therefore, it is at distance less than $\delta$ from $\epsilon_{\sigma}(e)X_{\sigma(e)}^{n+1}$.

\item In the second case, $w$ comes from a vertex $w_0$ of $\Gamma$ with $\chi(w_0)=(1, j)$, so $c_0(w) = \Omega_1$ and $c_0(v)=c(v)$.
First suppose $j=1$. This implies that $||\pi(c(w)) - \Omega_1|| <2 \ray[1]$. Since $\chi(v)=\infty$, we have $||\Omega_1-c(v)|| > \ray[2]$.
Then, using the previous lemma and triangle inequalities:
\begin{align*}&\left|\left| \frac{c_0(v)-c_0(w)}{||c_0(v)-c_0(w)||} -\epsilon_{\sigma}(e) X_{\sigma(e)}^{n+1}\right|\right|= \left|\left|  \frac{c(v)-\Omega_1}{||c(v)-\Omega_1||} -\epsilon_{\sigma}(e) X_{\sigma(e)}^{n+1}
\right|\right| \\
&\leq \left|\left|  \frac{c(v)-\pi(c(w))}{||c(v)-\pi(c(w))||} -\epsilon_{\sigma}(e) X_{\sigma(e)}^{n+1}
\right|\right|+  \left|\left|\frac{c(v)-\Omega_1}{||c(v)-\Omega_1||} - \frac{c(v)-\pi(c(w))}{||c(v)-\pi(c(w))||}\right|\right| \\
&<\delta +  2\frac{||\Omega_1-\pi(c(w))||}{||\Omega_1-c(v)||} \\
&\leq \delta + 2\frac{2\ray[1]}{\ray[2]}= 3\delta < 9 \delta
\end{align*}

Suppose now $j>1$. Then $||\Omega_1-c(w)||< 2\delta\ray[j+1]$, and $\pi(c(w)) = c(w)$. Since $\chi(v)=\infty$, we have $|| c(v)-c(w)|| > \ray[j+1]$. 
As in the previous computation, and since $\delta < \frac14$, we get
\begin{align*}&\left|\left| \frac{c_0(v)-c_0(w)}{||c_0(v)-c_0(w)||} -\epsilon_{\sigma}(e) X_{\sigma(e)}^{n+1}\right|\right|\\ 
&\leq \left|\left| \frac{c(v)-c(w)}{||c(v)-c(w)||} - \epsilon_{\sigma}(e) X_{\sigma(e)}^{n+1}\right|\right| + \left|\left| \frac{c_0(v)-c_0(w)}{||c_0(v)-c_0(w)||}-  \frac{c(v)-c(w)}{||c(v)-c(w)||}\right|\right|\\
& < \delta +  2\frac{||\Omega_1-c(w)||}{|| c(w)-c(v)||}\\
&\leq \delta + 2\frac{2\delta\ray[j+1]}{\ray[j+1] } <9 \delta.
\end{align*}

\item The third case, can be studied exactly like the second one.

\item In the last case, note that $c(v) \in B_{2\delta\ray[2k+1]}(M_1)=B_{\frac\delta2}(M_1)$ and $c(w)\in B_{\frac\delta2}(M_2)$.
The direction we look at is $G_{ext}(c_0(v), c_0(w)) = G_{ext}(\Omega_1, \Omega_2) = (0, \ldots, 0,1)$. But, we have $ \left|\left| \frac{\pi(c(v))-\pi(c(w))}{||\pi(c(v))-\pi(c(w))||} - \epsilon_\sigma(e)X_{\sigma(e)}^{n+1}\right|\right|  <\delta$. The previous method yields

\begin{align*}&\left|\left| \frac{c_0(v)-c_0(w)}{||c_0(v)-c_0(w)||} -\epsilon_{\sigma}(e) X_{\sigma(e)}^{n+1}\right|\right|\\
&\leq \left|\left|  \frac{\pi(c(v))-\pi(c(w))}{||\pi(c(v))-\pi(c(w))||}-\epsilon_{\sigma}(e) X_{\sigma(e)}^{n+1}
\right|\right|+  \left|\left| \frac{\Omega_1-\Omega_2}{||\Omega_1-\Omega_2||} - \frac{\pi(c(v))-\pi(c(w))}{||\pi(c(v))-\pi(c(w))||}\right|\right| \\
&<\delta+  \left|\left| \frac{\Omega_1-\Omega_2}{||\Omega_1-\Omega_2||} - \frac{\pi(c(v))-\Omega_2}{||\pi(c(v))-\Omega_2||}\right|\right|+\left|\left| \frac{\pi(c(v))-\Omega_2}{||\pi(c(v))-\Omega_2||}- \frac{\pi(c(v))-\pi(c(w))}{||\pi(c(v))-\pi(c(w))||}\right|\right|\\
&\leq \delta +2\frac{||\pi(c(v))-\Omega_1||}{||\Omega_1-\Omega_2||} + 2\frac{||\pi(c(w))-\Omega_2||}{||\pi(c(v))-\Omega_2||}\\
& \leq \delta+ 2\frac{\frac\delta 2}{1} + 2\frac{\frac\delta2}{1 - \frac\delta2}\leq \left(1+1+\frac87\right)\delta< 9\delta.
\end{align*}

\end{itemize}
This concludes the proof of Lemma \ref{89}\qedhere

\end{proof}
For any $1\leq i \leq 2k$, set \[\hat \epsilon(i) = \hat\epsilon'(i) = \begin{cases} \epsilon_{\sigma}(\sigma^{-1}(i)) & \text{if $e\in \sigma_{S_1,S_2}(E(\Gamma_{S_1,S_2}))$,} \\
1 & \text{otherwise.}\end{cases}\] 
For any $1\leq i\leq 2k$, also set \[ Y_i^{n-1} = \begin{cases} G_{\sigma(e)}(c_0) & \mbox{if $i \in\sigma(E_i(\Gamma_{S_1,S_2}))$,} \\ X_i^{n-1} & \mbox{otherwise,}\end{cases} \mbox{ and } Y_i^{n+1} = \begin{cases} G_{\sigma(e)}(c_0) & \mbox{if $i \in\sigma(E_e(\Gamma_{S_1,S_2}))$,} \\ X_i^{n+1} & \mbox{otherwise.}\end{cases}\]

Lemma \ref{89} implies that $\overline Y= T_{\hat\epsilon, \hat\epsilon'}((Y_i^{n-1}, Y_i^{n+1})_{1\leq i \leq 2k})$ is at distance less than $9\delta$ from $(X_i^{n-1}, X_i^{n+1})_{1\leq i \leq 2k}$. So it belongs to $\mathcal O'_k$ and then to the set $\mathcal O_k$ of Corollary \ref{gras}, which is a contradiction since $\pi_{\Gamma_{S_1,S_2},\sigma}(T_{\hat\epsilon, \hat\epsilon'}(\overline Y))=G_{\Gamma_{S_1,S_2}}(c_0)$. This concludes the proof of Lemma \ref{87}.\qedhere\end{proof}

	\bibliographystyle{alpha}
	\bibliography{Article1}
\end{document}